\documentclass[11pt,leqno]{article}

\usepackage{amssymb,amsmath,amsthm}
\usepackage[safe]{tipa}
\newcommand{\n}{\noindent}   

\newcommand{\vp}{\varepsilon}
\newcommand{\bb}[1]{\mathbb{#1}}
\newcommand{\cl}[1]{\mathcal{#1}}
\newcommand{\sst}{\scriptstyle}
\newcommand{\ovl}{\overline}

\theoremstyle{plain}
\newtheorem{lem}{Lemma}[section]
\newtheorem*{lm}{Lemma}
\newtheorem{pro}[lem]{Proposition}
\newtheorem*{prop}{Proposition}
\newtheorem{thm}[lem]{Theorem}
\newtheorem{cor}[lem]{Corollary}

\theoremstyle{definition}
\newtheorem*{defn}{Definition}
\newtheorem{prbl}[lem]{Problem}

\theoremstyle{remark}
\newtheorem{rem}[lem]{Remark}

\newtheorem*{rk}{Remark}

\numberwithin{equation}{section}

\begin{document}

\pagenumbering{roman}

\title{Complex Interpolation\\
 between\\ Hilbert, Banach and Operator spaces \\}

\author{by\\
Gilles Pisier\footnote{Partially supported by NSF grant 0503688 and   ANR-06-BLAN-0015}\\
Texas A\&M University\\
College Station, TX 77843, U. S. A.\\
and\\
Universit\'e Paris VI\\
Equipe d'Analyse, Case 186, 75252\\
Paris Cedex 05, France}

\date{}
\maketitle

\pagenumbering{arabic}
\setcounter{page}{1}

\n{\large\bf Abstract.} 
Motivated by a question of Vincent Lafforgue,
we study the Banach spaces $X$ satisfying the following property:\ there is a function $\vp\to \Delta_X(\vp)$ tending to zero with $\vp>0$ such that every operator $T\colon \ L_2\to L_2$ with $\|T\|\le \vp$ that is simultaneously contractive (i.e.\ of norm $\le 1$) on $L_1$ and on $L_\infty$ must be of norm $\le \Delta_X(\vp)$ on $L_2(X)$.
  We show that  $\Delta_X(\vp)\in O(\vp^\alpha)$ for some $\alpha>0$ iff  $X$ is isomorphic to a quotient of a subspace of an ultraproduct of $\theta$-Hilbertian spaces for some $ \theta>0$ (see Corollary \ref{comcor4.3}), where   $\theta$-Hilbertian is meant in a slightly more general sense than in our previous paper \cite{P1}. 
Let   $B_{{r}}(L_2(\mu))$ be the space of all regular operators on $L_2(\mu)$. We are able to describe the complex interpolation space
\[
(B_{{r}}(L_2(\mu)), B(L_2(\mu)))^\theta.
\]
We show that $T\colon \ L_2(\mu)\to L_2(\mu)$ belongs to this space iff $T\otimes id_X$ is bounded on $L_2(X)$ for any $\theta$-Hilbertian space $X$. 

More generally, we are able to describe the spaces
$$  (B(\ell_{p_0}), B(\ell_{p_1}))^\theta \ {\rm or}\  (B(L_{p_0}), B(L_{p_1}))^\theta $$
for any pair $1\le p_0,p_1\le \infty$ and $0<\theta<1$.
In the same vein, given a locally compact Abelian group $G$, let $M(G)$ (resp.\ $PM(G)$) be the space of complex measures (resp.\ pseudo-measures) on $G$ equipped with the usual norm  $\|\mu\|_{M(G)} = |\mu|(G)$ (resp.
\[
\|\mu\|_{PM(G)} = \sup\{|\hat\mu(\gamma)| \ \big| \ \gamma\in\widehat G\}).
\]
We describe similarly the interpolation space $(M(G), PM(G))^\theta$. Various extensions and variants of this result will be given, e.g.\ to Schur multipliers on $B(\ell_2)$ and to operator  spaces.

\bigskip\bigskip

MSC Class: 46B70, 47B10, 46M05, 47A80

\vfill\eject

\tableofcontents
\vfill\eject
\n{\large\bf Introduction.}
This paper is a contribution to the study of the complex interpolation method. The latter originates in 1927 with the famous Marcel Riesz theorem which says that, if $1\le p_0 < p_1\le\infty$, and if $(a_{ij})$ is a matrix of norm $\le 1$ simultaneously on $\ell^n_{p_0}$ and $\ell^n_{p_1}$, then it must be also of norm $\le 1$ on $\ell^n_p$ for any $p_0 < p < p_1$, and similarly for operators on $L_p$-spaces. Later on in 1938, Thorin found the most general form using a complex variable method; see \cite{BL} for more on this history.

Then around 1960, J.L.~Lions and independently A.~Calder\'on \cite{Cal} invented the complex interpolation method, which may be viewed as a far reaching ``abstract'' version of the Riesz--Thorin theorem, see \cite{BL,KaM}. There the pair $(L_{p_0}, L_{p_1})$ can be replaced by a pair $(B_0,B_1)$ of Banach spaces (assumed compatible in a suitable way). One then defines for any $0<\theta<1$ the complex interpolation space $B_\theta = (B_0,B_1)_\theta$ which appears as a continuous deformation of $B_0$ into $B_1$ when $\theta$ varies from 0 to 1. In many ways, the unit ball ${\cl B}_\theta$  of the space $B_\theta$ looks like the ``geometric mean'' of the respective unit balls ${\cl B}_0$ and ${\cl B}_1$ of $B_0$ and $B_1$, i.e.\ it seems to be the multiplicative analogue of the Minkowski sum $(1-\theta){\cl B}_0 + \theta {\cl B}_1$. The main result of this paper relates directly to the sources of interpolation theory:\ we give a description of the space $B_\theta = (B_0,B_1)_\theta$ when $B_0 = B(\ell^n_{p_0})$ and $B_1 = B(\ell^n_{p_1})$, or more generally for the pair $B_0 = B(L_{p_0}(\mu))$, $B_1 = B(L_{p_1}(\mu))$.

Although our description of the norm of $B_\theta$ for these pairs is, admittedly, rather ``abstract'' it shows that the problem of calculating $B_\theta$ is equivalent to the determination of a certain class of Banach spaces
\[
(SQ(p_0), SQ(p_1))_\theta
\]
roughly interpolated between the classes $SQ(p_0)$ and $SQ(p_1)$ where $SQ(p)$ denotes the class of subspaces of quotients (subquotients in short) of $L_p$-spaces. 

When $p_0=1$ or $=\infty$, this class $SQ(p_0)$ is the class of all Banach spaces while when $p_1=2$, $SQ(p_1)$ is the class of all Hilbert spaces. In that case, the class $(SQ(p_0), SQ(p_1))_\theta$ is the class of all the Banach spaces $B$ which can be written (isometrically) as $B = (B_0,B_1)_\theta$ for some compatible pair $(B_0,B_1)$ with
\[
 B_j \in SQ(p_j)\qquad (j=0,1).
\]
We already considered this notion in a previous paper \cite{P1}. There we called $\theta$-Hilbertian the resulting spaces. However, in the present context we need to slightly extend the notion of $\theta$-Hilbertian, so we decided to rename ``strictly $\theta$-Hilbertian'' the spaces called $\theta$-Hilbertian in \cite{P1}. In our new notion of ``$\theta$-Hilbertian'' we found it necessary to use the complex interpolation method for ``families'' $\{B_z\mid z\in \partial D\}$ 
defined on the boundary of a complex domain $D$ and not only pairs of Banach spaces

This generalization was developed around 1980 in a series of papers mainly by Coifman, Cwikel, Rochberg, Sagher, Semmes and Weiss (cf.\ \cite{CCRS1,CCRS2,CCRS3,Sem,CS}). There $\partial D$ can be the unit circle and, restricting to the $n$-dimensional case for simplicity, we may take $B_z = ({\bb C}^n, \|~~\|_z)$ where $\{\|~~\|_z \mid z\in \partial D\}$ is a measurable family of norms on ${\bb C}^n$ (with a suitable nondegeneracy). The interpolated spaces now consist in a family $\{B(\xi)\mid \xi \in D\}$ which extends the boundary data $\{B_z\mid z\in \partial D\}$ in a specific way reminiscent of the harmonic extension. When $B_z = \ell^n_{p(z)}$ with $1\le p(z)\le \infty$ $(z\in\partial D)$ one finds $B(\xi) = \ell^n_{p(\xi)}$ where $p(\xi)$ is determined by
\[
p(\xi)^{-1} = \int\limits_{\partial D} p(z)^{-1} \mu_\xi(dz)
\]
where $\mu_\xi$ is the harmonic (probability) measure of $\xi\in D$ relative to $\partial D$. 

Consider then an $n\times n$ matrix $a = [a_{ij}]$, let $\beta_z= B(\ell^n_{p(z)})$ for $z\in\partial D$ and let $\beta(\xi)$ $(\xi\in D)$ be the resulting interpolation space.

One of our main results is the equality
\begin{equation}\label{intro1} 
\|a\|_{\beta(\xi)} = \sup\{\|a_X\colon \ \ell^n_{p(\xi)}(X)\to \ell^n_{p(\xi)}(X)\}
\end{equation}
where $a_X$ is the matrix $[a_{ij}]$ viewed as acting on $X^n$ in the natural way and where the supremum runs over all the $n$-dimensional Banach spaces $X$ in the class ${\cl C}(\xi)$. The class ${\cl C}(\xi)$ consists of all the spaces $X$ which can be written as $X = X(\xi)$ for some (compatible) family $\{X(z)\mid z\in \partial D\}$ such that $X(z)\in SQ(p(z))$ for all $z$ in $\partial D$ and of all ultraproducts of such spaces.

By a sort of ``duality,'' this also provides us with a characterization of this class ${\cl C}(\xi)$, or more precisely
of the class of subspaces of quotients of spaces in  ${\cl C}(\xi)$: 
a Banach space $ X$ belongs to the latter class (resp.\ is $C$-isomorphic to a space in that class) iff for any $n$
\begin{equation}\label{intro2} 
 \sup_{\|a\|_{\beta(\xi)}\le 1} \|a_X\colon \ \ell^n_{p(\xi)}(X)\to \ell^n_{p(\xi)}(X)\|\le 1 \ \ ({\rm resp.}\ \le C).
\end{equation}

Consider for example the case when $p(z)$ takes only two values $p(z)=1$ and $p(z)=2$ with measure respectively $1-\theta$ and $\theta$. Then $\beta(0) = (B(\ell^n_1), B(\ell^n_2))_\theta$ and ${\cl C}(0)$ is the class of all the spaces which can be written as $X(0)$ for some boundary data $\partial D\ni z\longmapsto X(z)$ such that $X(z)$ is Hilbertian on a subset of normalized Haar measure $\ge\theta$ (and is Banach on the complement). We call these spaces $\theta$-Euclidean and we call $\theta$-Hilbertian all ultraproducts of $\theta$-Euclidean spaces of arbitrary dimension.

Actually, our result can be formulated in a more general framework:\ we give ourselves classes of Banach spaces $\{{\cl C}(z)\mid z\in \partial D\}$ with minimal assumptions and we set by definition
\[
\|a\|_{\beta(z)} = \sup_{X\in {\cl C}(z)} \|a_X\colon \ \ell^n_{p(z)}(X) \to \ell^n_{p(z)}(X)\|.
\]
Then \eqref{intro1} and \eqref{intro2} remain true with ${\cl C}(z)$ in the place of $SQ(p(z))$. In particular, we may now restrict to the case when $p(z) = 2$ for all $z$ in $\partial D$. Consider for instance the case when ${\cl C}(z) = \ell^n_2$ for $z$ in a subset (say an arc) of $\partial D$ of normalized Haar measure $\theta$ and let ${\cl C}(z)$ be the class of all $n$-dimensional Banach spaces on the complement. Then \eqref{intro1} yields a description of the space $(B_0,B_1)_\theta$ when $B_0,B_1$ is the following pair of normed spaces consisting of $n\times n$ matrices:
\begin{align*}
 \|a\|_{B_0} &= \|[|a_{ij}|]\|_{B(\ell^n_2)}\\
\|a\|_{B_1} &= \|[a_{ij}]\|_{B(\ell^n_2)}.
\end{align*}
More generally, if $B_1 = B(L_2(\mu))$ and if $B_0$ is the Banach space $B_r(L_2(\mu))$ of all regular operators $T$ on $L_2(\mu)$ (i.e.\ those $T$ with a kernel $(T(s,t))$ such that $|T(s,t)|$ is bounded on $L_2(\mu)$), then we are able to describe the space $(B_r(L_2(\mu)), B(L_2(\mu)))^\theta$. By \cite{Be} this also yields $(B_0,B_1)_\theta$ as the closure of $B_0\cap B_1$ in $(B_0,B_1)^\theta$. 

The origin of this paper is a question raised by Vincent Lafforgue:\ what are the Banach spaces $X$ satisfying the following property:\ there is a function $\vp\to \Delta_X(\vp)$ tending to zero with $\vp>0$ such that every operator $T\colon \ L_2\to L_2$ with $\|T\|\le \vp$ that is simultaneously contractive (i.e.\ of norm $\le 1$) on $L_1$ and on $L_\infty$ must be of norm $\le \Delta_X(\vp)$ on $L_2(X)$ ?

We show that  $\Delta_X(\vp)\in O(\vp^\alpha)$ for some $\alpha>0$ iff  $X$ is isomorphic to  a subspace of  a quotient of a $\theta$-Hilbertian space for some $ \theta>0$ (see Corollary \ref{comcor4.3}). We also give
a sort of structural, but   less
satisfactory, characterization of the spaces $X$ such that
$\Delta_X(\vp)\to 0$ when $\vp\to 0$ (see Theorem \ref{comthm7.3}).

V. Lafforgue's question
 is motivated by  
 the (still open) problem whether
expanding graphs can be coarsely embedded into uniformly convex Banach spaces; he observed that such an embedding is impossible
into $X$  if $\Delta_X(\vp)\to 0$ when $\vp\to 0$.
See \S \ref{expan}  for more on this.

The preceding results all have analogues in the recently developed theory of operator spaces (\cite{ER,P5}). Indeed, 
the author previously introduced and studied 
mainly in \cite{P6,P7} all the necessary ingredients, notably complex interpolation and operator space valued non-commutative $L_p$-spaces. With these tools, it is an easy task to check the generalized statements, so that we merely review them, giving only indications of proofs. In addition, in the last section, we include an  example
   hopefully   demonstrating  
 that 
interpolation of {\it families} (i.e. involving more than a pair)   of operator spaces, appears very naturally in   harmonic analysis on the free group.

Let us now describe the contents, section by section. In \S \ref{comsec1}, we review some background on regular operators. An operator $T$ on $L_p(\mu)$ is called regular if there is a positive operator $S$, still bounded on $L_p(\mu)$, such that
\begin{equation}
|Tf| \le S(|f|).\tag*{$\forall f\in L_p(\mu)$}
\end{equation}
The regular norm of $T$ is equal to the infimum of $\|S\|$. These operators can be characterized in many ways. They play an extremal role in Banach space valued analysis because they are precisely the operators on $L_p(\mu)$ that extend (with the same norm) to $L_p(\mu;X)$ for any Banach space $X$.

In \S \ref{comsec2} we use the fact that  regular operators on $L_p(\mu)$ $(1<p<\infty)$ with regular norm $\le 1$ are closely related (up to a change of density) to what we call fully contractive operators, i.e.\ operators that are of norm $\le 1$ simultaneously on $L_1$ and $L_\infty$.

This allows us to rewrite the definition of $\Delta_X(\vp)$ in terms of regular operators.

In \S \ref{comsec3}, we describe a certain duality between, on one hand, classes of Banach spaces, and on the other one, classes of operators on $L_p$. Although these ideas already  appeared (cf. \cite{Kw1,Kw2,He,Ju}), the viewpoint we emphasize was left sort of implicit. We hope to stimulate further research on the list of related questions that we present in this section.

In \S \ref{comsec4}, we present background on the complex interpolation method for families (or ``fields'') of Banach spaces. This was developed mainly by   Coifman, Cwikel, Rochberg, Sagher, Semmes, and Weiss cf. \cite{CCRS1,CCRS2,CCRS3,CS,Sem}. 

In \S \ref{comsec4bis}, we generalize the notion of $\theta$-Hilbertian Banach space from our previous paper \cite{P1}. We first call $\theta$-Euclidean any $n$-dimensional space which can be obtained as the interpolation space at the center of the unit disc $D$ associated to a family of $n$-dimensional spaces $\{X(z)\mid z\in \partial D\}$ such that $X(z)$ is Hilbertian for a set of $z$ with (Lebesgue) measure $\ge\theta$.
Then we call  $\theta$-Hilbertian
all ultraproducts of $\theta$-Euclidean spaces. In our previous definition (now called strictly $\theta$-Hilbertian), we only considered a two-valued family $\{X(z)\mid z\in \partial D\}$. We are then able to describe the interpolation space
\[
 (B_r, B)^\theta
\]
where $B_r$ and $B$ denote respectively the regular and the bounded operators on $\ell_2$. We then characterize the Banach spaces $X$ such that $\Delta_X(\vp)\in O(\vp^\alpha)$ for some $\alpha>0$ as the subspaces of quotients of $\theta$-Hilbertian spaces.

In \S \ref{comsec6}, we briefly compare our notion of $\theta$-Hilbertian with the corresponding ``arcwise'' one, where the set of $z$'s for which $X(z)$ is Hilbertian is required to be an arc.

In \S \ref{comsec55}, we turn to Fourier and Schur multipliers:\ we can describe analogously the complex interpolation spaces $(B_0,B_1)^\theta$ when $B_0$ (resp.\ $B_1$) is the space of measures (resp.\ pseudo-measures) on a locally compact Abelian group $G$ (and similarly on an amenable group). We also treat the case when $B_0$ (resp.\ $B_1$) is the class of bounded Schur multipliers on $B(\ell_2)$ (resp.\ on the Hilbert--Schmidt class $S_2$ on $\ell_2$). In the latter case, $B_1$ can be identified with 
the space of bounded functions on ${\bb N}\times {\bb N}$.

In \S \ref{comsec7}, we give a characterization of ``uniformly curved'' spaces, i.e.\ the Banach spaces $X$ such that $\Delta_X(\vp)\to 0$ when $\vp\to 0$. This appears as a real interpolation result, but  is less satisfactory than in the case $\Delta_X(\vp)\in O(\vp^\alpha)$ for some $\alpha>0$ and many natural questions remain open.

In \S \ref{comsec8}, we generalize an extension property of regular operators from \cite{P2} which may be of independent interest. See \cite{MN} for related questions. This result will probably be relevant if one tries, in analogy with \cite{Kw2}, to characterize the subspaces or the complemented subspaces of $\theta$-Hilbertian spaces.
In particular we could not distinguish any of the two latter classes from that
of subquotients of  $\theta$-Hilbertian spaces. The paper \cite{MoG} 
contains useful related information.   We should mention that an
  extension property similar to ours appears in \cite[1.3.2]{Ju}.

In \S \ref{comsec12}, we describe the complex interpolation spaces $(B_0,B_1)^\theta$ when $B_0 = B(L_{p_0}(\mu))$ and $B_1 = B(L_{p_1}(\mu))$ with $1\le p_0,p_1\le \infty$. Actually, the right framework seems to be here again the interpolation of families $\{B_z\mid z\in \partial D\}$ where $B_z = B(\ell^n_{p(z)})$. We treat this case and an even more general one related to the ``duality'' discussed in \S 3, see Theorem \ref{comthm12.1} for the most general statement.

In \S \ref{comsec9} and \S \ref{sec13}, we turn to the analogues of the preceding results in the operator space framework. There operators on $L_p(\mu)$ are replaced by mappings acting on ``non-commutative'' $L_p$-spaces associated to a trace. The main results are entirely parallel to the ones obtained in \S \ref{comsec4bis} and \S \ref{comsec12} in the commutative case.

Lastly, in \S \ref{comsec10}, we describe a family of operator spaces closely connected to various works on the ``non-commutative Khintchine inequalities'' for homogeneous polynomials of degree $d$ (see e.g.\ \cite{PaP}). Here we specifically need to consider a family $\{X(z)\mid z\in \partial D\}$ taking $(d+1)$-values but we are able to compute precisely the interpolation at the center of $D$ (or at any point inside $D$).

\section{Preliminaries. Regular operators}\label{comsec1}

Let $1\le p<\infty$ throughout this section. For operators on  $L_p$
it is well known that the notions of ``regular' and ``order bounded" coincide,
so we will simply use the term regular.  We refer to \cite{Mey,Sch} for general facts on this.
The results of this section are all essentially well known, we only recall a few short proofs
for the reader's convenience and to place them in the context that is relevant for us.

\subsection{}\label{comsec1.1}
We say that an operator $T\colon \ L_p(\mu) \to L_p(\nu)$ is regular if there is a constant $C$ such that for all $n$ and all $x_1,\ldots, x_n$ in $L_p(\mu)$ we have
\[
\|\sup|Tx_k|~\|_p \le C\|\sup|x_k|~\|_p.
\]
We denote by $\|T\|_{\text{reg}}$ the smallest $C$ for which this holds and by $B_r(L_p(\mu), L_p(\nu))$ (or simply $B_r(L_p(\mu))$ if $\mu=\nu$) the Banach space of all such operators equipped with the norm $\|~~\|_{\text{reg}}$.

Clearly this definition makes sense more generally for operators $T\colon \ \Lambda_1\to \Lambda_2$ between two Banach lattices $\Lambda_1,\Lambda_2$.

\subsection{}\label{comsec1.2}
It is known that $T\colon \ L_p(\mu)\to L_p(\nu)$ is regular iff $T\otimes id_X\colon \ L_p(\mu; X)\to L_p(\nu; X)$ is bounded for \emph{any} Banach space $X$ 
and \begin{equation}\|T\|_{\text{reg}}=\sup_{X} \|  T\otimes id_X\colon \ L_p(\mu; X)\to L_p(\nu; X)\| .\end{equation}
This assertion 
follows from the fact that any finite dimensional subspace $Y\subset X$ can be embedded almost isometrically into
$\ell_\infty^n$ for some large enough $n$. See \ref{comsec1.8} below.
The preceding definition corresponds to $\ell^n_\infty$ for all $n$, or equivalently to $X=c_0$. 

\n Actually, $T\colon \ L_p(\mu)\to L_p(\nu)$ is regular iff there is a constant $C$ such that for all $n$ and all $x_1,\ldots, x_n$ in $L_p(\mu)$ we have
\[
\|\sum|Tx_k|~\|_p \le C\|\sum|x_k|~\|_p,
\] and the smallest such $C$ is equal to $\|T\|_{\rm reg}$.
This follows from the fact that any finite dimensional space $X$ is almost isometric   to
a quotient of $\ell_1^n$ for some large enough $n$.

\subsection{}\label{comsec1.3} A (bounded) positive (meaning positivity preserving)  operator $T$ is regular and $\|T\|_{\text{reg}} = \|T\|$.
More precisely, it is a classical fact that $T$ is regular iff there is a bounded positive operator $S\colon \ L_p(\mu)\to L_p(\nu)$ (here $1\le p<\infty$) such that $|T(x)|\le S(|x|)$ for any $x$ in $L_p(\mu)$. Moreover, there is a smallest $S$ with this property, denoted by $|T|$, and we have:
\[
\|T\|_{\text{reg}} = \|~|T|~\|.
\]
In case $L_p(\mu) = L_p(\nu) = \ell_p$, the operator $T$ can be described by a matrix $T = [t_{ij}]$. Then
\[
|T| = [|t_{ij}|].
\]
Similarly, if $T$ is given by a nice kernel $(K(s,t))$ then $|T|$ corresponds to the kernel $(|K(s,t)|)$.

\subsection{}\label{comsec1.3bis}
In this context, although we will not use this, we should probably mention the following identities (see \cite{P1}) that are closely related to Schur's criterion for boundedness of a matrix on $\ell_2$ and its (less well known) converse:
\begin{align*}
(B(\ell^n_1), B(\ell^n_\infty))_\theta &= B_r(\ell^n_p, \ell^n_p)\\
(B(\ell_1), B(c_0))^\theta &= B_r(\ell_p,\ell_p).
\end{align*}
These are isometric isomorphisms with $p$ defined as usual by $p^{-1} = (1-\theta)$. 

More explicitly, a matrix $b = (b_{ij})$ is in the unit ball of $B_r(\ell^n_p)$ iff there are matrices $b^0$ and $b^1$ satisfying
\[
|b_{ij}| \le |b^0_{ij}|^{1-\theta} |b^1_{ij}|^\theta\]
and such that
\[
\sup_i \sum\nolimits_j |b^0_{ij}|\le 1\quad \text{and}\quad \sup_j \sum\nolimits_i |b^1_{ij}|\le 1.
\]
The ``if" direction boils down to Schur's well known classical criterion when $p=2$ (see  also   \cite{Ko}).

\subsection{}\label{comsec1.7}
We will now describe the unit ball of the dual of $B_r(\ell^n_2)$.
\begin{lem}\label{comlem1.6}
Consider an $n\times n$ matrix $\varphi = (\varphi_{ij})$. Then
\begin{equation}\label{comeq1.1}
\|\varphi\|_{B_r(\ell^n_2)^*} = \inf\left\{\left(\sum\nolimits^n_1 |x_i|^2 \sum\nolimits^n_1 |y_j|^2\right)^{1/2}\right\}
\end{equation}
where the infimum runs over all $x,y$ in $\ell^n_2$ such that
\begin{equation}
|\varphi_{ij}| \le |x_i|~|y_j|.\tag*{$\forall i,j$}
\end{equation}
\end{lem}

\begin{proof}
Let $C$ be the set of all $\varphi$ for which there are $x,y$ in the unit ball of $\ell^n_2$ such that $|\varphi_{ij}|\le |x_i|~|y_j|$. Clearly we have for all $a$ in $B(\ell^n_2)$ 
\[
\|a\|_{B_r(\ell^n_2)} = \|[|a_{ij}|]\| = \sup_{\varphi\in C} \left|\sum \varphi_{ij}a_{ij}\right|.
\]
Therefore, to prove the Lemma it suffices to check that $C$ is convex (since the right-hand side of \eqref{comeq1.1} is the gauge of $C$). This is easy to check:\ consider $\varphi,\varphi'$ in $C$ and $0<\theta<1$ then assuming
\[
|\varphi_{ij}| \le |x_i|~|y_j|\quad \text{and}\quad |\varphi'_{ij}|\le |x'_i|~|y'_j|
\]
with $x, y, x', y'$ all in the Euclidean unit ball, we have by Cauchy-Schwarz
\[
|(1-\theta)\varphi_{ij} + \theta\varphi'_{ij}| \le ((1-\theta) |x_i|^2 + \theta|x'_i|^2)^{1/2} ((1-\theta) |y_j|^2 + \theta|y'_j|^2)^{1/2},
\]
which shows that $(1-\theta)\varphi + \theta\varphi'$ is in $C$.
\end{proof}

Let $C$ be as above. Then $\varphi\in C$ iff there are $h_i,k_j$ in ${\bb C}^n$ such that $\varphi_{ij} = \langle h_i,k_j\rangle$ and
\[
\sum\|h_i\|^2_{\ell^n_1}\le 1,\qquad \sum \|k_j\|^2_{\ell^n_\infty} \le 1.
\]
Indeed, if this holds we can write
\[
|\varphi_{ij}| \le \sum\nolimits_m |h_i(m)|~|k_j(m)| \le \|h_i\|_{\ell^n_1} \|k_j\|_{\ell^n_\infty}
\]
from which $\varphi\in C$ follows. Conversely, if $\varphi\in C$, we may assume $\varphi_{ij} = x_iy_j\gamma_{ij}$ with $|\gamma_{ij}|\le 1$, $\|x\|_2 \le 1$, $\|y\|_2 \le 1$. Let $(e_m)$ denote the canonical basis of ${\bb C}^n$. Then, letting
\[
h_i = x_ie_i\quad \text{and}\quad k_j = y_j\quad \sum\nolimits_m \gamma_{mj}e_m
\]
we obtain the desired representation.

\subsection{}\label{comsec1.6}
The predual of $B(L_2(\mu), L_2(\mu'))$ is classically identified with the projective tensor product $L_2(\mu)\widehat\otimes L_2(\mu')$, i.e.\ the completion of the algebraic tensor product $L_2(\mu)\otimes L_2(\mu')$ with respect to the norm
\[
\|T\|_\wedge = \inf \sum\|x_m\| \|y_m\|
\]
where the infimum runs over all representations of $T$ as a sum $T = \Sigma x_m\otimes y_m$ of rank one tensors. Let $T(s,t) = \Sigma x_m(s) y_m(t)$ be the corresponding kernel in $L_2(\mu\times \mu')$. An easy verification shows that
\[
\|T\|_\wedge = \inf\{\|h\|_{L_2(\ell_2)} \|k\|_{L_2(\ell_2)}\}
\]
where the infimum runs over all $h,k$ in $L_2(\ell_2)$ such that $T(s,t) = \langle h(s), k(t)\rangle$.

 We now describe a predual of $B_r(L_2(\mu), L_2(\mu'))$. For any $T$ in $L_2(\mu) \otimes L_2(\mu')$, let 
\begin{equation}\label{comeq1.2}
N_r(T) = \inf\{\|x\|_2\|y\|_2\}
\end{equation}
where the infimum runs over all $x$ in $L_2(\mu)$ and all $y$ in $L_2(\mu')$ such that
\[
|T(s,t)|\le x(s)y(t)
\]
for almost all $s,t$. Equivalently, we have
\begin{equation}\label{comeq1.3}
N_r(T) = \inf\left\{\left\|\sum\nolimits^n_1 |h_i|\right\|_2 \|\sup|k_j|\|_2\right\} = \inf\{\|h\|_{L_2(\ell^n_1)} \|k\|_{L_2(\ell^n_\infty)}\}
\end{equation}
where the infimum runs over all $n$ and all $h = (h_1,\ldots, h_n)$ $k = (k_1,\ldots, k_n)$ in $(L_2)^n$ such that
\begin{equation}\label{comeq1.4}
T(s,t) = \sum\nolimits^n_1 h_i(s) k_i(t).
\end{equation}
Indeed, it is easy to show  that the right-hand sides of both \eqref{comeq1.2} and \eqref{comeq1.3} are convex functions of $T$ and moreover 
(recalling \ref{comsec1.1}, \ref{comsec1.2} and \ref{comsec1.3}) that for any $b$ in $B_r(L_2(\mu), L_2(\mu'))$
\[
\|b\|_{\text{reg}} = \sup\{|\langle b,T\rangle|\}
\]
where the supremum runs over $T$ such that the right-hand side of either \eqref{comeq1.2} or \eqref{comeq1.3} is $\le 1$. This implies that \eqref{comeq1.2} and \eqref{comeq1.3} are equal. Let $L_2(\mu)\widehat\otimes_r L_2(\mu')$ be the completion of $L_2(\mu)\otimes L_2(\mu')$ with respect to this norm. Then there is an isometric isomorphism
\[
(L_2(\mu) \widehat\otimes_r L_2(\mu'))^* \simeq B_r(L_2(\mu), L_2(\mu'))
\]
associated to the duality pairing
\begin{equation}
\langle b,x\otimes y\rangle = \langle b(x), y\rangle\tag*{$\forall b\in B_r(L_2(\mu), L_2(\mu'))$}
\end{equation}

\subsection{}\label{comsec1.8}
More generally,  a predual of $B_r(L_p(\mu), L_p(\mu'))$ can be obtained as the completion of $L_{p'}(\mu) \otimes L_{p}(\mu')$ for the norm
\begin{equation}\label{comeq1.5}
\forall T\in L_{p'}(\mu')\otimes L_{p}(\mu)\qquad\qquad N_r(T) = \inf\left\{\left\| \sum\nolimits^n_1 |f_i|\right\|_{p'} \|\sup_{i\le n}|g_i|\|_{p}\right\}~~~~~~~~~~~~~~~~~~ 
\end{equation}
where the supremum runs  over all decompositions of the kernel of $T$ as $T(s,t) = \sum^n_1 f_i(s)g_i(t)$. To verify that \eqref{comeq1.5} is indeed a norm, we will first show that \eqref{comeq1.5} coincides with
\begin{equation}\label{comeq1.6}
M_r(T) = \inf\{\|\xi\|_{L_{p'}(Y^*)} \|\eta\|_{L_{p}(Y)}\}
\end{equation}
where the infimum runs over all finite dimensional normed spaces $Y$ and all pairs $(\xi,\eta)\in L_{p'}(\mu';Y^*)\times L_{p}(\mu,Y)$ such that $T(s,t) = \langle\xi(s), \eta(t)\rangle$.

Clearly $M_r(T) \le N_r(T)$. Conversely, given $Y$ as in \eqref{comeq1.6}, for any $\vp>0$ there is $n$ and an embedding $j\colon \ Y\to \ell^n_\infty$ such that $\|y\| \le \|j(y)\| < (1+\vp)\|y\|$ for all $y$ in $Y$. Let $(\xi,\eta)$ be as in \eqref{comeq1.6}. Let $\hat\eta = j\eta\in L_{p}(\ell^n_\infty)$. Note $\|\hat\eta\|\le (1+\vp)\|\eta\|$. Let $q=j^*\colon \ \ell^n_1\to Y^*$ be the corresponding surjection. By an elementary lifting, there is $\hat\xi$ in $L_{p'}(\ell^n_1)$ with $\|\hat\xi\|\le (1+\vp)\|\xi\|$ such that $\xi =q\hat \xi$.

We have then $T(s,t) = \langle\xi(s), \eta(t)\rangle = \langle q\hat\xi(s), \eta(t)\rangle = \langle\hat\xi(s), \hat\eta(t)\rangle$ and $\|\hat\xi\|_{L_{p'}(\ell^n_1)} \|\hat\eta\|_{L_{p}(\ell^n_\infty)}\le$ $(1+\vp)^2 \|\xi\|_{L_{p'}(Y^*)} \|\eta\|_{L_{p}(Y)}$. Thus we conclude that $M_r(T) \le N_r(T)$ and hence $M_r(T) = N_r(T)$.

To check that $N_r$ is a norm, we will prove it for $M_r$. This is very easy. Consider $T_1,T_2$ with $M_r(T_j) < 1$, $(j=1,2)$ and let $0\le \theta\le 1$. We can write 
\[
T_1(s,t) = \langle\xi_1(s), \eta_1(t)\rangle
\]
$$T_2(s,t) = \langle\xi_2(s), \eta_2(t)\rangle$$ with $(\xi_j,\eta_j) \in L_p(Y_j) \times L_{p'}(Y^*_j)$. Then $$(1-\theta) T_1(s,t) + \theta T_2(s,t) = \langle\xi(s), \eta(t)\rangle$$ where $(\xi,\eta) \in L_p(Y) \times L_{p'}(Y^*)$ with
\begin{alignat*}{2}
Y &= Y_1\oplus_p Y_2 &\qquad Y^* &= Y^*_1 \oplus_{p'} Y^*_2\\
\xi &= ((1-\theta)^{1/p}\xi_1 \oplus \theta^{1/p}\xi_2), &\qquad \eta &= ((1-\theta)^{1/{p'}} \eta_1 \oplus \theta^{1/{p'}}\eta_2).
\end{alignat*}
We conclude that
\[
M_p(T_1+T_2) \le \|\xi\|_{L_p(Y)} \|\eta||_{L_{p'}(Y^*)} \le 1.
\]
Now that we know that \eqref{comeq1.5} is indeed a norm, it is clear (either by \ref{comsec1.1} or \ref{comsec1.2}) that the completion $L_p(\mu) \widehat\otimes_r L_{p'}(\mu')$ of $(L_p(\mu) \otimes L_{p'}(\mu), N_r)$ is isometrically a predual of $B_r(L_p(\mu), L_{p'}(\mu')$).

\subsection{}\label{comsec1.9} 
 We refer e.g. to \cite{P00} for more information and references
  on all this subsection (see also \cite{ER,P5} for the operator space analogue). The original ideas  
  can  be traced back to \cite{Gr1}.

\n An operator $v\colon\ E\to F$ between Banach spaces is called nuclear if it can be written as an absolutely convergent series of rank one operators, i.e.
there are $x^*_n\in E^*$, $y_n\in F$ with
$\sum \|  x^*_n  \|  \|  y_n\|<\infty$ such that
$$v(x)=\sum \langle x^*_n ,x\rangle y_n\quad \forall x\in E.$$
The space of such maps is denoted by $N(E,F)$.
 The nuclear norm $N(v)$ is defined
as
$$ N(v)=\inf  \sum \|  x^*_n  \|  \|  y_n\|,$$
where the infimum runs over all possible such representations of $v$.
Equipped with this norm,
  $N(E,F)$ is a Banach space. 
  
If   $E$ and $F$ are finite dimensional, it is well known that  we have isometric identities
$$B(E,F)^*=N(F,E)\quad\text{and}\quad
N(E,F)^*=B(F,E)$$
with respect to the duality defined
for $u:\ E\to F$ and $v:\ F\to E$ by
$$\langle u,v\rangle=\text{tr}(uv).$$

We will denote by $\Gamma_H (E,F)$ the set of operators $u\colon \ E\to F$ that factorize through a Hilbert space, i.e.\ there are bounded operators $u_1\colon \ H\to F$, $u_2\colon \ E\to H$ such that $u=u_1u_2$. We equip this space with the norm $\gamma_H(.)$ defined by
\[
\gamma_H(u) = \inf\{\|u_1\| \ \|u_2\|\}
\]
where the infimum runs over all such factorizations.
    
We will denote by $\gamma^*_H(.)$ the  norm that is dual to $\gamma_H(.)$ in the above duality, i.e. for all $v\colon\ F\to E$ we set
$$\gamma^*_H(v)=\sup \{ | \text{tr}(uv)  |\mid u\in \Gamma_H (E,F), \gamma_H(u)\le 1.\}$$

\begin{pro}\label{compro7.2}
Consider $v\colon \ \ell^n_\infty\to \ell^n_1$.
\begin{itemize}
\item[\rm (i)] $\gamma^*_H(v)\le 1$ iff there are $\lambda,\mu$ in the unit ball of $\ell^n_2$ and $(a_{ij})$ in the unit ball of $B(\ell^n_2)$ such that $v_{ij} = \lambda_ia_{ij}\mu_j$.
\item[\rm (ii)] $N(v)\le 1$ iff there are $\lambda',\mu'$ in the unit ball of $\ell^n_2$ and $(b_{ij})$ in the unit ball of $B_r(\ell^n_2)$ such that
$
v_{ij} = \lambda'_ib_{ij}\mu'_j.
$
\end{itemize}
\end{pro}

\begin{proof}
(i) is a classical fact (cf.\ e.g.\ \cite[Prop.~5.4]{P4}). To verify (ii), note that $N(v) = \sum\limits_{ij}|v_{ij}|$. Assume $N(v)=1$. Let then $\lambda'_i = (\sum_j |v_{ij}|)^{1/2}$ and $\mu'_j = \Big(\sum\limits_i |v_{ij}|\Big)^{1/2}$ and $b_{ij}=v_{ij} (\lambda'_i \mu'_j)^{-1}$. We have then (with the convention $\frac00=0$)
\[
|b_{ij}|\le |b^0_{ij}|^{1/2} |b^1_{ij}|^{1/2}
\]
with $b^0_{ij} = |v_{ij}|(\sum_j|v_{ij}|)^{-1}$ and $b^1_{ij} = |v_{ij}| (\sum_i |v_{ij}|)^{-1}$. Since $\sup\limits_i \sum_j |b^0_{ij}| \le 1$ and $\sup\limits_j \sum_i|b^1_{ij}|\le 1$, by \ref{comsec1.3bis} we have $\|b\|_{\text{reg}} \le 1$.
\end{proof}

\begin{pro}\label{compro7.3}
 Consider $\varphi\colon \ \ell^n_2\to \ell^n_2$.
\begin{itemize}
 \item[\rm (i)] $\|\varphi\|_{B(\ell^n_2)^*} \le 1$ iff there are $\lambda,\mu$ in the unit ball of $\ell^n_2$ and $v\colon \ \ell^n_1\to \ell^n_\infty$ with $\gamma_H(v)\le 1$ such that $\varphi_{ij} = \lambda_i v_{ij}\mu_j$ for all $i,j$.
\item[\rm (ii)] $\|\varphi\|_{B_r(\ell^n_2)^*} \le 1$ iff there are $\lambda,\mu$ in the unit ball of $\ell^n_2$ and $v\colon \ \ell^n_1\to \ell^n_\infty$ with $\|v\|\le 1$ such that $\varphi_{ij} = \lambda_iv_{ij}\mu_j$ for all $i,j$.
\end{itemize}
\end{pro}

\begin{proof}
  (i) If $\varphi$ factors as indicated we have $v=v_1v_2$ with $v_1\colon \ H\to \ell^n_\infty$ and $v_2\colon\ \ell^n_1\to H$ such that $\|v_1\| \|v_2\|\le 1$.
\n Let $D_\lambda$ and $D_\mu$ denote the diagonal operators with coefficients $(\lambda_i)$ and $(\mu_j)$. We have then $\varphi = D_\lambda v_1v_2D_\mu$, hence using the Hilbert--Schmidt norm $\|\cdot\|_{HS}$ we find that $\|\varphi\|_{B(\ell^n_2)^*}$ (which is the trace class norm of $\varphi$) is $\le\|D_\lambda v_1\|_{HS} \|v_2D_\mu\|_{HS} \le 1$. Conversely, if the trace class norm of $\varphi$ is $\le 1$, then for some $H$ Hilbert (actually $H=\ell^n_2$) we can write $\varphi = \varphi_1\varphi_2$, $\varphi_2\colon \ \ell^n_2\to H$ and $\varphi_1\colon \ H\to \ell^n_2$ such that $\|\varphi_1\|_{HS} \|\varphi_2\|_{HS}\le 1$. Let $v_2\colon \ \ell^n_1\to H$ and $v_1\colon \ H\to \ell^n_\infty$ be the maps defined by $v_2e_j = (\varphi_2e_j) \|\varphi_2e_j\|^{-1}$ and $v^*_1e_i = (\varphi_1^*e_i)\|\varphi_1^*e_i\|^{-1}$. Note that $\|v_1\| = \|v_2\|=1$. Let $v=v_1v_2$ and $\lambda_i = \|\varphi_1^*e_i\|$, $\mu_j = \|\varphi_2 e_j\|$. We have then $\|v\|\le 1$ and $\varphi_{ij} =\langle \varphi e_j,e_i\rangle= \lambda_i v_{ij}\mu_j$, which verifies (i).

(ii) By Lemma \ref{comlem1.6}, $\|\varphi\|_{B_r(\ell^n_2)^*}\le 1$ iff there are $\lambda,\mu$ in the unit ball of $\ell^n_2$ and $v\colon \ \ell^n_1\to \ell^n_\infty$ with $\|v\| = \sup\limits_{ij} |v_{ij}| \le 1$ such that $ 
 \varphi_{ij} = \lambda_iv_{ij}\mu_j. $
\end{proof}

\subsection{}\label{comsec1.10}

In the sequel, we will invoke several times ``a measurable selection argument.'' Each time, the following well known fact will be sufficient for our purposes.
Consider a continuous surjection
$f\colon \ K \to L$ from a compact metric space $K$ onto another one $L$. Then
there is a Borel measurable map $g\colon \ L \to K$ lifting $f$, i.e.  such that $f\circ g$
is the identity on $L$.
This (now folkloric) fact   essentially goes back to von Neumann. 
 The references \cite[p. 9]{JR} or \cite[chap. 5]{Sr} contain considerably more sophisticated results.

\subsection{}\label{comsec1.11} 
 Throughout this memoir (at least until we reach \S \ref{comsec9}), given an operator $T\colon \ L_p(\mu)\to L_p(\nu)$ such that $T\otimes id_X$ extends to a bounded operator from $L_p(\mu;X)$ to $L_p(\nu;X)$, we will denote for short by
\[
T_X\colon \ L_p(\mu;X)\to L_p(\nu;X)
\]
the resulting operator. In \S \ref{comsec9}, this notation will be extended to the non-commutative setting.

\section{Regular and fully contractive operators}\label{comsec2}

For short we say that an operator $T\colon \ L_p(\mu)\to L_p(\nu)$, $(1<p<\infty)$ is ``fully contractive'' if $T$ is a contraction from $L_q(\mu)$ to $L_q(\nu)$ simultaneously for $q=1$ and $q=\infty$ (and hence for all $1\le q\le \infty$ by interpolation).

It is known (and easy to show) that this implies $\|T\|_{\text{reg}}\le 1$. Indeed, since for $T\colon \ L_q(\mu)\to L_q(\nu)$ we have $\|T\| = \|T\|_{\text{reg}}$ both for $q=1$ and $q=\infty$, it follows by interpolation that $\|T\colon\ L_p(\mu)\to L_p(\nu)\|_{\text{reg}}\le 1$. 

In the next statement, we use a change of density argument to replace fully contractive operators by regular ones in the definition of $\Delta_X$. This kind of argument (related to the classical  ``Schur test" and its converse) is well known, see e.g. \cite{Ko,JJ,We}.

\begin{pro}\label{compro2.1}
Let $X$ be a Banach space. Fix $0<\vp<1$ and $\delta>0$. The following are equivalent:
\begin{itemize}
\item[\rm (i)] For any $n$ and any $n\times n$ matrix $T = [a_{ij}]$ with $\|T\colon \ \ell^n_2\to \ell^n_2\| \le \vp$ and such that $\|T\colon \ \ell^n_2\to \ell^n_2\|_{\text{reg}}\le 1$, we have $\|T_X\colon \ \ell^n_2(X) \to \ell^n_2(X)\|\le \delta$.
\item[\rm (i)$'$] For any $n$ and any $n\times n$ matrix $T = [a_{ij}]$ with $\|T\colon \ \ell^n_2 \to \ell^n_2\|\le\vp$ which is fully contractive (i.e.\ contractive on $\ell^n_q$ for all $1\le q\le\infty$), we have
\[
\|T_X\colon \ \ell^n_2(X) \to \ell^n_2(X)\|\le \delta.
\]
\item[\rm (ii)] For any measure spaces $(\Omega,\mu)$, $(\Omega',\mu')$, for any regular operator $T\colon \ L_2(\mu)\to L_2(\mu')$ with $\|T\|_{\text{reg}}\le 1$ and $\|T\|\le\vp$, we have $\|T_X\colon\ L_2(X) \to L_2(X)\|\le\delta$.
\item[\rm (ii)$'$] For any $(\Omega,\mu), (\Omega',\mu')$ and any fully contractive $T\colon \ L_2(\mu) \to L_2(\mu')$ with $\|T\| \le \vp$ we have $\|T_X\colon \ L_2(X)\to L_2(X)\|\le \delta$.
\end{itemize}
\end{pro}

\begin{proof} 
By discretization arguments one can show rather easily that (i) $\Leftrightarrow$ (ii) and (i)$'$ $\Leftrightarrow$ (ii)$'$. Moreover (i) $\Rightarrow$ (i)$'$ and (ii) $\Rightarrow$ (ii)$'$ are trivial, since fully contractive implies regular. Therefore, it suffices to show (ii)$'$ $\Rightarrow$ (i). Assume (ii)$'$. It suffices to consider $T = [a_{ij}]$ with $\|T\colon \ \ell^n_2\to \ell^n_2\| \le \vp$
and such that $\|T\colon \ \ell^n_2\to \ell^n_2\|_{\text{reg}}< 1$, so that $S = [|a_{ij}|]$ also has norm $<1$ on $\ell^n_2$. Without loss of generality, we may assume (by density) that $|a_{ij}|>0$ for all $i,j$ and actually that $\|S\colon \ \ell^n_2\to \ell^n_2\| = 1$. By Perron--Frobenius, there is $\xi\in \ell^n_2$ with $\xi_i>0$ for all $i$ such that $SS^*\xi = \xi$. Let then $\eta = S^*\xi$. We have then a fortiori:
\[
S\eta \le \xi \quad (\text{resp. } S^*\xi\le \eta).
\]
This implies that $[\xi^{-1}_i|a_{ij}| \eta_j]$ (resp.\ $[\xi_i|a_{ij}|\eta^{-1}_j]$) is a contraction on $\ell^n_\infty$ (resp.\ $\ell^n_1$). Let $\mu$ (resp.\ $\mu'$) denote the measure $\sum \eta^2_j\delta_j$ (resp. $\sum \xi^2_i\delta_i$) on $\{1,\ldots, n\}$. Then, we find that the kernel $K(i,j) = \xi^{-1}_i a_{ij}\eta^{-1}_j$ defines a fully contractive operator $\widehat T$ from $L_1(\mu)$ to $L_1(\mu')$. More precisely, $\widehat T$ is defined by
\begin{equation}
\widehat Tf = \int K(s,t)f(t)d\mu(t)\tag*{$\forall f\in L_1(\mu)$}
\end{equation}
or equivalently
\begin{equation}\label{comeq2.1}
(\widehat Tf)_i = \sum\nolimits_j K(i,j) f(j)\eta^2_j = \sum\nolimits_j \xi^{-1}_i a_{ij}\eta_j.
\end{equation}
Assuming (ii)$'$, this implies
\[
\|\widehat T_X\colon \ L_2(\mu;X)\to L_2(\mu';X)\|\le \delta.
\]
More explicitly, for any $x = (x_1,\ldots, x_n)$ in $L_2(\mu;X)$ we have by \eqref{comeq2.1}
\[
\left(\sum\nolimits_i \xi^2_i\left\|\sum\nolimits_j\xi^{-1}_i a_{ij}\eta_jx_j\right\|^2 \right)^{1/2} \le \delta\left(\sum \eta^2_j\|x_j\|^2\right)^{1/2}
\]
Therefore (replacing $(x_j)$ by $(\eta^{-1}_jx_j)$) we conclude $\|T_X\colon \ \ell^n_2(X) \to \ell^n_2(X)\|\le \delta$.
\end{proof}

In view of V.~Lafforgue's question mentioned above it is natural to introduce the following ``modulus'' associated to any Banach space $X$
\[
\Delta_X(\vp) = \sup\{\|T_X\|\}
\]
where the supremum runs over all pairs of measure spaces $(\Omega,\mu)$, $(\Omega',\mu')$ and all operators $T\colon \ L_2(\mu)\to L_2(\mu')$ such that $\|T\|_{\text{reg}}\le 1$ and $\|T\|\le\vp$.

By the preceding statement, $\Delta_X(\vp)$ reduces to:
\begin{equation}\label{comeq2.2}
\Delta_X(\vp) = \sup_{n\ge 1} \sup \{\|T_X\|\mid T\in B(\ell^n_2)\||T|\|\le 1, \|T\|\le \vp\}.
\end{equation}
Moreover, we also have
\begin{equation}\label{comeq2.3}
\Delta_X(\vp) = \sup_{n\ge 1} \sup\{\|T_X\|\mid T\in B(\ell^n_2) \text{ fully contractive with } \|T\|\le \vp\}.
\end{equation}
We propose the following terminology:

\begin{defn}
A Banach space $X$ is  ``curved'' if $\Delta_X(\vp_0)<1$ for some $\vp_0<1$. \\
We say that $X$ is ``fully curved'' if $\Delta_X(\vp)<1$ for any $\vp<1$.\\
We say that $X$ is ``uniformly curved'' if $\Delta_X(\vp)\to 0$ when $\vp\to 0$.

\end{defn}
\n Equivalently, $X$ is   curved iff there is $\vp_0>0$ such that $\Delta_X(\vp)<1$ for all $\vp<\vp_0$.  

\n If $X$ is $C$-isomorphic to $Y$, then $C^{-1}\Delta_Y(\vp)\le \Delta_X(\vp)\le C\Delta_Y(\vp)$, so that $X$ is uniformly curved iff $Y$ also is.

\n Trivially, any Hilbert space $H$ is ``fully curved'' and 
``uniformly curved'' since $\Delta_H(\vp)=\vp$ for any $0<\vp<1$.

\n Every finite dimensional space is ``curved,'' since  $\Delta_X(\vp)\le (\dim(X))^{1/2}\vp$ for all $0<\vp<1$. Indeed,
it is well known that any $d$-dimensional space $X$ is $\sqrt{d}$-isomorphic to 
$\ell_2^d$ (see e.g. \cite[p.16]{P00}).

\n However, there are simple examples of $2$-dimensional
spaces that are not fully curved: for instance $\ell_1^2$ or $\ell_\infty^2$, since these are not ``uniformly nonsquare".

\n Recall that $X$ is called {\it uniformly nonsquare} if there is a number $\delta<1$ such that for any pair $x,y$ in the unit ball we have either $\|2^{-1}(x+y)\|\le\delta$ or $\|2^{-1}(x-y)\|\le \delta$.

\n By interpolation (see below),  if $1<p<\infty$, any $L_p$-space   is ``fully curved'' and 
``uniformly curved''.

 The operator that V.~Lafforgue had in mind initially is the operator ${\tau}\colon \ \ell^2_2\to \ell^2_2$ defined by ${\tau}(x,y) = (2^{-1}(x+y)$, $2^{-1}(x-y))$. This is clearly fully contractive but since ${\tau}$ is equal to $2^{-1/2}\times$ (rotation by $\pi/4$) we have $\|{\tau}\| = 2^{-1/2}<1$.

This shows $X$ fully curved $\Rightarrow X$ uniformly nonsquare.  Indeed, if $\Delta_X(2^{-1/2}) <1$ then $X$ is uniformly nonsquare since, using  ${\tau}_X$,   we find, for any pair $x,y$ in the unit ball
\[
(2^{-2}\|x+y\|^2 + 2^{-2}\|x-y\|^2)^{1/2} \le \Delta_X(2^{-1/2}) (\|x\|^2 + \|y\|^2)^{1/2}
\]
and hence
\[
\min\{\|2^{-1}(x+y)\|, \|2^{-1}(x-y)\|\} \le \Delta_X(2^{-1/2})<1.
\]

 Modulo the beautiful results of James and Enflo (\cite{Ja,E}) on ``super-reflexivity,'' this observation shows that any uniformly curved Banach space is isomorphic to a uniformly convex one. 

 Recall that a Banach space is called {\it  uniformly convex} if $\delta_X(\vp)>0$ for any $0<\vp\le 2$ where
\[
\delta_X(\vp) = \inf\{1-\|2^{-1}(x+y)\|\mid \|x\| \le 1, \|y\|\le 1, \|x-y\|\ge \vp\}.
\]
In \cite{P2}, we proved that any uniformly convex space can be equivalently renormed so that its uniform convexity modulus satisfies $\delta_X(\vp)\ge C\vp^q$ $\forall \vp\in [0,2]$ for some
$q<\infty$ (and for some constant $C>0$).

Thus it is natural to raise the following:

\begin{prbl}\label{comprbl2.2}
Can any uniformly  curved  (or even merely curved) space be equivalently renormed
to be fully curved ? 
 \end{prbl}
\begin{prbl}\label{comprbl2.22}
\n If  $X$ is uniformly curved, is it true that
$\Delta_X(\vp) \in O(\vp^\alpha)$ for some
$\alpha>0$ ?
 \end{prbl}

By \cite{P1}, the answer is positive for Banach lattices.

We will prove below (see Corollary \ref{comcor4.3}) that a Banach space $X$ satisfies this iff it is isomorphic to a
subquotient of a  $\theta$-Hilbertian space for some $\theta>0$. This provides a rather strong motivation for the preceding question.

 A similar reasoning applied to the operator ${\tau}^{\otimes n}$ shows that if $\Delta_X(2^{-n/2})<1$ then $X$ does not contain almost isometrically $\ell^{2^n}_1$ (viewed here, exceptionally,  as a real Banach space). We say $X$ contains $E$ almost isometrically if for any $\vp>0$ there is a subspace $\widetilde E\subset X$ with Banach--Mazur distance $d(E,\widetilde E) < 1+\vp$.

By known results (see \cite{Ma}), this shows that $X$  curved $\Rightarrow X$ of type $p$ for some $p>1$. However, a stronger conclusion can be reached using the Hilbert transform, say in its simplest discrete matricial form. Fix $n\ge 1$. Let $\Gamma(n) = [\Gamma(n)(i,j)]$ be the Hilbert $n\times n$-matrix defined by
\[
\Gamma(n)(i,j) =  (n-(i+j))^{-1}\quad {\rm if} \ \ n-(i+j)\not= 0,
\] \centerline{ and $\Gamma(n)(i,j) =0$ otherwise.}
Then it is well known that the operator $\Gamma(n)\colon \ \ell^n_2\to \ell^n_2$ satisfies
\[
\|\Gamma(n)\| \le C\quad \text{and}\quad \|\Gamma(n)\|_{\text{reg}} \le C {\log}(n+1)
\]
for some  constant $C$ (independent of $n$). Thus, for any Banach space $X$, we have
\[
\|\Gamma(n)_X\|\le C {\log}(n+1) \Delta_X(( {\log}(n+1))^{-1}).
\]
Therefore, we find
\[
X \quad \text{uniformly curved}\  \Rightarrow \ \|\Gamma(n)_X\| \in o({\log}(n)).
\]
As observed in \cite{P0}, this implies that $X$ is super-reflexive and hence, by Enflo's theorem \cite{E}, isomorphic to a uniformly convex space.

It is natural to wonder about the converse:

\begin{prbl}\label{comprbl2.3}
Is is true that 
\begin{quote}
uniformly convex $\Rightarrow$  curved or fully curved ?
\end{quote}

It is true that
\begin{quote}
super-reflexive $\Rightarrow$ uniformly curved ?
\end{quote}
Note that, as already mentioned,  ``uniformly curved'' is obviously stable under isomorphism, while this is obviously false for ``curved'' spaces.
\end{prbl}

\section{Remarks on expanding graphs}\label{expan}

Many interesting examples of regular operators
with regular norm $1$ and small norm on $L_2$ can be 
found in the theory of contractive 
semi-groups 
on $L_p$ (e.g. those generated by the Laplacian
on  a Riemannian manifold) or in that of random walks on groups. In this section, we illustrate this  
by 
combinatorial Laplacians on   expanding graphs (cf. e.g. \cite{Lu}).

We refer the reader to the web site
http://kam.mff.cuni.cz/~matousek/metrop.ps
for a list of problems (in particular  Linial's   contribution) 
 related to  this section.
See also \cite{NPSS} for a similar theme. 

Let $G$ be a finite set equipped with a graph structure, so that $G$ is the vertex set and we give ourselves a symmetric set of edges $E\subset G\times G$.

We assume the graph $G$ $k$-regular, i.e.\ such that for each $x$ there are precisely $k$ vertices $y$ such that $(x,y)\in E$. Let $M_G\colon \ \ell_2(G)\to \ell_2(G)$ be the ``Markov operator'' defined by
\begin{equation}
 M^Gf(x) = \frac1k \sum_{(x,y)\in E} f(y).\tag*{$\forall f\in \ell_2(G)$}
\end{equation}
For simplicity, we consider only non-bipartite graphs  (equivalently $-k$ is not an eigenvalue of $M^G$). Let $Ef$ denote the average of $f$ over $G$, i.e.
\begin{equation}
Ef(x) = |G|^{-1} \sum_{y\in G} f(y)\tag*{$\forall f\in \ell_2(G)$}
\end{equation}
and let
\[
\ell^0_2(G) = \{f\in \ell_2(G)\mid Ef=0\}.
\]
We set
\[
 \vp(G) = \|M^G|_{\ell^0_2(G)}\|
\]
so that $\vp(G)$ is the smallest constant such that
\begin{equation}
\|M^Gf-Ef\|_2 \le \vp(G)\|f\|_2.\tag*{$\forall f\in \ell_2(G)$}
\end{equation}
A sequence of $k$-regular graphs $\{G(m)\}$ is called expanding if $\sup\limits_m \vp(G(m))<1$ and $|G(m)|\to \infty$ when $m\to\infty$. In connection with the Baum--Connes conjecture (see \cite{Gro,Roe,KY}) it is of interest to understand in which Banach spaces $X$ we can embed coarsely, but uniformly over $m$, such a sequence $\{G(m)\}$ viewed as a sequence of metric spaces.

More precisely, we will say that $\{G(m)\}$ uniformly coarsely embeds in $X$ if there are a function $\rho\colon\ {\bb R}_+ \to {\bb R}_+$ such that $\lim\limits_{r\to\infty} \rho(r) = \infty$ and mappings $f_m\colon \ G(m)\to X$ such that
\begin{equation}\label{comseceq1}
\forall x,y\in G(m)\qquad \qquad \rho(d(x,y))\le \|f_m(x) - f_m(y)\| \le d(x,y).
\end{equation}
Here $d(x,y)$ denotes the geodesic distance on the graph. 
This notion is due to Gromov (see \cite[p. 211]{Gro}, where it is called
``uniform embedding"). 

Given a Banach space $X$, let us denote by $\vp(G,X)$ the $X$-valued version of $\vp(G,X)$, i.e.\ the smallest $\vp$ such that for any $f\colon \ G\to X$ with mean zero (i.e.\ $\sum f(x)=0$) we have
\[
\|M^Gf-Ef\|_{\ell_2(X)} \le \vp\|f\|_{\ell_2(X)},
\]
where $M^G$ and $E$ are defined as before but now for $X$-valued functions on $G$.

By a well known argument, it can be shown that if $ \sup_m \vp(G(m), X)<1$, then $X$ cannot contain $\{G(m)\}$ uniformly coarsely. This is the motivation behind V.~Lafforgue's question mentioned above. 

To explain this, consider again a finite $k$-regular graph $G$ as before.
 Let $\vp = \vp(G,X)$. Assume that $\vp<1$. Assume that there is $n\ge 1$ such that $\delta = 2\Delta_X(\vp^n/2)<1$. We have then for any $f\colon \ G\times G\to X$
\begin{equation}\label{comseceq2}
\left(|G|^{-1} \sum_{(x,y)\in G^2} \|f(x)-f(y)\|^2\right)^{1/2} \le 2(1-\delta)^{-1} \|f-(M^G)^nf\|_{\ell_2(X)}.
\end{equation}
Indeed, for any $n\ge 1$ we have
$
 \|(M^G-E)^n (1-E)\|_{B(\ell_2(G))}\le \vp^n
$
but also (note $M^G   E=E M^G  =E$)
\[
(M^G-E)^n(1-E) = (M^G)^n (1-E) = (M^G)^n-E
\]
and hence (since $\|M^G\|_{\text{reg}}$ and $\|E\|_{\text{reg}}$ are $\le 1$)
\[
\|(M^G-E)^n(1-E)\|_{\text{reg}}\le 2.
\]
Therefore
\[
\|(M^G)^n-E\|_{B(\ell_2(G;X))} \le 2\Delta_X(\vp^n/2) = \delta,
\]
and hence for any $f\colon \ G\to X$
\begin{align*}
\|f-Ef\|_{\ell_2(X)} &\le \|f-(M^G)^nf\|_{\ell_2(X)} + \|(M^G)^nf -Ef\|_{\ell_2(X)}\\
&\le \|f-(M^G)^nf\|_{\ell_2(X)} + \delta\|f-Ef\|_{\ell_2(X)}
\end{align*}
from which \eqref{comseceq2} follows immediately, using  the following classical and elementary inequality
\begin{equation}\label{comseceq3}
\left(|G|^{-2} \sum_{(x,y)\in G\times G} \|f(x)-f(y)\|^2\right)^{1/2} \le 2\left(|G|^{-1} \sum_{x\in G} \|f(x)-Ef\|^2\right)^{1/2}.
\end{equation}
 We now show that \eqref{comseceq2} is an obstruction to \eqref{comseceq1}. Indeed, 
 since $\|f_m(x)-f_m(y)\|\le 1$ if $(x,y)\in E$   we find easily
\[
 \|f_m-(M^G)^nf_m\|_{\ell_2(X)} \le |G|^{1/2}n
\]
but also for any $R>0$ if we set
\[
 E_R(m) = \{(x,y)\in G(m)^2\mid d(x,y)\ge R\}, 
\]
since $\|f(x) - f(y)\| \ge \rho(R)$ for any $(x,y)$ in $E_R$, we find
\[
 \rho(R) \left(\frac{|E_R(m)|}{|G(m)|^2}\right)^{1/2} \le 2(1-\delta)^{-1}n.
\]
But now, since $|G(m)|\to\infty$, for each fixed $R$ we have
\[
 |G(m)|^{-2} |E_R(m)|\to1
\]
and hence we obtain
\[
 \rho(R) \le 2(1-\delta)^{-1}n
\]
which contradicts the assumption that $\rho(R)$ is unbounded when $R\to\infty$. Thus we obtain (as observed by V. Lafforgue):

\begin{prop}
 If $X$ is uniformly curved, $X$ cannot contain uniformly coarsely an expanding sequence.
\end{prop}

\begin{proof}
Indeed, if $\vp = \sup\limits_m \vp(G(m))<1$ and $X$ is uniformly curved, then $2\Delta_X(\vp^n/2)\to 0$ when $n\to\infty$ so we can always choose $n$ so that $\delta<1$.
\end{proof}

In \cite{La}, in answer to a question of Naor,  V. Lafforgue shows
 that certain specific expanding sequences do not embed uniformly coarsely even in any
Banach space with non trivial type. This is a  stronger statement
than the preceding one, since the class of spaces with non trivial type
is strictly larger than that of uniformly curved spaces (because
uniformly curved implies super-reflexive), but so far this is known only
for special expanding sequences. For example it is not known
either for the expanding sequence derived from a finite set of generators
and the family of finite quotient groups of $SL_3({\bb  Z})$, or for the ``Ramanujan graph" expanding sequences   (see  \cite{Lu}).

\begin{prbl} Characterize the Banach spaces $X$ for which there is
a function $d \colon \ (0,1]\to [0,2]$, with $d(\vp)$ tending to zero
when $\vp\to 0$, such that for any
finite graph $G$ as above we have
$$\vp(G,X) \le d (\vp(G)).$$
\end{prbl}
See \cite{Lam} for related results.

\begin{rk} Note that \eqref{comseceq2} implies that there is a constant $C$ such that
\begin{equation}
\left(|G|^{-1} \sum_{(x,y)\in G^2} \|f(x)-f(y)\|^2\right)^{1/2} \le C \|f-(M^G)f\|_{\ell_2(X)}.
\end{equation}
Indeed, since we have  $\|f- (M^G)^nf\| \le \sum_{k=0}^{n-1} \|(M^G)^k (f- (M^G)f)\|$
we can take $C=2n (1-\delta)^{-1}$.
A fortiori this implies
\begin{equation}\label{cot}
 |G|^{-2} \sum_{(x,y)\in G^2} \|f(x)-f(y)\|  \le C \sup_{d(x,y)=1}\{ \|f(x)-f(y)\|\} .
\end{equation}
The preceding (well known) argument clearly shows that  $X$ 
 cannot contain uniformly coarsely an expanding sequence if it satisfies
\eqref{cot} with the same $C$ when $G$ runs over that   sequence.
However, it may be that \eqref{cot} is satisfied by a much larger class of spaces than those of non trivial type.
Indeed, it is an open problem whether any space with nontrivial {\it cotype} satisfies 
\eqref{cot} with a fixed constant $C$ valid for any $G$ running in a given expanding sequence.
This would show that finite cotype is an obstruction to containing uniformly coarsely an expanding sequence.
Note that    the metric space $\ell_1$ 
with metric $(x,y)\mapsto \|x-y\|^{1/2}$ embeds isometrically into Hilbert space
and hence does not contain uniformly coarsely an expanding sequence (see \cite{Mat1}).
More generally, see the appendix of \cite{Oz} for the case of   Banach lattices  of finite cotype,
or spaces  with a unit ball uniformly homeomorphic to a subset of $\ell_2$.

 We refer to \cite{Ma} for a discussion of type and cotype and to \cite{Ost,Mat2} for a survey of recent work on coarse embeddability
between Banach spaces.

\end{rk}

\section{A duality operators/classes of Banach spaces}\label{comsec3}

In this section, we describe a ``duality'' (or a ``polarity'') between classes of Banach spaces on one hand and classes of operators on $L_p$ on the other hand. The general ideas underlying this are already implicitly in Kwapi\'en's remarkable papers \cite{Kw1,Kw2} and  Roberto Hernandez's thesis (cf.\ \cite{He}). See also \cite{Her,Buk1,Buk2,Vi}. 
In connection with the notion of $p$-complete boundedness
from  \cite{P1+}, we refer the interested reader  to \cite[Th 1.2.3.7 and Cor. 1.2.4.7]{Ju} for an even more general duality.
 
We seize the occasion to develop this theme more explicitly in this paper.

Fix $1\le p\le \infty$. Consider a class of operators ${\cl C}$ acting between $L_p$-spaces. A typical element of ${\cl C}$ is a (bounded) operator $T\colon \ L_p(\mu)\to L_p(\mu')$ where $(\Omega,{\cl A},\mu)$ and $(\Omega', {\cl A}', \mu')$ are measure spaces.
We can associate  to ${\cl C}$ the class ${\cl C}^{\circ}$ of all Banach spaces $X$ such that $T_X\colon \ L_p(\mu;X)\to L_p(\mu';X)$ is bounded for any $T$ in ${\cl C}$ (as above).

Conversely, given a class ${\cl B}$ of Banach spaces, we associate to it the class ${\cl B}^{\circ}$ of all operators $T$ between two $L_p$-spaces (as above) such that $T_X$ is bounded for any $X$ in ${\cl B}$. This immediately raises two interesting problems (the second one is essentially solved by Theorem \ref{comthm3.1} below):

\begin{prbl}\label{comprbl3.1}
Characterize ${\cl C}^{{\circ\circ}}$.
\end{prbl}

\begin{prbl}\label{comprbl3.2}
Characterize ${\cl B}^{{\circ\circ}}$.
\end{prbl}

\subsection{}\label{comsec3.1}

One can also formulate an \emph{isometric variant} of these problems:\ one defines ${\cl C}^{\circ}$ (resp.\ ${\cl B}^{\circ}$) as those $X$'s (resp.\ those $T$'s) such that $\|T_X\| = \|T\|$ for all $T$ in ${\cl C}$ (resp.\ for all $X$ in ${\cl B}$). In particular:
 \begin{prbl}\label{comprbl3.22}  Given 
an $n\times n$ matrix $a=[a_{ij}]$  of norm $1$ 
when acting on the $n$-dimensional Euclidean space $\ell_2^n$, characterize the class $\{a\}^{{\circ\circ}}$, i.e. the class of
all $n\times n$ matrices $b=[b_{ij}]$ such that $\|b_X\|\le 1$ for all $X$ such that $\|a_X\|\le 1$.
\end{prbl}

\begin{rem}\label{ultr}
Note that for any ${\cl C}$, the class of all $X$ such that $\|T_X\|\le 1$ for all
$T$ in ${\cl C}$ is stable under ultraproducts and $\ell_p$-sums. Moreover,
if $X$ belongs to the class, all subspaces  and all quotients of $X$ also do. 
\end{rem}
\subsection{}\label{comsec3.2}
We should immediately point out that the extreme cases $p=1$ and $p=\infty$ are trivial:\ In that case, ${\cl C}^{\circ}$ always contains \emph{all} Banach spaces and ${\cl B}^{\circ}$ all operators. In the isometric variant the same is  true.

In what precedes, we have allowed $\mu$ and $\mu'$ to vary. If we now fix $\mu$ and $\mu'$ and consider ${\cl C} \subset B(L_p(\mu), L_p(\mu'))$ we can define ${\cl C}^{\circ\circ}$ as the class of those $T$ in $B(L_p(\mu), L_p(\mu'))$ such that $T_X$ is bounded
for all $X$ in ${\cl C}^{\circ}$. In this framework, Problem \ref{comprbl3.1} is somewhat more natural. 

Nevertheless, in both forms, very little is known about Problem \ref{comprbl3.1} except some special cases as follows:

\subsection{}\label{comsec3.3}
If ${\cl C}$ is the class of all bounded operators between $L_p$-spaces, it is trivial that ${\cl C} = {\cl C}^{\circ\circ}$.

\subsection{}\label{comsec3.4}
If $p=2$ and ${\cl C} = \{{\cl F}\}$ where ${\cl F}$ denotes the Fourier transform acting on $L_2$ either on ${\bb Z}, {\bb T}$ or ${\bb R}$ then ${\cl C}^{\circ\circ}$ is the class of all bounded operators between $L_2$-spaces. In the isometric variant, the same holds. This is due to Kwapi\'en \cite{Kw1}.

\subsection{}\label{comsec3.5}
The preceding can also be reformulated in terms of type 2 and cotype 2 (see \cite{Ma} for these notions), using $\Omega = \{-1,1\}^{\bb N}$ equipped with its usual probability measure $\mu$ and denoting by $(\vp_n)$ the coordinates on $\Omega$. Let $T_1\colon \ \ell_2\to L_2(\Omega,\mu)$ (resp.\ $T_2\colon \ L_2(\Omega,\mu)\to \ell_2$) be the linear contractions defined by $T_1e_n = \vp_n$ and $T_2f = \Sigma e_n\int f\vp_n\ d\mu$. Then $\{T_1\}^{\circ}$ (resp.\ $\{T_2\}^{\circ}$) is the class of Banach spaces $X$ of type 2 (resp.\ such that $X^*$ is of type 2). Therefore,
$\{T_1,T_2\}^{\circ}$ is the class of $X$ such that $X$ and its dual are of type 2, i.e. are Hilbertian, and hence  $\{T_1,T_2\}^{\circ\circ}$ is the class of all bounded operators between $L_2$-spaces.

\subsection{}\label{comsec3.5bis} If $(\Omega,\mu)$ is as in   \ref{comsec3.5}, let $P:\ L_2(\mu)\to L_2(\mu)$ denote the orthogonal projection onto the closed span of the sequence of coordinate functions $\{\vp_n\}$. Then $\{P\}^{{\circ}}$ is the class of the so-called $K$-convex spaces, that coincides with the class of spaces that are of type
$p$ for some $p>1$, or equivalently do not contain $\ell_1^n$'s uniformly (see \cite{Ma} for all this).

Bourgain's papers \cite{Bo1,Bo2,Bo3} on the Hausdorff-Young inequality in the Banach space valued case
are also somewhat related to the same theme.

\subsection{}\label{comsec3.6}
If ${\cl C} = \{T\}$ where $T$ is the Hilbert transform on $L_p({\bb T})$ (or $L_p({\bb R})$), then ${\cl C}^{{\circ\circ}}$ contains all martingale transforms and a number of Fourier multipliers of H\"ormander type. This is due to Bourgain \cite{Bo}. Conversely if ${\cl C}$ is the class of all martingale transforms, then it was known before Bourgain's paper that ${\cl C}^{{\circ\circ}}$ contains the Hilbert transform and various singular integrals. This is due to Burkholder and McConnell (cf.\ \cite{BGS, Bu}). This leaves entirely open many interesting questions. For instance, it would be nice (perhaps not so hopeless ?) to have a description of $\{T\}^{{\circ\circ}}$ when $T$ is the Hilbert transform on $L_p({\bb T})$, say for $p=2$. See \cite{Fi,Pe} for results in this direction.

\subsection{}\label{comsec3.6bis}
In sharp contrast, Problem \ref{comprbl3.2} is in some sense already solved by the following theorem of Hernandez \cite{He}, extending Kwapi\'en's ideas in \cite{Kw2}.
Kwapie\'n proved that a Banach space $X$ is $C$-isomorphic to a subquotient
of $L_p$ iff $\|T_X\|\le C\|T\|$ for any operator on $L_p$. In particular, with the  
notation in Problem \ref{comprbl3.1},
taking $C=1$, this says that if ${\cl C}$ is the class of all operators on $L_p$,
and ${\cl B}$ is reduced to  the single space $\bb C$, so that  ${\cl C}={\cl B}^{{\circ}}$
then ${\cl C}^{{\circ}}={\cl B}^{{\circ\circ}}$ is the class of subquotients
of $L_p$.

\begin{thm}[\cite{He}]\label{comthm3.1}
With the above notation, ${\cl B}^{{\circ\circ}}$ is the class of all subspaces of quotients of ultraproducts   of 
direct sums in the $\ell_p$-sense of   finite families of spaces of ${\cl B}$, i.e. spaces of
the form $\left(\oplus \sum\limits_{i\in I} X_i\right)_p$ with $X_i\in {\cl B}$ for all $i$ in a finite set $I$.
\end{thm}
Note that for $p=1$ (resp.\ $p=\infty$) every Banach space is a quotient of $\ell_1$ (resp.\ a subspace of $\ell_\infty$), so we obtain all Banach spaces, in agreement with \ref{comsec3.2} above.

To abbreviate, from now on we will say {\it  ``subquotient" } instead of subspace of a quotient. Note the 
following elementary fact: the class of all subquotients of a Banach space $X$ is identical to the class of  all quotient spaces of a subspace of $X$.  Therefore there is no need to further iterate
passage to subspaces and quotients.

The isomorphic version is as follows:
\begin{thm}[\cite{He}]\label{comthm3.1bis} Let $C\ge 1$ be a constant. 
The following properties of a Banach space  $X$ are equivalent:
\begin{itemize}
\item[\rm (i)] $X$ is $C$-isomorphic to a  subquotient of an ultraproduct   of 
direct sums in the $\ell_p$-sense of   finite families of spaces of ${\cl B}$, i.e. spaces of
the form $\left(\oplus \sum\limits_{i\in I} X_i\right)_p$ with $X_i\in {\cl B}$ for all $i$ in a finite set $I$.
\item[\rm (ii)] $\|T_X\|\le C$ for any $n$ and any $T\colon\ \ell_p^n\to \ell_p^n$ such that $\sup\{\|T_Y\| \ \mid \ Y \in {\cl B}\} \le 1$.
\end{itemize}

\end{thm}

Since the detailed proof in \cite{He2}  might be difficult to access, we decided to outline one here.

\begin{proof}[First part of Proof of Theorem \ref{comthm3.1bis}]
(i) $\Rightarrow$ (ii) is obvious. Assume (ii). Since ${\cl B}$ is stable by $\ell_p$-sums we may reduce to the case when ${\cl B}$ is reduced to a single space, that is to say ${\cl B}=\{ \ell_p({\cl X})\}$, (just take  $ {\cl X} =  (\oplus\sum\limits_{Y\in {\cl B}}Y)_p )$. Then since $X$ embeds isometrically into an ultraproduct of its finite dimensional subspaces and since (ii) is inherited by all finite dimensional subspaces of $X$, we may reduce the proof of (ii) $\Rightarrow$ (i) to the case when $\dim(X)<\infty$. Thus we assume that $\dim(X)<\infty$. Fix $\vp>0$. There is an integer $N$ and an embedding $J_1\colon \ X\to \ell^N_\infty$ such that 
\begin{equation}\label{R0}
 \forall x\in X\qquad\qquad \|x\|(1+\vp)^{-1} \le \|J_1x\| \le \|x\|.
\end{equation}
Similarly there is (for $N$ large enough) an embedding $J_2\colon \ X^*\to\ell^N_\infty$ such that
\begin{equation}\label{R00} \forall \xi\in X^*\qquad\qquad
 \|\xi\|(1+\vp)^{-1} \le \|J_2\xi\| \le \|\xi\|.  
\end{equation}
Let $u = J_1J^*_2\colon \ \ell^N_1\to \ell^N_\infty$. Note $\|u\| \le \|J_1\| \|J_2\| \le 1$. Then we easily deduce from (ii) that for all $n$ and all $T\colon \ \ell^n_p\to \ell^n_p$ we have
\[
 \|T\otimes u\colon \ \ell^n_p(\ell^N_1) \to \ell^n_p(\ell^N_\infty)\|\le C\|T_{\cl X}\|.
\]
The proof will be completed easily using the next Lemma.
\end{proof}

\begin{lem}\label{rob1}
Consider a linear map $u\colon \ \ell^N_1\to \ell^N_\infty$. Let ${\cl X}$ be an $N$-dimensional Banach space and $1\le p \le \infty$.  Assume that for any $T\colon \ \ell^N_p\to \ell^N_p$ we have
\[
 \|T\otimes u\colon \ \ell^N_p(\ell^N_1) \to \ell^N_p(\ell^N_\infty)\|\le C\|T_{\cl X}\|.
\]
Let ${\cl B}$ be the class of all spaces of the form $\ell^n_p({\cl X})$ with $n\ge 1$. Then
\[
 \gamma_{\cl B}(u) \le C.
\]
\end{lem}

To prove this we first need:

\begin{lem}\label{rob2}
Let ${\cl B}$ be as in the preceding Lemma. For $v\colon \ \ell^N_\infty\to \ell^N_1$ we set
\[
 \gamma^*_{\cl B}(v) = \sup\{|tr(uv)| \ \Big|\ u\colon \ \ell^N_1\to \ell^N_\infty,\  \gamma_{\cl B}(u)\le 1\}.
\]
Then $\gamma^*_{\cl B}(v) \le 1$ iff there are diagonal operators $D\colon \ \ell^N_\infty\to \ell^N_p$, $D'\colon \ \ell^N_p\to \ell^N_1$ with $\|D\| \le 1$, $\|D'\|\le 1$ and $T\colon \ \ell^N_P\to \ell^N_p$ with $\|T_{\cl X}\|\le 1$ such that
\begin{equation}\label{R1}
v^* = D'TD.
\end{equation}
Equivalently if we denote by $(\lambda_j)$ (resp.\ $\mu_i$) the diagonal coefficients of $D$ (resp.\ $D'$) we have
\[
 \gamma^*_{\cl B}(v) = \inf\left\{\left(\sum |\mu_j|^{p'}\right)^{1/p'} \|T_X\| \left(\sum|\lambda_j|^p\right)^{1/p}\right\}
\]
where the infimum runs over all possible factorizations of the form \eqref{R1}.
\end{lem}

\begin{proof}
Let $[u_{ij}]$ and $[v_{ij}]$ be the matrices associated to $u$ and $v$ so that
\[
 tr(uv) = \sum u_{ij}v_{ji}.
\]
Assume $\gamma^*_{\cl B}(v)\le 1$. Consider an $N$-tuple $x_1,\ldots, x_N$ (resp.\ $\xi_1,\ldots, \xi_N$) of elements of $\ell^n_p({\cl X})$ (resp.\ $\ell^n_p({\cl X}^*)$) for some integer $n\ge 1$. Assume $u_{ij} = \langle \xi_i,x_j\rangle$. Then $u = u^*_1u_2$ where $u_2e_j=x_j$ and $u^*_1e_i=\xi_i$ so that we have
\[
 \gamma_{\cl B}(u) \le \sup_j \|x_j\| \sup_i \|\xi_i\|,
\]
and hence
\begin{equation}\label{R2}
 \left|\sum v_{ji}\langle \xi_i,x_j\rangle\right| \le \sup_j \|x_j\| \sup_i\|\xi_i\|.
\end{equation}
Let us denote $x_j = (x_j(k))$ and $\xi_i = (\xi_i(k))$ so that $\|x_j\| = \left(\sum_k\|x_j(k)\|^p\right)^{1/p}$ and $\|\xi_i\| = \left(\sum_k\|\xi_i(k)\|^{p'}\right)^{1/p'}$. With this notation \eqref{R2} becomes
\[
 \left|\sum_k\sum_{ij} v_{ji}\langle\xi_i(k), x_j(k)\rangle\right| \le \sup_j \left(\sum\nolimits_k\|x_j(k)\|^p \right)^{1/p} \sup_i\left(\sum\nolimits_k\|\xi_i(k)\|^{p'}\right)^{1/p'}.
\]
We will use the elementary identity (variant of the arithmetic/geometric mean inequality)
\begin{equation}\label{R3}
 \forall s,t>0\qquad \qquad s^{\frac1p}t^{\frac1{p'}} = \inf_{\omega>0} \left(\frac{\omega^ps}p + \frac{\omega^{-p'}t}{p'}\right).
\end{equation}
This yields
\[
\left|\sum_k\sum_{ij} v_{ji} \langle\xi_i(k), x_j(k)\rangle\right| \le \frac1p \sup_j \sum_k \|x_j(k)\|^p + \frac1{p'} \sup_i \sum_k\|\xi_i(k)\|^{p'},
\]
and hence (note that the right hand side remains invariant if we replace, say, $(x_j(k))$ by $(\alpha_k x_j(k))$ with $|\alpha_k|=1$)
\begin{equation}\label{R4}
 \sum\nolimits_k\left|\sum_{ij} v_{ji} \langle\xi_i(k), x_j(k)\rangle\right|\le \frac1p \sup_j \sum\nolimits_k\|x_k(k)\|^p + \frac1{p'} \sup_i \sum\nolimits_k\|\xi_i(k)\|^{p'}.
\end{equation}
Then by a (by now routine) Pietsch style application of the Hahn--Banach theorem (see e.g. [~~~]) \eqref{R4} implies that there are probabilities $P$ and $Q$ on $[1,\ldots, N]$ such that the left side of \eqref{R4} is
\[
 \le \frac1p \int \sum\nolimits_k\|x_j(k)\|^p dP(j) + \frac1{p'} \int \sum\nolimits_k \|\xi_i(k)\|^{p'} dQ(i).
\]
 In particular for any $x_1,\ldots, x_N$ in ${\cl X}$ and any $\xi_1,\ldots, \xi_N$ in ${\cl X}^*$ we have
\[
\left|\sum v_{ji}\langle \xi_i,x_j\rangle\right| \le \frac1p \sum\|x_j\|^p P(\{j\}))^{1/p} \left(\sum \|\xi_i\|^{p'} Q(\{i\})\right)^{1/p'} + \frac1{p'}.
\]
Since $\langle\xi_i,x_j\rangle= \langle\omega^{-1}\xi_i, \omega x_j\rangle$ for all $\omega>0$ we can use \eqref{R3} again and we obtain
\[
 \left|\sum v_{ji}\langle \xi_i,x_j\rangle\right| \le \left(\sum \|x_j\|^p P(\{j\})\right)^{1/p} \left(\sum \|\xi_i\|^{p'} Q(\{i\})\right)^{1/p'}.
\]
Let then $T_{ij} = v_{ji}P(\{j\})^{-1/p} Q(\{i\})^{-1/p'}$. We obtain for all $x_j$ in ${\cl X}$ and $\xi_i$ in ${\cl X}^*$
\[
 \left|\sum T_{ij}\langle \xi_i,x_j\rangle\right| \le \left(\sum\|x_j\|^p\right)^{1/p} \left(\sum \|\xi_i\|^{p'}\right)^{1/p'},
\]
which means that $\|T_{\cl X}\|\le 1$. Letting $\lambda_j = P(\{j\})^{1/p}$ and $\mu_i = Q(\{i\})^{1/p'}$ we obtain $v_{ji} = \mu_iT_{ij}\lambda_j$, i.e.\ $v^* = D'TD$ with $\|D\| = \left(\sum|\lambda_j|^p\right)^{1/p}\le 1$ and $\|D'\| = \left(\sum |\mu_i|^{p'}\right)^{1/p'} \le 1$. This proves the ``only if'' part. The proof of the ``if part'' is easy to check by running the preceding argument in reverse.
\end{proof}

\begin{proof}[Proof of Lemma \ref{rob1}]
We claim that the assumption on $u$ implies that for any $v\colon \ \ell^N_\infty\to \ell^N_1$ with $\gamma^*_{\cl B}(v)\le 1$ we have $|tr(uv)|\le C$. Indeed by Lemma \ref{rob2} any such $v$ has a matrix $v_{ij}$ that can be written $v_{ji} = \mu_iT_{ij}\lambda_j$ with $\|T_{\cl X}\|\le 1$, $\sum|\lambda_j|^p \le 1$, $\sum|\mu_i|^{p'}\le 1$. We have then
\begin{align*}
 tr(uv) = \sum u_{ij}v_{ji} = \sum u_{ij}\mu_iT_{ij}\lambda_j=
  \langle (T\otimes u)(a),b\rangle
\end{align*}
where $a\in \ell^N_p\otimes \ell^N_1$ and $b \in \ell^N_{p'} \otimes \ell^N_\infty$ are defined by $a = \sum \lambda_j e_j\otimes e_j$ and $b = \sum \mu_ie_i\otimes e_i$.

But then we have by our assumption on $u$
\[
 |tr(uv)| = |\langle T\otimes u(a), b\rangle| \le \|T\otimes u\|\|a\|_{\ell^N_p(\ell^N_1)} \|b\|_{\ell^N_{p'}(\ell^N_\infty)} \le C\|T_{\cl X}\| \le C.
\]
By duality we conclude that
\[
\gamma_{\cl B}(u) = \sup\{|\langle u,v\rangle|\ \Big| \ \gamma_{{\cl B}^*}(v)\le 1\}\le C.\qquad \qed
\]
\renewcommand{\qed}{}\end{proof}

\begin{rem} The last equality  can be rewritten,
with the preceding notation,
 as
one about Schur multipliers, as follows:
$$\sup\left\{ \|[u_{ij}T_{ij}]\|_{B(\ell^n_p)} \mid\ \|T_{\cl X}\|_{B(\ell^n_p({\cl X}))}\le 1\right\}=
\inf\left \{ \sup_j\left(\sum\nolimits_{k}\|x_j(k)\|_{\cl X}^p \right)^{1/p}  \sup_i\left(\sum\nolimits_{k}\|\xi_i(k)\|_{{\cl X}^*}^{p'}\right)^{1/p'}  \right\}$$
 where the inf runs over all $x_1,\cdots,x_n$ in $\ell_p({\cl X})$
 and all $\xi_1,\cdots,\xi_n$ in $\ell_{p'}({\cl X}^*)=\ell_p({\cl X})^*$ such that
 $u_{ij}=\langle \xi_i,x_j\rangle$ for all $i,j=1,\cdots,n$.
Indeed, the left hand side is equal to $\sup\{|\langle u,v\rangle|\ \Big| \ \gamma_{{\cl B}^*}(v)\le 1\}$
and the right hand side to $\gamma_{{\cl B} }(u)$ with 
${\cl B}=\{ \ell_p({\cl X})  \}$.
\end{rem}

\begin{proof}[Proof of Theorem  \ref{comthm3.1bis}]
 By Lemma \ref{rob1} applied to $u = J_1J^*_2\colon \ \ell^N_1\to \ell^N_\infty$ we have $\gamma_{\cl B}(u) \le C$. Assume for simplicity $\gamma_{\cl B}(u) <C$. Then, for some integer $n$, we have a factorization $u= \alpha\beta$ with $\alpha\colon \ \ell^n_p({\cl X}) \to \ell^N_\infty$ and $\beta\colon \ \ell^N_1 \to \ell^n_p({\cl X})$ such that $\|\alpha\|\|\beta\|<C$. Let $S\subset \ell^n_p({\cl X})$ be the range of $\beta$ and let ${S_0}\subset S$ be the kernel of $\alpha_{|S}$. Let $q\colon \ S\to S/{S_0}$ denote the quotient map and let $\tilde\alpha\colon \ S/{S_0}\to \ell^N_\infty$ be the map defined by $\alpha_{|S}= \tilde\alpha q$. Also let 
 $\tilde\beta\colon\ \ell^N_1\to S/{S_0}$ be defined by $\tilde\beta(\cdot) = q\beta(\cdot)$. We have then $u = J_1J^*_2 = \tilde\alpha \tilde\beta$. But now $\tilde\alpha$ is injective and $\tilde\beta$ surjective. Consider the mapping $w\colon \ S/{S_0}\to X$ defined by $w(\cdot) = J^{-1}_1\tilde\alpha(\cdot)$. Recalling \eqref{R0} we find $\|w\| \le (1+\vp) \|\tilde\alpha\| \le (1+\vp)\|\alpha\|$. Moreover, for any $x$ in $X$ since $J^*_2\colon \ \ell^N_1\to X$ is onto, there is $y$ in $\ell^N_1$ such that $J^*_2y=x$, so that $w$ is invertible  and we have $w^{-1}x = \tilde\beta y$. Recalling \eqref{R00} again this gives us $\|w^{-1}\| \le \|\tilde\beta\| (1+\vp) \le \|\beta\| (1+\vp)$. Thus we conclude $\|w\|\|w^{-1}\| \le (1+\vp)^2 |\alpha\| \|\beta\|\le C(1+\vp)^2$. This shows that  $X$ is $C(1+\vp)^2$-isomorphic to
 $S/{S_0}$. Letting $\vp\to 0$, we find that $X$ is $C$-isomorphic to an ultraproduct of spaces
 of the form $S/{S_0}$, i.e. of subquotients of $\ell_p^n({\cl X})$ ($n\ge1$).
\end{proof}

\subsection{}\label{comsec3.7}
Actually, Hernandez considers more generally 
(instead of an identity operator) bounded linear maps $u\colon \ Z\to Y$ between Banach spaces that factor through a space $X$ that is a subquotient   of an ultraproduct of the spaces described   in Theorem \ref{comthm3.1bis} above and defines
\[
\gamma_{SQ{\cl B}}(u) = \inf\{\|u_1\| \|u_2\|\}
\]
where the infimum runs over all $u_2\colon\ Z\to X$ and $u_1\colon \ X\to Y$ 
such that $u=u_1u_2$ and over all such spaces $X$. In this broader framework, he shows that
\[
\gamma_{SQ{\cl B}}(u) = \sup\{\|T\otimes u\colon \ L_p(Z)\to L_p(Y)\|\}
\]
where the supremum runs over all $T\colon \ L_p\to L_p$ (or $T\colon \ \ell^n_p\to \ell^n_p$) such that $\|T_X\|\le 1$ for any $X$ in ${\cl B}$.

In this framework, the ``duality" we are interested in becomes more transparent: given a class ${\cl C}$ of operators, say, from  $L_p(\mu)$  to $L_p(\mu')$,
we may introduce the dual class ${\cl C}^\dagger$ formed of all operators $u\colon \ Z\to Y$ between Banach spaces such
$$ \forall T\in {\cl C}\quad \|T\otimes u\colon \ L_p(Z)\to L_p(Y)\|\le 1.$$
Similarly, given a class ${\cl B}$  of operators $u$, we may introduce the dual class ${\cl B}^\dagger$ formed of all $T$ such that
$$ \forall u\in {\cl B}\quad \|T\otimes u\colon \ L_p(Z)\to L_p(Y)\|\le 1.$$
Here again, one can replace 1 by a fixed constant to treat the ``isomorphic" (as opposed to isometric) variant of this.
\subsection{}\label{comsec3.9}

More generally, let $1\le p\le \infty$. We recall for future reference that if ${\cl B}$ is a (non-void) class of Banach spaces that is stable under $\ell_p$-sums (i.e. $X,Y\in {\cl B}\Rightarrow X\oplus_p Y\in {\cl B}$) then the ``norm'' of factorization through a space belonging to ${\cl B}$ is indeed a norm. More precisely, let us define for any operator $u\colon\ E\to F$ acting between Banach spaces the norm
\[
 \gamma_{\cl B}(u) = \inf\{\|u_1\| \|u_2\|\}
\]
where the infimum runs over all $X$ in ${\cl B}$ and all factorizations $E \overset{\sst u_2}{\longrightarrow} X \overset{\sst u_1}{\longrightarrow} F$. Let $\Gamma_{\cl B}(E,F)$ denote the space of those $u$ that admit such a factorization. Then $\Gamma_{\cl B}(E,F)$ is a vector space and $\gamma_{\cl B}$ is a norm on it. See \cite{Kw2, Pie} for details.

\subsection{}\label{comsec3.10}

Note that $\|u\| \le \gamma_{\cl B}(u)$ and equality holds if $u$ has rank 1. Therefore we have for all $u\colon \ E\to F$ 
\[
 \|u\|\le \gamma_{\cl B}(u) \le N(u).
\]

\section{Complex interpolation of families of Banach spaces}\label{comsec4}

Let $D = \{z\in {\bb C}\mid |z|<1\}$. Consider a measurable family $\{\|~~\|_z\}$ of norms on ${\bb C}^n$ indexed by $z\in \partial D$. By measurable, we mean that $z\to \|x\|_z$ is measurable for any $x$ in ${\bb C}^n$. We will need to assume   that there are positive functions $k_1>0$ and $k_2>0$ such that their logarithms are in $L_1(m)$ on $\partial D$  equipped with its normalized Lebesgue measure $m$, and satisfying
\begin{equation}\label{compat}
{\forall z\in \partial D~\forall x\in {\bb C}^n}\quad
k_1(z)\|x\| \le \|x\|_z \le k_2(z)\|x\|,
\end{equation}
where $\|x\|$ denotes   the Euclidean norm of $x$ (here any fixed norm would do just as well),
sometimes denoted also below by  $\|x\|_{\ell_2^n}$.

Let $X(z) = ({\bb C}^n, \|~~\|_z)$. Following \cite{CCRS3} we say that $\{X(z)\mid z\in \partial D\}$ is a compatible family of Banach spaces. When this holds for constant functions $k_1>0$ and $k_2>0$, we will say that
the family is strongly compatible.

 By \cite{CCRS3} (see also \cite{CCRS1, CCRS2, CS, Sem}), the complex interpolation method can be extended to this setting and produces a family $\{X(z)\mid z\in D\}$ of which the original family appears as boundary values in a suitable sense (see also \cite{Roc1,Roc2,Roc3,Her} for further developments).

We need to recall how the space $X(0)$ is defined (cf.\ \cite{CCRS1,CCRS2}). Let $W_j$ be an outer function on $D$ admitting $k_j$ as its boundary values (recall that $W_j = \exp(u_j+i\tilde u_j)$ where $u_j$ is the Poisson integral of ${\log} k_j$). Let $0<p\le \infty$. We define the space $H^p_\#$ as formed of all analytic  functions $f$ such that $W_1f$ is in the classical ${\bb C}^n$-valued Hardy space $H^p(D,{\bb C}^n)$ with boundary values such that $z\to \|f(z)\|_z$ is in $L_p(m)$. We set $\|f\|_{H^p_\#} = (\int\limits_{\partial D} \|f(z)\|^p_z dm(z))^{1/p}$, with the usual convention for $p=\infty$. Then
\begin{equation}\label{comeq4.22}
\forall x\in {\bb C}^n\qquad\qquad \|x\|_{X(0)}  = \inf\{\|f\|_{H^\infty_\#}\}
~~~~~~~~~~~~~~~~~~~~~~~~~~~~~~~~~~
\end{equation}
where the infimum runs over all $f$ in $H^\infty_\#$ such that $f(0) = x$. 
Although it appears so at first glance, this definition does not depend on the choice of $k_1$   as long as we choose $k_1>0$ such that ${\log} k_1$ is in $L_1$
(this can be verified using \eqref{comeq4.23++} below).

An equivalent definition which makes this clear is as follows
(see below for more  on $N^+$):
\begin{equation}\label{comeq4.23}
\|x\|_{X(0)} = \inf\{\text{ess}\sup\|f(z)\|_z\}
\end{equation}
where the infimum runs over all functions
$f$ in the class ${N^+}(D,{\bb C}^n)$   such that $f(0) = x$. The latter class ${N^+}(D,{\bb C}^n)$ is defined
as that of all $f$ of the form $f(z) = \sum^k_1 f_jx_j$ with $x_j\in {\bb C}^n$, $f_j\in N^+$ $(k\ge 1)$ .

We now give some useful background on  the Nevanlinna class $N^+$ (see \cite[p. 25]{Du} and 
\cite[\S 4]{RR}). 

The class $N^+$ is the set of all analytic functions $f\colon \ D\to {\bb C}$ such that $\sup\limits_{r<1} \int {\log}^+|f(rz)| \ dm(z)<\infty$ (i.e.\ $f\in N$) and moreover such that, when $r\to 1$
\[
\int {\log}^+|f(rz)|\ dm(z)\to \int {\log}^+|f(z)| \ dm(z),
\]
where $(f(z))_{z\in\partial D}$ denote the non-tangential limit of $f$ (which exists a.e.\ as soon as $f\in N$).
Let $f_r(z)=f(rz)$. Equivalently, assuming $f\not \equiv 0$,   $f\in N^+$ iff the family $\{{\log}^+ |f_r|\mid 0<r<1\}$
is uniformly integrable on the unit circle (see \cite[p.65-66]{RR} for this specific fact).
In that case, we have
$$\forall \xi \in D\qquad\log | f(\xi)| \le \int_{\partial D} \log | f(z)|  d\mu^\xi(z)$$
where $\mu^\xi $ denotes the harmonic probability mesure on $\partial D$ for the point 
$\xi\in D$. Equivalently   $\mu^\xi$ is the probability  on $\partial D$ admitting the Poisson kernel of $\xi$
as density relative to Lebesgue measure. More explicitly
 if 
\[
 \xi = re^{it_0}\quad z=e^{it} \qquad (0<r<\infty,\  t_0,t\in[0,2\pi]) 
\]
then we have
\begin{equation}\label{poisson}
 d\mu^\xi(z)=\frac{1-|\xi|^2}{|z-\xi|^2}dm(z)=\frac{1-r^2}{1-2r\cos(t-t_0)+r^2} \ \frac {dt}{2\pi}.
\end{equation}

More generally,  for any norm on
$\bb C^n$ (writing the norm as a sup of linear functionals) and any $f$ in ${N^+}(D,{\bb C}^n)$, we have 
\begin{equation}\label{interpol3}\forall \xi \in D\qquad\log \| f(\xi)\| \le \int_{\partial D} \log \| f(z)\|  d\mu^\xi(z).\end{equation}
and hence 
\begin{equation}\label{interpol4}
\forall \xi \in D\qquad  \| f(\xi)\| \le {\rm ess}\sup_{\partial D}   \| f(z)\|   .
\end{equation}
It is known that $N^+$ is a vector space and even an algebra
(see \cite[p.65-66]{RR})  and $H^p\subset N^+ \subset N$ for any $0< p \le \infty$. Clearly, these inclusions remain valid for the corresponding spaces of ${\bb C}^n$-valued functions,
denoted by $H^p(D, {\bb C}^n),   N^+(D, {\bb C}^n), N(D, {\bb C}^n)$. In particular, we have 
(note that $W_1$ and its inverse are in $N^+$)  $H^\infty_\#\subset N^+(D, {\bb C}^n)$.

We will need several basic properties of the interpolation spaces.

\subsection{}\label{comsec4.0} 

Assume that  the   family $\{X(z)\mid z\in \partial D\}$ takes only two values, i.e. we have a pair $(X_0,X_1)$ and a measurable partition $\partial D=\partial_0 \cup \partial_1$ such that
$X(z)=X_0$ when $z\in \partial_0$ and $X(z)=X_1$ when $z\in \partial_1$.
Let  $\theta$ denote the normalized Haar measure of  $\partial_1$. Then   we recover the classical
(Calder\'on-Lions)  complex interpolation space $(X_0,X_1)_\theta$. Indeed, it is proved in
\cite[Cor. 5.1]{CCRS3} that the family $\{X(z)\mid z\in   D\}$ that extends and interpolates
the data on $\partial D$, is simply the one given by setting
$$X(z)=(X_0,X_1)_{\theta(z)}$$
where, for any $z$ inside $D$,  $\theta(z)$ is defined as 
the harmonic extension inside $D$ of the indicator function of $\partial_1$.
In particular, since $\theta(0)=\theta$, we have (isometrically)
 $$X(0)=(X_0,X_1)_{\theta}.$$

\subsection{}\label{comsec4.1}
For any analytic function $f$ in $H^\infty_\#$, or  in $N^+(D, {\bb C}^n)$,   and any $\xi\in D$ we have
$$\log \| f(\xi)\|_{X(\xi)} \le \int_{\partial D} \log \| f(z)\|_{X(z)} d\mu^\xi(z),$$
and the function $\xi\mapsto \log \| f(\xi)\|_{X(\xi)}$  is subharmonic in $D$. 
 In particular
\begin{equation}\label{interpol}
\log \| f(0)\|_{X(0)} \le \int_{\partial D} \log \| f(z)\|_{X(z)} dm(z).\end{equation}
This shows that for any $x\in {\bb C}^n$
\begin{equation}\label{comeq4.23++}
 \|x\|_{X(0)} = \inf\left\{\exp\left(\int_{\partial D} \log \| f(z)\|_{X(z)} dm(z)\right)\right\}
\end{equation}
where the infimum runs over all $f$ in $H^\infty_\#$ (or  in $N^+(D, {\bb C}^n)$) with $f(0)=x$.

One can deduce from this that if we replace $H^\infty_\#$ by $H^p_\#$ in \eqref{comeq4.22}
 ($0<p\le \infty$) the infimum remains the same, i.e.\ we have
\begin{equation}
\|x\|_{X(0)}  = \inf\{\|f\|_{H^p_\#}\mid f\in H^p_\# \quad f(0) = x\}.\tag*{$\forall x\in {\bb C}^n$}
\end{equation}
In particular, using say $p=2$, one can describe the dual $X(0)^*$ as the interpolation space associated to the dual family $\{X(z)^*\mid z\in \partial D\}$. More precisely, if we set $Y(z) = X(z)^*$ (note that this is still compatible), then we have
\[
X(0)^* = Y(0)\quad \text{isometrically.}
\]

\subsection{}\label{comsec4.1+} The fundamental interpolation principle for families
(inspired by Elias Stein's classical result) takes the following form:
Let $z\mapsto T(z)$ be a function in $N^+$ with values in the normed space $B( {\bb C}^n   )$ of
linear mappings on   $({\bb C}^n,\|\cdot\|)$.  Then
\begin{equation}
\label{interpol2}\log \|T(0)\|_{B(X(0))}\le \int \log \|T(z)\|_{B(X(z))} dm(z).\end{equation}
Indeed, for any $f \in N^+(D, {\bb C}^n)$, the function $z\mapsto T(z)f(z)$
is in $N^+(D, {\bb C}^n)$, and hence if $f(0)=x$ we deduce from \eqref{interpol}
$$\log \|T(0)x\|_{ X(0)}\le \int_{\partial D} \log \| T(z)f(z)\|_{X(z)} dm(z)\le \int_{\partial D} \log \| T(z) \|_{X(z)} dm(z) +\log \|f\|_{H^\infty_\#} ,$$ from which \eqref{interpol2} is immediate.

\subsection{}\label{comsec4.2}
Let $z\to\gamma(z)>0$ be a positive function on $\partial D$ such that ${\log} \gamma$ is in $L_1(m)$. Then there is an outer function $F$ on $D$ in the Nevanlinna class $N^+$ admitting $\gamma$  as (nontangential) boundary values. Let $\{X(z)\mid z\in \partial D\}$ be a compatible family, with norms $\|~~\|_z$. Let $Y(z)$ be ${\bb C}^n$ equipped with the norm
\[
x\to |F(z)|\|x\|_z.
\]
Then $Y(0)$ is isometric to $X(0)$, and actually:
\begin{equation}\label{comeq4.3}
\forall x\in {\bb C}^n\qquad \qquad \|x\|_{Y(0)} = |F(0)|\|x\|_{X(0)}.~~~~~~~~~~~~~~~~~~~~~~~~~~~~~~~~~~~
\end{equation}
Indeed, note that $\{Y(z)\}$ is a compatible family and if $W_1,W_2$ are associated to $k_1,k_2$ as above, then $W_1F$ and $W_2F$ play the same r\^ole for the family $\{Y(z)\}$. Let $H^\infty_\#\{X\}$ and $H^\infty_\#\{Y\}$ denote the corresponding spaces $H^\infty_\#$ as defined above. It is then a trivial exercise to check that
\[
f\in H^\infty_\#\{X\} \Leftrightarrow fF^{-1}\in H^\infty_\#\{Y\}.
\]
By \eqref{comeq4.22}, this clearly implies \eqref{comeq4.3}.$\hfill\square$

\subsection{}\label{comsec4.3}
In particular, this shows that we can always reduce to the case when $k_1\equiv 1$ (resp.\ $k_2\equiv 1$) simply by choosing $F=W^{-1}_1$ (resp.\ $F=W^{-1}_2$) in \ref{comsec4.2}. 

\subsection{}
We will say that two compatible families $\{X(z)\mid z\in\partial D\}$ and $\{Y(z)\mid z\in\partial D\}$ of $n$-dimensional spaces are ``equivalent'' if there is an {\it invertible} analytic function $F\colon \ D\to M_n$ in the Nevanlinna class $N^+$ (i.e.\ with each entry in $N^+$) such that
\begin{equation}
\|F(z)x\|_{Y(z)} = \|x\|_{X(z)}.\tag*{$\forall z\in\partial D$}
\end{equation}
Here, ``invertible'' means that $F(z)$ is invertible for all $z$ in $D$, and that $z\to F(z)^{-1}$ is also in $N^+$. 
Recall that $N^+$ is an algebra. Then, using \eqref{comeq4.23},   we have clearly:\ $X(0)\simeq Y(0)$ isometrically and $F(0)\colon \ X(0)\to Y(0)$ is an isometric isomorphism.

The following lemma shows that on the subset $C\subset\partial D$ where $X(z)$ is Hilbertian, we may assume after passing to an equivalent family that $X(z)=\ell^n_2$ for (almost) all $z$ in $C$.

\begin{lem}\label{comlem4.1}
Any compatible family $\{X(z)\mid z\in\partial D\}$ of $n$-dimensional Banach spaces is equivalent to a family $\{Y(z)\mid z\in\partial D\}$ such that
\[
 Y(z) = \ell^n_2
\]
for (almost) all $z$ such that $X(z)$ is Hilbertian.
\end{lem}
 
\begin{proof}
By \ref{comsec4.3}, we may assume that the compatibility condition has the form $k_1\|x\| \le \|x\|_{X(z)}\le \|x\|$ $\forall x\in {\bb C}^n$, with ${\log} k_1$ in $L_1$.

Let $C\subset\partial D$ be the set of $z$'s such that $X(z)$ is (isometrically) Hilbertian. Then there is clearly a measurable function $\varphi\colon \ C\to (M_n)_+$ such that
\begin{equation}
 \|x\|^2_{X(z)} = \langle\varphi(z)x,x\rangle\tag*{$\forall x\in {\bb C}^n$}
\end{equation}
where $\langle \cdot,\cdot\rangle$ is the usual scalar product in ${\bb C}^n$. Let us extend $\varphi$ arbitrarily outside $C$, say by setting $\varphi(z) = I$ for all $z$ outside $C$. By the compatibility assumption, we have clearly $k_1(\cdot)I \le \varphi(\cdot)\le I$. By the classical matricial Szeg\"o theorem (cf.\ e.g.\ \cite{Hel}) there is a bounded outer function $F\colon \ D\to M_n$ such that $|F(z)|^2 = F(z)^* F(z) = \varphi(z)$ on $\partial D$. We now define $Y(z)$ by setting
\begin{equation}
 \|x\|_{Y(z)} = \|F(z)^{-1}x\|_{X(z)}.\tag*{$\forall x\in {\bb C}^n$}
\end{equation}
 Note that for any $z$ in $C$ we have
\[
 \|x\|^2_{Y(z)} = \langle \varphi(z) F(z)^{-1}x, F(z)^{-1}x\rangle = \|x\|^2_{\ell^n_2}.
\]
This proves the lemma.
\end{proof}

\subsection{}\label{comsec4.4}  Let $\{X(z)\mid z\in\partial D\}$ be a 
compatible family. Let $\|\cdot\|_{\xi}$ denote the norm in ${X(\xi)}$
($\xi\in \bar D$).
For any 
${\bb C}^n$-valued analytic function $f$ on $D$,  
the function $\xi\mapsto \log \| f(\xi)\|_{\xi}$  is subharmonic in $D$. 
A fortiori, by Jensen, the function $\xi\mapsto   \| f(\xi)\|_{ \xi}$  is subharmonic in $D$. 
Any norm valued function on $D$ with this property is called subharmonic in \cite{CS}.
Now suppose for any $x\in {\bb C}^n$ the function $x\mapsto \|x\|_{z}$ is continuous and bounded
on $\partial D$. Then for any $z\in \partial D$ and any $x\in {\bb C}^n$
$$\limsup_{\xi\to z} \|x\|_\xi \le \|x\|_{z}.$$
It turns out that the norm valued function $\|\cdot\|_\xi$ is the largest
among all the subharmonic ones  on $D$ satisfying this. Coifman and Semmes \cite{CS,Sem}
used this characterization to develop an interpolation method
on domains on ${\bb C}^d$ with $d>1$. 
Independently Slodkowski \cite{Slo1,Slo2} introduced several
new interpolation methods. The possible consequences of these ideas for
the geometry of Banach spaces have yet to be investigated.

\section{$\pmb{\theta}$-Hilbertian spaces}\label{comsec4bis}

\begin{defn}\label{thetahilb}
\begin{itemize}
\item We will say that a finite dimensional Banach space $X$ is $\theta$-Euclidean if $X$ is isometric to the complex interpolation space $X(0)$ associated to a compatible family $\{X(z)\mid z\in \partial D\}$ of Banach spaces such that $X(z)$ is Hilbertian for all $z$ in a subset of $\partial D$ of (normalized) Haar measure $\ge \theta$. 

\item If this holds simply for all $z$ in $J_\theta = \{e^{2\pi it}\mid 0<t<\theta\}$ then we say that $X$ is arcwise $\theta$-Euclidean.

\item We will say that a Banach space is $\theta$-Hilbertian (resp. arcwise $\theta$-Hilbertian)
if it is isometric to an ultraproduct of a family of
$\theta$-Euclidean (resp. arcwise $\theta$-Euclidean) finite dimensional  spaces.
 \end{itemize}\end{defn}

The preceding terminology is different from the one in our previous paper \cite{P1}. There we called $\theta$-Hilbertian the Banach spaces that can be written as $(X_0,X_1)_\theta$ for some interpolation pair of Banach spaces with $X_1$ Hilbertian.
We prefer to change this: we will call these spaces strictly
$\theta$-Hilbertian.

We suspect that there are $\theta$-Hilbertian spaces that are not strictly
$\theta$-Hilbertian (perhaps even not quotients of subspaces of ultraproducts
of  strictly
$\theta$-Hilbertian spaces), but we have no example yet. This amounts to show that in
\eqref{comeq4.0} below we cannot  restrict $X$ to be strictly
$\theta$-Hilbertian.

Let $Z,Y$ be Banach spaces. Recall that we denote by $\Gamma_H(Z,Y)$ the set of operators $u\colon \ Z\to Y$ that factorize through a Hilbert space, i.e.\ there are bounded operators $u_1\colon \ H\to Y$, $u_2\colon \ Z\to H$ such that $u=u_1u_2$. We denote
\[
\gamma_H(u) = \inf\{\|u_1\| \ \|u_2\|\}
\]
where the infimum runs over all such factorizations.

Similarly, we denote by $\Gamma_{\theta H}(Z,Y)$ (resp.  $\Gamma_{  \widetilde{ \theta H }(Z,Y)}$) the set of $T$'s that factor through a $\theta$-Hilbertian space (resp. arcwise $\theta$-Hilbertian) and we denote
\[
\gamma_{\theta H}(u) = \inf\{\|u_1\| \|u_2\|\}\quad ({\rm resp.}\   \gamma_{   \widetilde{ \theta H }}(u)= \inf\{\|u_1\| \|u_2\|)
\]
where the infimum runs over all factorizations $u=u_1u_2$ with $u_1\colon \ X\to Y$, $u_2\colon \ Z\to X$ and with $X$  $\theta$-Hilbertian (resp. arcwise $\theta$-Hilbertian). Note that, since  ``$\theta$-Hilbertian" (resp. arcwise $\theta$-Hilbertian) is stable under $\ell_2$-sums, this is a norm (see \ref{comsec3.9}).

%

\begin{lem}\label{comlem5.10}
 Fix $n\ge 1$ and $0<\theta<1$. Let $E,F$ be $n$-dimensional Banach spaces. Then for any linear map $u\colon \ E\to F$ we have
\[
 \gamma_{\theta H}(u) \le \gamma_{   \widetilde{ \theta H }}(u)\le  \|u\|_{(B(E,F), \Gamma_H(E,F))_\theta}.
\]
More precisely, if $\|u\|_{(B(E,F), \Gamma_H(E,F))_\theta} \le 1$ and if $u$ is a linear isomorphism, then $u$ admits a factorization $u=u_1u_2$ with $\|u_2\colon \ E\to X\|\le 1$ and $\|u_1\colon \ X\to F\|\le 1$ where $X=X(0)$ for a compatible family $\{X(z)\mid z\in\partial D\}$ such that $X(z)\simeq \ell^n_2$ $\forall z\in J_\theta$ and $X(z)\simeq F$ $\forall z\notin J_\theta$. (We also have $X=Y(0)$ for a family $\{Y(z)\}$ such that $Y(z)\simeq \ell^n_2$ $\forall z\in J_\theta$ and $Y(z)\simeq E$ $\forall z\notin J_\theta$.)
\end{lem}

\begin{proof}
Since linear isomorphisms are dense in $B(E,F)$, it clearly suffices to prove the second (``more precise'') assertion.   Assume $\|u\|_{(B(E,F), \Gamma_H(E,F))_\theta}\le 1$. 
As is well known, by conformal equivalence, we may use the unit disc instead of the strip to define
complex interpolation. We have then $u = u(0)$ for some function $z\mapsto u(z)$ in $H^\infty(D,B(E,F))$ such that ${\rm ess}\sup_{z\not\in J_\theta} \|u(z)\|_{B(E,F)}\le 1$ and ${\rm ess}\sup_{z\in J_\theta} \|u(z)\|_{\Gamma_H(E,F)}\le 1$. Therefore, (using a measurable selection, see \ref{comsec1.10}) we can find bounded measurable functions $z\to u_2(z)\in B(E,\ell^n_2)$ and $z\to u_1(z)\in B(\ell^n_2, F)$ such that $u(z) =u_1(z)u_2(z)$ and ${\rm ess}\sup_{z\in J_\theta} \|u_1(z)\|\le 1$, ${\rm ess}\sup_{z\in J_\theta} \|u_2(z)\|\le 1$. 

We may
view our operators as $n\times n$ matrices
(with respect to fixed bases) and define
their determinants as those of the corresponding matrices.
Note that since $z \mapsto \det(u(z))$ is analytic, bounded and non identically zero, we must have $  \det(u(z))\not = 0 $ a.e. on $\partial D$ and (by Jensen, see e.g. \cite{Du,Ga})
$\log |\det(u(z))|$ must be in $L_1(\partial D)$. Moreover, the classical formula for the inverse of a matrix shows (since $u$ is bounded on $D$) that there is a constant $c$ such that
$\|   u(z)^{-1}\|\le c |\det(u(z))|^{-1}$ a.e. on $\partial D$. Therefore, if $\varphi(z)=\|   u(z)^{-1}\|^{-1}$, we have
$$\int_D \log \varphi >-\infty$$
and since $\varphi(z)\le \|   u(z)\|$, we obtain
$\log \varphi \in L_1(\partial D)$.  Note that we have
$$\det(u(z))=\det(u_1(z))\det(u_2(z))$$ and hence since $u_1,u_2$ are bounded on $\partial D$, there are constants $c_1,c_2$ such that 
$ |\det(u(z))|\le c_j  |\det(u_j(z))|$ a.e. on $\partial D$ for $j=1,2$. Thus letting
$\varphi_j=\|   u_j(z)^{-1}\|^{-1}$, the same argument
yields
$$\log \varphi_j \in L_1(\partial D).$$

We then define for any $x$ in $E$ 
\begin{align}
 \|x\|_{X(z)} &= \|u(z)x\|_F, \|x\|_{Y(z)} = \|x\|_E\tag*{$\forall z\not\in J_\theta$}\\
\|x\|_{X(z)} &= \|x\|_{Y(z)} = \|u_2(z)x\|_{\ell^n_2}.\tag*{$\forall z\in J_\theta$}
\end{align}
From the preceding observations on $\varphi_1$ and 
$\varphi_2$ it is easy to deduce that  \eqref{compat}
holds with $\log(k_1)$ and $\log(k_2)$ in $L_1(\partial D)$. Therefore $(X(z))$ and $(Y(z))$ are compatible families.
Then $\forall z\in\partial D$
\[
 \|u(z)x\|_F \le \|x\|_{X(z)} \le \|x\|_E.
\]
By consideration of the constant function equal to $x$ we find 
 \[
  \|x\|_{X(0)}\le \|x\|_E.
 \]
Then, for any analytic $f\colon\ D\to E$ in $H^\infty_\#$ with $f(0)=x$ we can write
by \eqref{interpol4}
\[
\|ux\|_F  = \|u(0)f(0)\|_F\le {\text{ess}}\sup_{z\in \partial D}\|u(z)f(z)\|_F  \le {\text{ess}}\sup_{z\in \partial D}\| f(z)\|_{X(z)}   ,
 \]
 thus, taking the infimum over all possible $f$'s, we conclude
 \[
\|ux\|_F\le \|x\|_{X(0)}\le \|x\|_E.
 \]
 
This clearly shows that $u$ factors through $X(0)$ with constant of factorization $\le 1$. Similarly, $u$ factors through $Y(0)$ with factorization constant $\le 1$.
\end{proof}

\begin{thm}\label{comthm5.7}
Let $0<\theta<1$. For any $n\ge 1$ we have   isometric identities
\begin{equation}\label{comeq5.13}
(B(\ell^n_1,\ell^n_\infty), \Gamma_H(\ell^n_1, \ell^n_\infty))_\theta = \Gamma_{\theta H}(\ell^n_1,\ell^n_\infty)= \Gamma_{ \widetilde{\theta H}}(\ell^n_1,\ell^n_\infty).
\end{equation}
\end{thm}

\begin{proof}
Let $u\colon \ \ell^n_1\to\ell^n_\infty$ be such that $\gamma_{\theta H}(u)<1$, i.e.\ there is $X$ $\theta$-Hilbertian and $u_1\colon \ X\to \ell^n_\infty$, $u_2\colon \ \ell^n_1\to X$ with $\|u_1\|<1$, $\|u_2\|<1$ such that $u = u_1u_2$. By compactness, we may as well assume that $X$ is finite dimensional and $\theta$-Euclidean. Let then $h_j = u_2e_j\in X $ and $k_i = u^*_1e_i\in X^*$. We have $u_{ij} = \langle k_i,h_j\rangle$, and 
\[
\sup\|h_j\|_X < 1,\qquad \sup\|k_i\|_{X^*} < 1.
\]
Writing $X = X(0)$ for a compatible family such that $X(z)$ is Hilbertian on a set of measure $\ge \theta$ we find analytic functions $h_i(z)$, $k_j(z)$ in $N^+(D, {\bb C}^n)$ such that
\[
\text{ess~sup}\|h_j(z)\|_z < 1,\qquad \text{ess~sup}\|k_i(z)\|^*_z<1
\]
and $h_j(0) = h_j$, $k_i(0) = k_i$.
This gives us a matrix valued function $F_{ij}(z) = \langle k_i(z), h_j(z)\rangle$ in $N^+$  with associated operator $F(z)\colon\ \ell^n_1\to\ell^n_\infty$ such that $\gamma_H(F)<1$ on a set of measure $\ge\theta$ and $\|F\|\le 1$ otherwise. (Indeed, $\|F\| = \sup\limits_{ij}|F_{ij}|$ $\le \sup\limits_i \|k_i\|^*_z  \sup\limits_j \|h_j\|_z)$. 
By   \ref{comsec4.0}, this shows that
\[ 
\|F(0)\|_{(B(\ell^n_1,\ell^n_\infty), \Gamma_H(\ell^n_1,\ell^n_\infty))_\theta} \le 1,
\]
and since $F(0) = u$, we have proved $\|u\|_{(B(\ell^n_1,\ell^n_\infty), \Gamma_H(\ell^n_1, \ell^n_\infty))_\theta}\le \gamma_{\theta H}(u)$. Conversely, by   Lemma \ref{comlem5.10}, we have 
$ \gamma_{ \widetilde{\theta H}}(u)\le \|u\|_{(B(\ell^n_1,\ell^n_\infty), \Gamma_H(\ell^n_1, \ell^n_\infty))_\theta} $.
Since we have trivially $   \gamma_{ {\theta H}}(u)\le \gamma_{ \widetilde{\theta H}}(u)$, this proves \eqref{comeq5.13}.
\end{proof}

To supplement the preceding result, the next proposition gives a simple description of the norm in the dual of the space $(B(\ell^n_1,\ell^n_\infty)$, $\Gamma_H(\ell^n_1,\ell^n_\infty))_\theta$. By the duality property of complex interpolation (cf.\ \cite[p.~98]{BL}) the latter space coincides isometrically with $(B(\ell^n_1,\ell^n_\infty)^*, \Gamma_H(\ell^n_1,\ell^n_\infty)^*)_\theta$. We will use the classical duality between linear maps $u\colon \ \ell^n_1\to\ell^n_\infty$ and linear maps $v\colon \ \ell^n_\infty \to \ell^n_1$ defined by
\[
 \langle v,u\rangle = tr(uv) = tr(vu).
\]
With this duality, it is classical (see \S \ref{comsec1} above) that we have isometric identities
\[
 B(\ell^n_1,\ell^n_\infty)^* = N(\ell^n_\infty,\ell^n_1) = \ell^{n^2}_1
\]
and
\[
 \Gamma_H(\ell^n_1,\ell^n_\infty)^* = \Gamma^*_H(\ell^n_\infty,\ell^n_1)
\]
where $\Gamma^*_H(\ell^n_\infty,\ell^n_1)$ denotes
$B  (\ell^n_\infty,\ell^n_1) $ equipped with the norm
$\gamma^*_H(.)$ defined by
  (see \ref{comsec1.9})
\[
 {\gamma^*_H(v)} = \inf\left\{\left(\sum|\lambda_i|^2\right)^{1/2} \|a\|_{B(\ell^n_2)} \left(\sum|\mu_j|^2\right)^{1/2}\right\}
\]
with the infimum running over all factorizations of $v$ of the form
\begin{equation}\label{comeq5.16}
 \forall i,j=1,\ldots, n\qquad\qquad v_{ij} = \lambda_ia_{ij}\mu_j.~~~~~~~~~~~~~~~~~~
\end{equation}

\begin{pro}\label{compro5.9}
Fix $n\ge 1$ and $0<\theta<1$. Consider $v\colon \ \ell^n_\infty\to \ell^n_1$. Then $v$ belongs to the (closed) unit ball of $(N(\ell^n_\infty,\ell^n_1), \Gamma^*_H(\ell^n_\infty,\ell^n_1))_\theta$ iff there are $\lambda,\mu$ in the unit ball of $\ell^n_2$ and $a$ with $\|a\|_{(B_r(\ell^n_2), B(\ell^n_2))_\theta}\le 1$ such that $v_{ij} = \lambda_ia_{ij}\mu_j$ for all $i,j=1,\ldots, n$. 
\end{pro}

\begin{proof}
We first prove the ``if part.'' Fix $\lambda,\mu$ in the unit ball of $\ell^n_2$. By Proposition \ref{compro7.2}, the linear mapping $a\to [\lambda_ia_{ij}\mu_j]$ has norm $\le 1$ simultaneously from $B_r(\ell^n_2)$ to $N(\ell^n_\infty,\ell^n_1)$ and from $B(\ell^n_2)$ to $\Gamma^*_H(\ell^n_\infty, \ell^n_1)$. Therefore, by interpolation it also has norm $\le 1$ from $(B_r(\ell^n_2), B(\ell^n_2))_\theta$ to $(N(\ell^n_\infty,\ell^n_1), \Gamma^*_H(\ell^n_\infty,\ell^n_1))_\theta$. This proves the ``if part.''

 Conversely, assume that $v$ is in the open unit ball of $(N(\ell^n_\infty,\ell^n_1), \Gamma^*_H(\ell^n_\infty, \ell^n_1))_\theta$. \\
Let $S = \{z\in {\bb C}\mid 0 < \text{Re}(z) < 1\}$
be the classical ``unit strip", and let
$\partial_j= \{z\in {\bb C}\mid  \text{Re}(z) =j\}$, $j=0,1$. By the definition of complex interpolation,  by Proposition \ref{compro7.2} and a routine measurable selection argument (see \ref{comsec1.10}),
 we can write $v_{ij} = v_{ij}(\theta)$ where $v_{ij}(\cdot)$ is a bounded analytic matrix valued function defined on the strip $S$  admitting a factorization of the form $v_{ij}(z) = \lambda_i(z) a_{ij}(z) \mu_j(z)$ where
\[
 {\rm ess}\sup_{z\in\partial_0\cup \partial_1} \max\left\{\sum|\lambda_j(z)|^2,\sum |\mu_j(z)|^2\right\}<1
\]
together with
\begin{equation}\label{comeq5.17}
 {\rm ess}\sup_{z\in\partial_0} \|(a_{ij}(z))\|_{B_r(\ell^n_2)} < 1
\end{equation}
and
\begin{equation}\label{comeq5.18}
 {\rm ess}\sup_{z\in\partial_1} \|(a_{ij}(z))\|_{B(\ell^n_2)} < 1.
\end{equation}
Replacing $\lambda_i(z)$ and $\mu_j(z)$ respectively by $\vp + |\lambda_i(z)|$ and $\vp+|\mu_j(z)|$ with $\vp>0$ chosen small enough we may as well assume that $|\lambda_i(\cdot)|$ and $|\mu_j(\cdot)|$ are all bounded below. Then there are bounded outer functions $F_i,G_j$ on $S$ such that
\[
 |F_i(z)| = |\lambda_i(z)|\quad \text{and}\quad |G_j(z)| = |\mu_j(z)|.
\]
We may without loss of generality (since the bounds \eqref{comeq5.17}, \eqref{comeq5.18} are preserved in this change) assume that $\lambda_i(z) = F_i(z)$ and $G_j(z) = \mu_j(z)$. But then
\[
 a_{ij}(z) = F_i(z)^{-1} v_{ij}(z) G_j(z)^{-1}
\]
must be analytic on $S$ and hence by \eqref{comeq5.17} and \eqref{comeq5.18} we have
\[
 \|(a_{ij}(\theta))\|_{(B_r(\ell^n_2), B(\ell^n_2))_\theta} \le 1.
\]
Thus we obtain
\[
 v_{ij} = \lambda_ia_{ij}\mu_j
\]
with $\lambda_i = F_i(\theta), \mu_j= G_j(\theta)$, $a_{ij} = a_{ij}(\theta)$.

Note that $\lambda,\mu$ are in the unit ball of $\ell^n_2$ by the maximum principle applied to $(F_i)$ and $(G_j)$. This proves the ``only if part'' except that we replaced the closed unit ball by the open one. An elementary compactness argument completes the proof.
\end{proof}

The main result of this section is perhaps the following.

\begin{thm}\label{comthm5.1}
Let ${\cl H}(\theta,n)$ be the set of $n$-dimensional arcwise $\theta$-Hilbertian Banach spaces. Let $0<\theta<1$. Consider the complex interpolation space $B(\theta,n) = (B_r(\ell^n_2), B(\ell^n_2))_\theta$. Then for any $T$ in $B(\ell^n_2)$ we have
\begin{equation}\label{comeq4.0}
\|T\|_{B(\theta,n)} = \sup_{X\in {\cl H}(\theta,n)} \|T_X\|_{B(\ell^n_2(X))}. 
\end{equation}
Moreover, we will show that equality still holds if we restrict the supremum to those $X$ that can be written $X = X(0)$ for a compatible family such that $X(z) \simeq \ell^n_2$ $\forall z\in J_\theta$ and $X(z)\simeq \ell^n_\infty$ $\forall z\notin J_\theta$, where $\simeq$ means here isometric. 
\end{thm}

\begin{proof}
On one hand, we have for any Banach space $X$
\[
\|T_X\| \le \|T\|_{B_r(\ell^n_2)}
\]
and on the other hand if $X$ is Hilbertian
\[
\|T_X\|\le \|T\|_{B(\ell^n_2)}.
\]
Therefore, if 
$X$ is arcwise $\theta$-Euclidean (resp. $X\in {\cl H}(\theta,n)$), the inequality $\|T_X\|_{B(\ell^n_2(X))} \le \|T\|_{B(\theta,n)}$ is a direct consequence of the interpolation principle \ref{comsec4.1+} (resp. and Remark \ref{ultr}). We skip the easy details.

To prove the converse, we will use duality. We claim that if $\varphi$ in $B(\ell^n_2)$ is in the open unit ball of $B(\theta,n)^*$  and is   invertible, then there is a  space $X$ in ${\cl H}(\theta,n)$ such that
\begin{equation}\label{comeq4.1}
|\varphi(T)| \le \|T_X\|_{B(\ell^n_2(X))} 
\end{equation}
for any $T$ in $B(\ell^n_2)$. 

This obviously implies that the supremum of $|\varphi(T)|$ over such $\varphi$'s is $\le \sup\limits_{X\in {\cl H}(\theta,n)} \|T_X\|$ and hence, by Hahn--Banach and the density of invertible matrices,  equality follows in \eqref{comeq4.0}. 

We now proceed to prove this claim. Let $\varphi$ be as above. Note that (cf. \cite[p.~98]{BL})
\begin{equation}\label{comeq4.2}
B(\theta,n)^* = (B_r(\ell^n_2)^*, B(\ell^n_2)^*)_\theta.
\end{equation}
The norm in $B(\ell^n_2)^*$ is of course the trace class norm. 

Since $\|\varphi\|_{B(\theta,n)^*} < 1$, there is a bounded analytic function $\varphi\colon \ S\to M_n$ such that $\varphi(0) = \varphi$ and
\[
 \max\{{\rm ess}\sup_{z\in\partial_0} \|\varphi(z)\|_{B_r(\ell^n_2)^*}, {\rm ess}\sup_{z\in\partial_1} \|\varphi(z)\|_{B(\ell^n_2)^*}\}<1.
\]
By Proposition \ref{compro7.3} and a routine measurable selection argument (see \ref{comsec1.10}), we can find measurable functions $\lambda(z)$ and $\mu(z)$ on $\partial S$ such that
\begin{equation}\label{comeq5.20}
 {\rm ess}\sup_{z\in\partial S} \max\left\{\sum|\lambda_i(z)|^2, \sum|\mu_j(z)|^2\right\}<1
\end{equation}
and a measurable function $v\colon \ z\to B(\ell^n_1, \ell^n_\infty)$ such that
\begin{equation}\label{comeq5.21}
 \max\{{\rm ess}\sup_{z\in\partial_0}\|v(z)\|_{B(\ell^n_1,\ell^n_\infty)}, {\rm ess}\sup_{z\in\partial_1}\|v(z)\|_{\Gamma_H (\ell^n_1,\ell^n_\infty)}\}<1,
\end{equation}
satisfying
\begin{equation}\label{comeq5.22}
 \varphi_{ij}(z) = \lambda_i(z) v_{ij}(z) \mu_j(z).
\end{equation}
We may replace $(\lambda_i)$ and $(\mu_j)$ by $(|\lambda_i|+\vp)$ and $(|\mu_j|+\vp)$ where $\vp>0$ is chosen small enough so that \eqref{comeq5.20} still holds. If we modify $v_{ij}(z)$ so that \eqref{comeq5.22} still holds (i.e.\ replace $v_{ij}(z)$ by $\lambda_i(|\lambda_i|+\vp)^{-1} v_{ij}\mu_j(|\mu_j|+\vp)^{-1}$), then \eqref{comeq5.21} is preserved.

But now there are bounded outer functions $(F_i)$ and $(G_j)$ on $S$ such that $|F_i| = \lambda_i$ and $|G_j| = \mu_j$ on $\partial S$. Therefore if we define
\[
 V_{ij}(z) = F_i(z)^{-1} \varphi_{ij}(z) G_j(z)^{-1}
\]
then $z\to V(z)$ is analytic on $S$ and satisfies
\[
 \max\{{\rm ess}\sup_{z\in\partial_0} \|V(z)\|_{B(\ell^n_1,\ell^n_\infty)}, {\rm ess}\sup_{z\in\partial_1}\|V(z)\|_{\Gamma_H (\ell^n_1,\ell^n_\infty)} \}< 1.
\]
Therefore, on the one hand
\begin{equation}\label{V(0)}
 \|V(0)\|_{(B(\ell^n_1,\ell^n_\infty), \Gamma_H(\ell^n_1,\ell^n_\infty))_\theta} < 1,
\end{equation}
 on the other hand by the maximum principle 
\[
 \max\left\{\sum|F_i(0)|^2, \sum|G_j(0)|^2\right\}<1,
\]
and finally $\varphi_{ij} = \varphi_{ij}(0) = F_i(0) V_{ij}(0)G_j(0)$. Since $\varphi$ is invertible $V(0)$ is also invertible. By \eqref{V(0)} and Theorem \ref{comthm5.7}, there is a space $X$ in ${\cl H}(\theta,n)$ and elements $x_i\in X^*$, $y_j\in X$ with $\|x_i\|<1$, $\|y_j\|<1$ such that
\[
 V_{ij}(0) = \langle x_i,y_j\rangle.
\]
Let $x'_i= F_i(0)x_i$ and $y'_j = G_j(0)y_j$. Then we find
\[
 \varphi(T) = \sum \varphi_{ij}T_{ij} = \sum\nolimits_i \left\langle x'_i, \sum\nolimits_j T_{ij}y'_j\right\rangle
\]
and hence as announced
\[
 |\varphi(T)| \le \|T_X\|\left(\sum\|x'_i\|^2 \sum\|y'_j\|^2\right)^{1/2} \le \|T_X\|. 
\]
\end{proof}

\begin{rem}\label{comrem5.10}
 Consider $\varphi$ in $B(\ell^n_2)$. It follows from the preceding proof that
\[
\|\varphi\|_{B(\theta,n)^*}  = \inf\left\{\left(\sum|\lambda_i|^2\right)^{1/2} \|v\|_{(B(\ell^n_1,\ell^n_\infty), \Gamma_H(\ell^n_1,\ell^n_\infty))_\theta} \left(\sum|\mu_j|^2\right)^{1/2}\right\}
\]
where the infimum runs over all factorizations of $\varphi$ of the form $\varphi = \lambda_iv_{ij}\mu_j$. Indeed, $\le$ follows immediately from Proposition \ref{compro7.3} by interpolation and the converse is proved above (since we may reduce by density to the case when $\varphi$ is invertible).
\end{rem}

\begin{cor}\label{comcor4.2}
Let $(\Omega, {\cl A}, \mu)$, $(\Omega', {\cl A}', \mu')$ be a pair of measure spaces. Let $B_r = B_r(L_2(\mu), L_2(\mu'))$ and $B = B(L_2(\mu), L_2(\mu'))$. Fix $0<\theta<1$. Then the space $(B_r,B)^\theta$ consists of those $T$ in $B$ such that $T_X$ is bounded for any arcwise $\theta$-Hilbertian Banach space $X$. Moreover,
\begin{equation}\label{comeq4.8}
\|T\|_{(B_r,B)^\theta} = \sup_{n, X\in {\cl H}(\theta,n)} \|T_X\| = \sup\|T_X\|
\end{equation}
where the last supremum runs over all arcwise $\theta$-Hilbertian Banach spaces $X$.
\end{cor}

\begin{proof}
By routine arguments, we can reduce this to the case when $\mu,\mu'$ are probabilities. Then it is easy to see (e.g.\ using \cite[p. 229]{KPS}) that
\[
\|T\|_{(B_r,B)^\theta} = \sup \|{\bb E}^{{\cl B}'} T{\bb E}^{\cl B}\|_{(B_r,B)_\theta}
\]
where the supremum runs over all finite $\sigma$-subalgebras ${\cl B}\subset {\cl A}$, ${\cl B}' \subset {\cl A}'$. Similarly $\|T_X\| = \sup\limits_{{\cl B}, {\cl B}'} \|({\bb E}^{{\cl B'}}T{\bb E}^{\cl B})_X\|$. This reduces \eqref{comeq4.8} to the case of finite measure spaces which can be deduced from Theorem \ref{comthm5.1} by elementary arguments
(approximate the weights by rational numbers, then consider the span
of suitable blocks on the canonical basis of $\ell_2^n$).
\end{proof}

\begin{cor}\label{comcor4.3} Assume that a Banach space $X$
satisfies $\Delta_X(\vp)\in O(\vp^\alpha)$ for some $\alpha>0$. Then, for any $\theta<\alpha$,  $X$ is 
isomorphic to a subquotient
of an arcwise $\theta$-Hilbertian space.
\end{cor}\begin{proof}
By Theorem \ref{comthm5.1} and Theorem \ref{comthm3.1} (with $p=2$),
it suffices to show that there is a constant $C$
such that, for any $n$ and any $T\in B(\ell_2^n)$, we have
\begin{equation}\label{desired}
\|T_X\|\le C\|T\|_{B(\theta,n)}.\end{equation}
Our assumption implies by homogeneity that for any such $T$
$$\|T_X\|\le C\|T\|^\alpha \|T\|^{1-\alpha}_{\text{reg}}.$$
By a classical   real interpolation result (see \cite[p. 58]{BL}),
this can be rewritten as
$$\|T_X\|\le C'\|T\|_{(B_r(\ell_2^n),B(\ell_2^n))_{\alpha,1}}.$$
Therefore, the conclusion follows from the general fact (see \cite[p. 102]{BL})
that whenever we have an inclusion $B_0\subset B_1$,
there is a constant $C^{''}$ (depending only on 
$\alpha$ and $\theta$) such that for any
$\theta<\alpha$ and any $T$ in $B_0$
$$\|T\|_{(B_0,B_1)_{\alpha,1}} \le C^{''}\|T\|_{(B_0,B_1)_{\theta}}.$$
This yields the announced \eqref{desired}.
\end{proof}

 Recalling Theorem \ref{comthm3.1bis}, 
the next statement is now a consequence of Theorem
\ref{comthm5.1}. 

\begin{thm}\label{comthm5.6}
Let $C\ge 1$ be a constant. The following are equivalent:
\begin{itemize}
\item[\rm (i)] $X$ is $C$-isomorphic to a quotient of a subspace of an arcwise $\theta$-Hilbertian space.
\item[\rm (ii)] $\|T_X\|\le C$ for any $n$ and any $T$ such that $\|T\|_{(B_r(\ell^n_2), B(\ell^n_2))_\theta}\le 1$.
\end{itemize}
\end{thm}
\begin{proof} 
Just observe that the class of arcwise $\theta$-Euclidean (or that of $\theta$-Hilbertian) spaces is stable
by $\ell_2$-sums and apply Theorem \ref{comthm3.1bis} to the class
${\cl B}$ formed of all the $\theta$-Euclidean spaces.
\end{proof}

\section{Arcwise versus not arcwise}\label{comsec6}

At the time of this writing, we do not see (although we suspect there is) a \emph{direct} way to reduce the arcwise notions of $\theta$-Euclidean or $\theta$-Hilbertian to the non-arcwise ones. However, it follows a posteriori from our main result that any $\theta$-Hilbertian space is a subquotient
(or equivalently a quotient of a subspace)  of an arcwise one.

Indeed, consider a $\theta$-Euclidean family $\{X(z)\mid z\in\partial D\}$,  i.e.\ one such that 
\[
m(\{z\in\partial D\mid X(z) \text{ Hilbertian}\})\ge \theta.
\]
Let $\alpha(z)$ be the indicator function of the set $\Gamma_1 = \{z\in \partial D\mid X(z)$ is Hilbertian$\}$, and let $\Gamma_0 = \partial D\backslash \Gamma_1$. We now let
\[
B_0 = B_r(\ell^n_2)\qquad B_1 = B(\ell^n_2),
\]
and consider the family
\[
\beta(z) = B_{\alpha(z)}\qquad z\in \partial D.
\]
Note that $B_{\alpha(z)} = B_1$ on $\Gamma_1$ and $B_{\alpha(z)} =B_0$ on $\Gamma_0$. By \cite[Cor. 5.1]{CCRS3}, we have isometrically
\begin{equation}\label{comeq6.1}
\beta(0) = (B_0,B_1)_{\alpha(0)}
\end{equation}
where $\alpha(0) = \int \alpha(z)\ dm(z) = m(\Gamma_1)\ge \theta$. Consider now $T$ in the open unit ball of $B(\theta,n)$. By \eqref{comeq6.1}, we have
\[
\|T\|_{\beta(0)}< 1 .
\]
Therefore, there is an analytic function $T(\cdot)$ in the space $H^\infty_\#$ relative to
the family $\{\beta(z)\mid z\in\partial D\}$ satisfying
$\text{ess}\sup_{\partial D} \| T(z)\|_{Ê\beta(z)}<1$ and $T(0)=T$.
Let $Y(z)=\ell^n_2(X(z))$. Clearly (see \cite{CCRS2, Her})
we have $Y(0)=\ell^n_2(X(0))$ isometrically. 
Recall $\partial D= \Gamma_0\cup \Gamma_1$. Note that for any $z\in {\Gamma_j}$ ($j=0,1$) we have
$\beta(z)=B_j$ and hence  $\|T(z)\|_{B_j}\le 1$. Therefore  
\[
\text{ess}\sup_{z\in \partial D }\|T(z)\otimes id_{X(z)}\colon \ \ell^n_2(X(z)) \to \ell^n_2(X(z))\| \le 1.
\]
By the interpolation principle \ref{comsec4.1+},   we must have
\[
\|T(0)\otimes id_{X(0)}\colon \ \ell^n_2(X(0)) \to \ell^n_2(X(0))\| \le 1,
\]
so that  since $T=T(0)$ and $X=X(0)$ we obtain $\|T_X\| \le 1$.  By homogeneity, this shows
\begin{equation}\label{comeq6.2} \forall T\in {B(\theta,n)} \qquad 
\|T_X\| \le \|T\|_{B(\theta,n)} 
\end{equation}
for all $\theta$-Euclidean spaces $X$ (and not only the arcwise ones as in Theorem \ref{comthm5.1}). But, since the class of $X$ satisfying \eqref{comeq6.2} is trivially stable under ultraproducts, subspaces and quotients, we find that \eqref{comeq6.2} holds for all quotient of subspaces of $\theta$-Hilbertian spaces. By Theorem \ref{comthm5.6}, this establishes the following:

\begin{thm}\label{comthm6.1}
Any $\theta$-Hilbertian space is a quotient of a subspace (or equivalently is a subquotient)  of an arcwise $\theta$-Hilbertian space.
\end{thm}

\begin{rem}\label{comrem6.1+}
In particular,    since a  quotient  of a subspace of a quotient  of a subspace
is again a quotient  of a subspace,
  this shows that the properties in Theorem \ref{comthm5.6}
are equivalent to \begin{itemize}
\item[\rm (i)'] $X$ is $C$-isomorphic to a quotient of a subspace (or equivalently to a subquotient) of a  $\theta$-Hilbertian space.
   \end{itemize}
   Moreover, this shows that Corollary \ref{comcor4.2} also holds with
   ``$\theta$-Hilbertian" instead of  ``arcwise $\theta$-Hilbertian" everywhere.
\end{rem}
\section{Fourier and Schur multipliers}\label{comsec55}

Let $G$ be a locally compact Abelian group. Let $M(G)$ be the classical Banach space of complex (Radon) measures on $G$ equipped with the total variation norm:\ $\|\mu\|_{M(G)} = |\mu|(G)$. We view as usual $L_1(G)$ as (isometric to) a subspace of $M(G)$. Recall that the Fourier transform ${\cl F}$ takes $M(G)$ to $C_b(\widehat G)$ and $L_1(G)$ to the subspace of $C_b(\widehat G)$ of functions tending to zero at $\infty$. It is traditional to introduce the space $PM(G)$ of pseudo-measures on $G$ formally as the set ${\cl F}^{-1}(L_\infty(\widehat G))$ of inverse Fourier transforms, so that $\mu\in PM(G)$ iff 
\[
\hat\mu \in L_\infty(\widehat G)\quad \text{and}\quad \|\mu\|_{PM(G)} \overset{\sst\text{def}}{=} \|\hat\mu\|_\infty.
\]
Then the multiplication by $\hat \mu$ is a bounded operator on $L_2(\widehat G)$ (``Fourier multiplier''), so that the convolution operator $[\mu]\colon \ f\to f*\mu$ is bounded on $L_2(G)$ and
\begin{equation}\label{comeq5.8}
\|[\mu]\colon \ L_2(G)\to L_2(G)\| = \|\mu\|_{PM(G)}.
\end{equation}
Similarly, it is not hard to see that $[\mu]$ is regular iff $\mu\in M(G)$ and we have
\begin{equation}\label{comeq5.9}
\|[\mu]\colon \ L_2(G)\to L_2(G)\|_{\text{reg}}  =\|\mu\|_{M(G)}.
\end{equation}
Note that with this notation we have a canonical inclusion
\[
M(G) \subset PM(G).
\]

\begin{thm}\label{comthm5.3}
Let $0<\theta<1$. The space $(M(G), PM(G))^\theta$ consists of the pseudo-measures $\mu$ such that $[\mu]_X$ is bounded on $L_2(G;X)$ for any $\theta$-Hilbertian Banach space $X$ and we have
\[
\|\mu\|_{(M(G),PM(G))^\theta} = \sup_X\|[\mu]_X\colon \ L_2(X)\to L_2(X)\|
\]
where the supremum runs over all such $X$'s. The space $(L_1(G), PM(G))_\theta$ coincides isometrically with the closure of $L_1(G)\subset M(G)$ in this space.
\end{thm}

\begin{proof}
By the amenability of $G$ it is well known that we have a norm 1 projection
\[
P\colon \ B(L_2(G))\to B(L_2(G))
\]
from $L_2(G)$ onto the subspace $\{[\mu]\mid \mu\in PM(G)\}$. This is analogous to the projection of $B(\ell_2)$ onto the space of diagonal matrices. Identifying $\mu$ with $[\mu]$ by \eqref{comeq5.8}, this gives us a contractive projection
\[
P\colon \ B(L_2(G))\to PM(G)\subset B(L_2(G)).
\]
Note that by \eqref{comeq5.9} we have an isometric embedding $M(G)\subset B_r(L_2(G))$. It is easy to show (left to the reader) that $P(B_r(L_2(G))) = M(G)$. By a well known argument, it follows that $(M(G), PM(G))^\theta$ can be isometrically identified with the subspace of $(B_r,B)^\theta$ that is the range of $\mu\to [\mu]$. 
From this the first assertion follows by Corollary \ref{comcor4.2} and Remark \ref{comrem6.1+}.

\n As for the second one,  by \cite{Be}, the closure of $M(G)$ in $(M(G), PM(G))^\theta$ can be identified with $(M(G), PM(G))_\theta$. Then using the fact that there is a net of maps $\varphi_\alpha \ \colon \ M(G)\to L_1(G)$ that are simultaneous contractions on $M(G)$ and on $PM(G)$ (think of convolution by the Fejer kernel on ${\bb T}$ as the model) and such that $\|\varphi_\alpha(\mu) - \mu\|_1\to 0$ for any $\mu$ in $L_1(G)$, it is easy to deduce that $(L_1(G), PM(G))_\theta$ is the closure of $L_1(G)$ in $(M(G), PM(G))_\theta$.
\end{proof}

\begin{rem}\label{comrem5.4}
It is easy to see that the preceding statement remains valid when $G$ is a compact (non-Abelian) group or an amenable group. When $G$ is locally compact amenable,
we replace $PM(G)$ by the von Neumann algebra of $G$, i.e. the von Neumann
algebra generated by the left convolutions by elements of $L_1(G)$, or equivalently
of $M(G)$. Thus we again have an inclusion $\mu \mapsto [\mu]$
from $M(G)$ to $PM(G)$, that we use for
``compatibility" of the interpolation pair and the preceding statement continues to hold.
\end{rem}

We can reformulate Theorem \ref{comthm5.7} as one about ``Schur multipliers.'' Let ${\cl M}$ denote the Banach space of all bounded Schur multipliers on $B(\ell_2)$, i.e.\ ${\cl M}$ consists of bi-infinite matrices $[\varphi_{ij}]$ such that the map
\[
M_\varphi\colon \ [a_{ij}]\to [\varphi_{ij}a_{ij}]
\]
is bounded on $B(\ell_2)$, with norm
$\|\varphi\|_{  {\cl M} }=\| M_\varphi\colon \ B(\ell_2)\to    B(\ell_2)\|$. Let
\[
{\cl M}[n] = \{\varphi\in {\cl M}\mid \varphi_{ij}=0\quad  \forall i>n\  \forall j>n\}.
\]
We obviously have natural inclusions
\begin{align*}
{\cl M} &\subset \ell_\infty({\bb N}\times {\bb N})\\
{\cl M}[n] &\subset \ell^{n^2}_\infty.
\end{align*}

\begin{cor}\label{comcor5.8}
Let $0<\theta<1$. Fix $n\ge 1$. We have isometrically 
$$(\ell^{n^2}_\infty, {\cl M}[n])_\theta\simeq \Gamma_{\theta H}(\ell^n_1, \ell^n_\infty).$$
 Consider an $n\times n$ complex matrix $\varphi = [\varphi_{ij}]$. Then $\varphi$ belongs to the closed unit-ball of $(\ell^{n^2}_\infty,{\cl M}[n])_\theta$ iff there are a $\theta$-Hilbertian space $X$,  $k_j$ in the closed unit ball  of $X$ and $h_i$ in that of $X^*$,       such that
\begin{equation}\label{comeq5.14}
\forall i,j\qquad\qquad \varphi_{ij} = \langle k_j, h_i\rangle.
~~~~~~~~~~~~~~~~~~~~~~~~~~~~~~~~~~~~~
\end{equation}
More generally, for any $\varphi = [\varphi_{ij}] \in \ell_\infty({\bb N}\times {\bb N})$ 
\begin{equation}\label{comeq5.15}
\|\varphi\|_{(\ell_\infty({\bb N}\times{\bb N}),{\cl M})^\theta} = \inf\{\sup_j\|k_j\|_X\sup_i\|h_i\|_{X^*} \}
\end{equation}
where the infimum runs over all $\theta$-Hilbertian spaces $X$ and all bounded sequences  $(k_j)$ in $X$ and $(h_i)$ in $X^*$ such that \eqref{comeq5.14} holds. 
\end{cor}

\begin{proof}
By a classical result, we know that ${\cl M}[n]\simeq \Gamma_H(\ell^n_1,\ell^n_\infty)$ isometrically (see e.g. \cite[Th. 5.1]{P4}); moreover it is obvious that $\ell^{n^2}_\infty\simeq B(\ell^n_1,\ell^n_\infty)$ isometrically, and ``compatibly'' with the preceding isomorphism. Therefore we may isometrically identify $(\ell^{n^2}_\infty, {\cl M}[n])_\theta$ with $(B(\ell^n_1,\ell^n_\infty), \Gamma_H(\ell^n_1,\ell^n_\infty))_\theta$, which, by Theorem \ref{comthm5.7} coincides with $\Gamma_{\theta H}(\ell^n_1, \ell^n_\infty)$. But now a mapping $\varphi\colon\ \ell^n_1\to \ell^n_\infty$ is in the unit ball of $\Gamma_{\theta H}(\ell^n_1,\ell^n_\infty)$ iff it satisfies the condition in Corollary \ref{comcor5.8}. This proves the first (finite dimensional) assertion. We leave the extension to the infinite dimensional case to the reader.
\end{proof}

\begin{cor}\label{cor5.8bis}
The norm   \eqref{comeq5.15} coincides with the norm of $(\varphi_{ij})$ acting as a Schur multiplier on $(B_r(\ell_2), B(\ell_2))_\theta$ or on $(B_r(\ell_2), B(\ell_2))^\theta$. Moreover, it is also equal to its norm as a multiplier from $(B_r(\ell_2), B(\ell_2))_\theta$ to $B(\ell_2)$.
\end{cor}

\begin{proof}
Assume that $(\varphi_{ij})$ is supported on $[1,\ldots,n]\times [1,\ldots, n]$. Note that $\|\varphi\|_{\cl M} = \|M_\varphi\colon \ B(\ell^n_2)\to B(\ell^n_2)\|$ and $\|\varphi\|_{\ell_\infty({\bb N}\times {\bb N})} = \|M_\varphi\colon \ B_r(\ell^n_2)\to B_r(\ell^n_2)\|$. Hence, if \eqref{comeq5.15} is $\le 1$, by interpolation we must have $\|M_\varphi\colon \ B(\theta,n)\to B(\theta,n)\|\le 1$. A fortiori, we have $\|M_\varphi\colon \ B(\theta,n)\to B(\ell^n_2)\|\le 1$ or equivalently for any $\lambda,\mu$ in $\ell^n_2$ and $a$ in $B(\ell^n_2)$
\[
\left|\sum \varphi_{ij} a_{ij}\lambda_i\mu_j\right|\le \|a\|_{B(\theta,n)} \|\lambda\|_{\ell^n_2} \|\mu\|_{\ell^n_2}.
\]
By Proposition \ref{compro5.9}, the latter implies that
\[
\left|\sum \varphi_{ij}v_{ij}\right| \le \|v\|_{(N(\ell^n_\infty,\ell^n_1), \Gamma^*_H(\ell^n_\infty, \ell^n_1))_\theta}
\]
and hence
\[
 \|\varphi\|_{(B(\ell^n_1,\ell^n_\infty), \Gamma_H(\ell^n_1, \ell^n_\infty))_\theta}\le 1.
\]
By Theorem \ref{comthm5.7}, this implies that \eqref{comeq5.15} is $\le 1$. The preceding chain of implications completes the proof when $(\varphi_{ij})$ is finitely supported. The general case is easy to deduce by an ultraproduct argument and pointwise compactness of the unit balls of the spaces of Schur multipliers under consideration. We skip the details.
\end{proof}

\begin{rem}\label{comrem5.20}
In answer to a question of Peller, it was proved
in \cite{Har} (see also \cite{Xu}) that, for any $1<p<\infty$, the interpolation space considered in   \eqref{comeq5.15}   {\it does not } contain    the space of Schur multipliers 
that are bounded on $S_p$. The corresponding fact for Fourier multipliers on $L_p(G)$ (with respect to Theorem \ref{comthm5.3})
has been   in the harmonic analysis folklore for a long time. 
\end{rem}

\section{A characterization of uniformly curved spaces}\label{comsec7}

We will now try to obtain some sort of ``geometric'' (or ``structural'') equivalent description of uniformly curved Banach spaces, i.e.\ those $X$'s such that $\Delta_X(\vp)\to 0$ when $\vp\to 0$.

The main point is the following.  This is sort of a real interpolation variant
of \eqref{comeq5.13}.

\begin{lem}\label{comlem7.1}
Let $X$ be a finite dimensional Banach space. Assume that for some $n\ge 1$, we have   linear maps $J\colon \ X\to \ell^n_\infty$ and   $Q\colon\ \ell^n_1\to X$ both with
norm $\le1$. Then, for any $\vp>0$, the composition $JQ\colon \ \ell^n_1\to \ell^n_\infty$ admits a decomposition $JQ = u_0+u_1$ with $u_j\colon \ \ell^n_1\to \ell^n_\infty$ such that:
\[
\gamma_H(u_1)\le 2\vp^{-1}\Delta_X(\vp)\quad \text{and}\quad \|u_0\| \le 2\Delta_X(\vp).
\]
\end{lem}

The proof combines duality with a change of density argument. By duality, our assertion is equivalent to the following:\ for any $v\colon \ \ell^n_\infty\to \ell^n_1$ we have
\begin{equation}\label{comeq7.1}
|tr(vJQ)|\le 2\vp^{-1}\Delta_X(\vp) \gamma^*_H(v) + 2\Delta_X(\vp)N(v),
\end{equation}
where $N(v)$ denotes the nuclear norm of $v$. To prove this we need to recall 
Proposition \ref{compro7.2}.

\begin{proof}[Proof of Lemma \ref{comlem7.1}]
It suffices to establish \eqref{comeq7.1} for any $v\colon \ \ell^n_\infty \to \ell^n_1$. Assume that $\vp^{-1}\gamma^*_H(v) + N(v)\le 1$. Then $N(v)\le 1$ and $\gamma^*_H(v)\le \vp$. By Proposition \ref{compro7.2}, we can write $v_{ij} = \lambda'_ib_{ij}\mu'_j$ and $v_{ij} = \lambda_ia_{ij}\mu_j$ with $\lambda,\mu,\lambda',\mu'$ in the unit ball of $\ell^n_2$ and with $\|b\|_{\text{reg}} \le 1$, $\|a\|_{B(\ell^n_2)}\le \vp$. Let $\xi_i= |\lambda'_i|\vee |\lambda_i|$ and $\eta_j = |\mu_j| \vee |\mu'_j|$. Then $2^{-1/2}\xi$ and $2^{-1/2}\eta$ are in the unit ball of $\ell^n_2$ and we can write $v_{ij} = \xi_i t_{ij}\eta_j$ with
\[
\|t\|\le \vp\quad \text{and}\quad \|t\|_{\text{reg}}\le 1.
\]
(Indeed, $t_{ij} = \lambda_i\xi^{-1}_i(a_{ij})\mu_j\eta^{-1}_j$ with $|\lambda_i\xi^{-1}_i|\le 1$ and $|\mu_j\eta^{-1}_j|\le 1$, so that $\|t\|\le \|a\|$ and similarly   $\|t\|_{\text{reg}} \le \|b\|_{\text{reg}}$).
Therefore, we have $\|t_X\|_{B(\ell^n_2(X))} \le \Delta_X(\vp)$. But now, if we denote
\[
Qe_i = x_i\in X,\qquad J^*e_i =x^*_i\in X^*
\]
we find
\[
tr(vJQ) = \sum_{ij} v_{ij}\langle x^*_i,x_j\rangle = \sum_{ij} t_{ij}\langle \xi_ix^*_i,\eta_jx_j\rangle
\]
and hence
\begin{align*}
|tr(vJQ)| &\le \|t_X\|\left(\sum\|\xi_ix^*_i\|^2\right)^{1/2} \left(\sum\|\eta_jx_j\|^2\right)^{1/2}\\
&\le 2\|t_X\|\le 2\Delta_X(\vp).
\end{align*}
Thus, by homogeneity we obtain \eqref{comeq7.1}.
\end{proof}

\begin{thm}\label{comthm7.3}
Let $X$ be a Banach space. The following properties are equivalent.
\begin{itemize}
\item[\rm (i)] $X$ is uniformly curved.
\item[\rm (ii)] For some index set $I$, there are a metric surjection $Q\colon \ \ell_1(I)\to X$ and an isometric embedding $J\colon \ X\to \ell_\infty(I)$ such that $JQ$ belongs to the closure of $\Gamma_H(\ell_1(I), \ell_\infty(I))$ in $B(\ell_1(I), \ell_\infty(I))$.
\item[\rm (iii)] For any $\delta>0$ there is a constant $C(\delta)$ satisfying the following:\ for any measure spaces $(\Omega,\mu)$, $(\Omega',\mu')$, for any maps $v\colon \ X\to L_\infty(\mu)$ and $v'\colon \ L_1(\mu')\to X$ and for any $\delta>0$ there is $u$ in $\Gamma_H(L_1(\mu'), L_\infty(\mu))$ with $\gamma_H(u)\le C(\delta)\|v\|\|v'\|$ and $\|u-vv'\|<\delta \|v\|\|v'\|$.
 \item[(iv)] For each $\delta>0$, there is $C(\delta)$ satisfying:\ there are a Hilbert space $H$ and functions $\varphi\colon \ X\to H$, $\psi\colon \ X^*\to H$ such that for all $(\xi,x)$ in $X^*\times X$ we have
\begin{gather*}
|\xi(x) - \langle\psi(\xi),\varphi(x)\rangle| \le \delta \|\xi\|\|x\|\\
\|\varphi(x)\| \le C(\delta)^{1/2}\|x\|\\
\|\psi(\xi)\| \le C(\delta)^{1/2}\|\xi\|.
\end{gather*}
\end{itemize}
\end{thm}

\begin{proof}
We first show (i) $\Rightarrow$ (ii). Assume (i). Let $(X_\alpha)$ be a directed family of finite dimensional subspaces of $X$ with $\ovl{\cup X_\alpha} = X$. Let $\vp_\alpha>0$ be chosen so that $\vp_\alpha\to 0$ when $\alpha\to\infty$ in the directed family. For each $\alpha$, we can choose a finite $\vp_\alpha$-net $(x_i)_{i\in I_\alpha}$ in the unit ball of $X_\alpha$. Let $I$ be the disjoint union of the sets $I_\alpha$ and let $Q\colon \ \ell_1(I)\to X$ be the map defined by $Qe_i = x_i$. Clearly $Q$ is a metric surjection onto $X$.

Enlarging the sets $I_\alpha$ if necessary we may also find a family $(x^*_i)_{i\in I_\alpha}$ in the unit ball of $X^*$ such that $x^*_{i|X_\alpha}$ is an $\vp_\alpha$-net in the unit ball of $X^*_\alpha$.

We then define $R\colon \ \ell_1(I)\to X^*$ again by $Re_i = x^*_i$ and let $J = R^*_{|X}\colon \ X\to \ell_\infty(I)$. Note that $J$ is an isometry.

Given any bilinear form $\varphi$ on $\ell_1(I) \times \ell_1(I)$, we will denote by $\gamma_H(\varphi)$ the $\gamma_H$-norm of the associated linear map from $\ell_1(I)$ to $\ell_\infty(I)$.
 Let $\varphi\colon \ \ell_1(I)\times \ell_1(I)\to {\bb C}$ be the bilinear form associated to $JQ\colon \ \ell_1(I)\to \ell_\infty(I)$. 
 
 Let $G_\alpha=\cup \{I_\beta\mid \beta\le \alpha\}$ so that $G_\alpha$ is an increasing
 family of finite subsets of $I$ with union $I$. Observe that, if we view  $\ell_1(G_\alpha)$
 and $\ell_\infty(G_\alpha)$  as subspaces of respectively $\ell_1(I)$
 and $\ell_\infty(I)$, we have $Q\ell_1(G_\alpha) \subset X_\alpha$ and 
 $J(X_\alpha   )\subset \ell_\infty(G_\alpha)$.
 
 Let $\varphi_\alpha\colon \ \ell_1(G_\alpha)\times \ell_1(G_\alpha)\to {\bb C}$ be the restriction of $\varphi$. By the preceding observation,  the associated linear map factors through $X_\alpha$.
 Choose  $\vp$   so that $2\Delta_X(\vp)<\delta$.
By Lemma  \ref{comlem7.1}, we have a decomposition 
\[
\varphi_\alpha = \varphi^0_\alpha + \varphi^1_\alpha
\]
with $\|\varphi^0_\alpha\|\le   \delta$ and $\gamma_H(\varphi^1_\alpha) \le    C$ where $C = 2\vp^{-1} \Delta_X(\vp)$.
Let $\psi_\alpha\colon \ \ell_1(I)\times \ell_1(I)\to {\bb C}$ be defined by $\psi_\alpha(x,y) = \varphi_\alpha(x_{|G_\alpha}, y_{|G_\alpha})$. Clearly $\psi_\alpha$ admits a similar decomposition $\psi_\alpha = \psi^0_\alpha + \psi^1_\alpha$, with the same bounds. Moreover for any $x,y$ in $\ell_1(I)$
\[
\psi_\alpha(x,y)\to \varphi(x,y).
\]
Let ${\cl U}$ be an ultrafilter refining the directed net of the $\alpha$'s and let 
\begin{align*}
\varphi^0(x,y) &= \lim_{\cl U} \psi^0_\alpha(x,y)\\
\varphi^1(x,y) &= \lim_{\cl U} \psi^1_\alpha(x,y) 
\end{align*}
Then $\varphi = \varphi^0+\varphi^1$, $\|\varphi^0\| \le \delta$ and $\gamma_H(\varphi^1) \le C$, so that (ii) holds.

This shows (i) $\Rightarrow$ (ii). 
Assume (ii). Clearly, this implies the following assertion. $\forall \delta>0$ $\exists u_\delta\in \Gamma_H(\ell_1(I), \ell_\infty(I))$ such that $\|u_\delta -JQ\|<\delta$. Let $C(\delta) = \gamma_H(u_\delta)$. Assume first that $L_1(\mu) = \ell_1(I_1)$ and $L_\infty(\mu') = \ell_\infty(I_2)$, relative to sets $I_1,I_2$. By the lifting (resp.\ extension) property of $\ell_1(I_1)$ (resp.\ $\ell_\infty(I_2)$) it is easy to deduce (iii) in that case from the preceding assertion. But now since $L_1$-spaces are stable under ultraproducts, one can easily obtain (iii) with the same bound $C(\delta)$ for arbitrary measure spaces.

We have trivially (iii) $\Rightarrow$ (ii). We will now show (ii) $\Rightarrow$ (i). Assume (ii). Let $u_\delta$ and $C(\delta)$ be as above. Consider $T\colon \ \ell^n_2\to \ell^n_2$ with $\|T\|_{\text{reg}}\le 1$ and $\|T\|\le \vp$. We have
\begin{align*}
\|T\otimes JQ\| &\le \|T\otimes u_\delta\| + \|T\otimes (JQ -u_\delta)\|\\
&\le \|T\| \gamma_H(u_\delta) + \|T\|_{\text{reg}} \delta\\
&\le \vp C(\delta) + \delta
\end{align*}
and hence passing to subspaces and quotients
\[
\|T_X\|\le \vp C(\delta) + \delta.
\]
Thus we see that
\[
\Delta_X(\delta C(\delta)^{-1}) \le 2\delta
\]
which shows that the (nondecreasing) function $\Delta_X(\vp)$ takes arbitrarily small values, i.e.\ (i) holds.
This shows the equivalence of the first three properties. We now turn to the equivalence with (iv).
Let $I = B_{X^*}\times B_X$ and let $Jy(\xi,x) = \xi(y)$ and $Q^*\eta(\xi,x) = \eta(x)$ for all $(\xi,x)$ in $I$, $y\in X$, $\eta\in X^*$. Then it is easy to see that (iv) implies (ii). Conversely assume (iii). Let $(\Omega,\mu)$ (resp.\ $(\Omega',\mu')$) be the unit sphere of $X$ (resp.\ $X^*$) equipped with the discrete counting measure. Let $J\colon \ \ell_1(\Omega)\to X$ and $Q\colon \ X\to\ell_\infty(\Omega')$ be defined by $Je_\omega=\omega$ and $Q^*e_{\omega'} = \omega'$. Then $Q$ is isometric and $J$ is a metric surjection. Assuming (iii), we can write $\|QJ-vv'\|<\delta$ for some $v'\colon \ \ell_1(\Omega)\to H$, $v\colon \ H\to\ell_\infty(\Omega')$ with $\|v'\|\le C(\delta)^{1/2}$, $\|v\|\le C(\delta)^{1/2}$. We then find $\forall(\xi,x)\in \Omega'\times\Omega$
\[
 \langle\omega',\omega\rangle = \langle QJ(e_\omega), e_{\omega'}\rangle = \langle v'(e_\omega), v^*(e_{\omega'})\rangle
\]
and hence if we set $\varphi(\omega) = v'(e_\omega)$ and $\psi(\omega') = v^*(e_{\omega'})$ we obtain the desired functions but restricted to the unit spheres. Extending $\varphi$ (resp.\ $\psi$) by homogeneity to functions on the whole of $X$ (resp.\ $X^*$) yields (iv).

\end{proof}

\section{Extension property of regular operators}\label{comsec8}

Let $(\Omega,\mu)$, $(\Omega,\mu')$ be measure spaces. Let $1< p<\infty$. In this section, we will consider a subspace $S_p \subset L_p(\mu)$ and a quotient $Q_p = L_p(\mu')/R_{p}$ where $R_{p} \subset L_{p}(\mu')$ is another closed subspace. We denote by $q\colon \ L_p(\mu')\to Q_p$ the quotient map. We seek to characterize the mappings $u\colon \ S_p\to Q_p$ that admit a regular ``extension'' to a mapping $\hat u\colon \ L_p(\mu) \to L_p(\mu')$. 
We will use the following suggestive notation (motivated  by $H_p$-space theory):\ for any Banach space $X$ we denote by $S_p(X)$   the closure in $L_p(X)$   of the (algebraic) tensor product $S_p\otimes X$,
and similarly for $R_{p}(X)$. We will denote by $Q_p(X)$ the quotient space $L_p(\mu';X)/R_{p}(X)$. 
Note that we have natural linear  inclusions (with dense range)
 $S_p\otimes X\subset S_p(X)$ 
 and $Q_p\otimes X\subset Q_p(X)$.
  Consider $u\colon \ S_p\to Q_p $ and let  $u\otimes I_{X}\colon\   S_p\otimes X\to    Q_p\otimes X$ be the associated linear map.
 When this map is bounded for the induced norms,
 it uniquely extends to a bounded mapping, denoted
 by $u_X  \colon \ S_p(X)\to Q_p(X)$. In that case, we will simply say that $u_X$ is bounded.

Let $T_{p'}\subset L_{p'}(\mu')$ denote the orthogonal of $R_p$, so that, equivalently $R_p=T_{p'}^{\bot}$.
We may view $u$ as a bilinear form $\varphi$ on $S_p\times T_{p'}$. The tensor product of $u$ and the duality map $X\times X^*\to {\bb C}$ defines a bilinear form
\[
\varphi_X\colon \ (S_p\otimes X) \times (T_{p'}\otimes X^*) \to {\bb C}.
\]
Assume for simplicity that $X$ is finite dimensional (or merely reflexive). In that case
$$Q_p(X)={T_{p'}(X^*)}^*,$$ so that $\varphi_X$ extends to a bounded bilinear form on $S_p(X) \times T_{p'}(X^*)$ iff $u_X  \colon \ S_p(X)\to Q_p(X)$
is bounded.

The following extends a result from \cite{P2}.
\begin{thm}\label{comthm8.1}
Let $u\colon \ S_p\to Q_p$ be as above. Fix $C\ge 1$. The following are equivalent.
\begin{itemize}
\item[\rm (i)] There is a regular operator $\tilde u\colon \ L_p(\mu)\to L_p(\mu')$ with $\|\tilde u\|_{\rm reg} \le C$ such that $u = q\tilde u_{|S_p}$.
\item[\rm (ii)] For any Banach space $X$, $u_X$ is bounded and $\|u_X\colon \ S_p(X) \to Q_p(X)\|\le C$. 
\item[\rm (iii)] Same as (ii) for any finite dimensional Banach space.
\end{itemize}
\end{thm}

\begin{proof}
(i) $\Rightarrow$ (ii) $\Rightarrow$ (iii) are obvious so it suffices to show (iii) $\Rightarrow$ (i). Assume (iii). Let $\varphi\colon \ S_p\times T_{p'}$ be the bilinear form associated to $u$. It suffices to show that for any $v\in S_p \otimes T_{p'}$ such that $v$ lies in the unit ball of $L_p \widehat\otimes_r L_{p'}$ (see \ref{comsec1.8} above) we have
\begin{equation}\label{comeq8.1}
|\langle\varphi,v\rangle|\le C.
\end{equation}
Indeed, if this holds, $\varphi$ will admit a Hahn--Banach extension $\widetilde\varphi$ with $\|\widetilde\varphi\|_{(L_p\widehat\otimes_r L_{p'})^*} \le C$ and the linear map $\tilde u$ associated to $\widetilde\varphi$ will satisfy (i).

Thus it suffices to show \eqref{comeq8.1}. Let $v\in S_p \otimes T_{p'}$ be such that $\|v\|_{L_p\widehat\otimes_r  L_{p'}} < 1$. Note that $v \in E\otimes F$ where $E\subset S_p$, $F\subset T_{p'}$ are finite dimensional subspaces, therefore we may as well assume, without loss of generality, that $S_p$ and $T_{p'}$ themselves are finite dimensional.

By \eqref{comsec1.7} there is a finite dimensional Banach space $Y$ and there are $\xi\in L_p(\mu;Y)$, $\eta\in L_{p'}(\mu';Y^*)$ with $\|\xi\|\le 1$, $\|\eta\| \le 1$ such that  
\begin{equation}\label{comeq8.2}
v(s,t) = \langle \xi(s),\eta(t)\rangle.
\end{equation}
We claim that we can replace $Y$ by a finite dimensional space $\widetilde Y$, and replace $\xi,\eta$ by $\tilde\xi, \tilde\eta$ so that $\tilde\xi$ (resp.\ $\tilde\eta$) is in the unit ball of $S_p(\widetilde Y)$ (resp.\ $T_{p'}(\widetilde Y^*)$), but so that we still have
\begin{equation}\label{comeq8.3}
v(s,t) = \langle\tilde\xi(s), \tilde\eta(t)\rangle.
\end{equation}
Let ${\cl S} \subset L_p(\mu)$ be any subspace supplementary to $S_p$ (recall that we assume $S_p$ finite dimensional). Clearly $\xi\in L_p\otimes Y$ can be written
\[
\xi = \xi_1 + \xi_2
\]
with $\xi_1\in S_p\otimes Y$ and $\xi_2\in {\cl S}\otimes Y$. But since $v(s,t) = \langle \xi(s), \eta(t)\rangle\in S_p \otimes L_{p'}$ and $\langle\xi_2(s), \eta(t)\rangle \in {\cl S}\otimes L_{p'}$ we must have
\begin{equation}\label{comeq8.4}
\langle\xi_2(s),\eta(t)\rangle = 0.
\end{equation}
Let $Z\subset Y$ be the closed span in $Y$ of all elements of the form $\int x(s) \xi_2(s)d\mu(s)$ with $x$ a scalar valued function  in $L_{p'}$.
Note that $\xi_2$ is $Z$-valued.  By \eqref{comeq8.4} $\langle z,\eta(t)\rangle = 0$ for any $z$ in $Z$, so that $\eta$ defines an element of $L_{p'}(Z^\bot)$ with $\|\eta\|\le 1$. Let $q_1\colon \ Y\to Y/Z$ be the quotient map. Let $\hat\xi = q_1(\xi)$. Then $\hat\xi$ is in the unit ball of $L_p(Y/Z)$ and \eqref{comeq8.2} is preserved:\ if we denote $\hat\eta$ the same as $\eta$ but viewed as an element of the unit ball of $L_{p'}(Z^\bot)$, we have
\begin{equation}\label{comeq8.5}
v(s,t) = \langle\hat\xi(s), \hat\eta(t)\rangle.
\end{equation}
The advantage of this is that now $\hat\xi\in S_p(Y/Z)$ (because, since $\xi_2$ takes its values in $Z$,  $ \hat\xi=q_1(\xi_1) $). Let $\widehat Y = Y/Z$. The only missing point is that $\hat\eta$ should be also in $T_{p'}(\widehat Y^*)$ instead of only in $L_{p'}(\widehat Y^*)$. To achieve this we repeat the preceding operations on $\hat\eta$: \ we consider a direct sum decomposition $L_{p'} = T_{p'} \oplus {\cl T}$ and a decomposition 
\[
\hat\eta = \eta_1+\eta_2
\]
with $\eta_1 \in T_{p'} (\widehat Y^*)$, $\eta_2 \in {\cl T}\otimes \widehat Y^*$. We introduce the subspace $W\subset \widehat Y^*$ spanned by all integrals of the form $\int y(t) \eta_2(t) d\mu'(t)$ with $y\in L_p(\mu')$. Let $r\colon \ \widehat Y^*\to\widehat Y^*/W$ be the quotient map, and let $\widetilde Y$ be such that $\widehat Y^*/W = (\widetilde Y)^*$. Then, we set $\tilde\eta = r(\hat\eta)$. Note that $\tilde\eta$ is in the unit ball of $T_{p'}(\widetilde Y^*)$. By \eqref{comeq8.5}, since $v \in S_p \otimes T_{p'} \subset L_p \otimes T_{p'}$ we must have $\langle\hat\xi, w\rangle = 0$ for any $w$ in $W$, and hence we find that $\hat\xi(t)\in W^\bot = \widetilde Y$, so that $\hat\xi$ can be identified with an element $\tilde\xi$ in the unit ball of $S_p(\widetilde Y)$, and we have \eqref{comeq8.3}. This proves the announced claim. From this claim, it is now easy to conclude:\ we have
\[
\langle \varphi,v\rangle = \varphi_{\widetilde Y} (\tilde\xi, \tilde\eta)
\]
and hence if (iii) holds
\[
|\langle \varphi,v\rangle|\le \|\varphi_{\widetilde Y}\| \|\tilde\xi\|_{S_p(\widetilde Y)} \|\tilde\eta\|_{T_{p'}(\widetilde Y^*)} \le \|\varphi_{\widetilde Y}\| = \|u_{\widetilde Y}\|\le C.
\]
This establishes \eqref{comeq8.1}.
\end{proof}

Consider subspaces $F\subset E\subset L_p(\mu)$ and $F'\subset E' \subset L_p(\mu')$. Let $E/F$ (resp.\ $E'/F'$) be the associated subquotient of $L_p(\mu)$ (resp.\ $L_p(\mu')$). Let $v\colon\ L_p(\mu)\to L_p(\mu')$ be an operator such that $v(E)\subset E'$ and $v(F) \subset F'$. Then $v$ defines  an operator $u\colon \ E/F\to E'/F'$ with $\|u\|\le \|v\|$. We call $u$ the ``compression'' of $v$ and say that $v$ is a  ``dilation'' of $u$. Given a Banach space $X$, we can equip $(E/F)\otimes X$ with the norm of the space $E(X)/F(X)$, we denote by $(E/F)[X]$ the completion of $(E/F)\otimes X$ for this norm; and similarly for $E'/F'$. Consider an operator $u\colon \ E/F\to E'/F'$. Assume that $u\otimes I_X\colon \ (E/F) \otimes X\to (E'/F')\otimes X$ is bounded for the norms of $E(X)/F(X)$ and $E'(X)/F'(X)$. In that case, we denote by $u_X\colon \ (E/F)[X]\to (E'/F')[X]$ the resulting operator and we simply say that $u_X$ is bounded.

\begin{cor}\label{comcor8.2}
Let $F\subset E\subset L_p(\mu)$ and $F'\subset E'\subset L_p(\mu')$ be as above. Let $C\ge 1$. The following properties of a linear map $u\colon \ E/F\to E'/F'$ as equivalent:
\begin{itemize}
\item[\rm (i)] There is a regular operator $\tilde u\colon \ L_p(\mu)\to L_p(\mu')$ dilating $u$ (in particular such that $\tilde u(E) \subset E'$ and $\tilde u(F)\subset F'$) with $\|\tilde u\|_{\text{reg}} \le C$.
\item[\rm (ii)] For any  Banach space $X$, $\|u_X\|\le C$.
\item[\rm (iii)] For any finite dimensional Banach space $X$, $\|u_X\|\le C$.
\end{itemize}
\end{cor}

\begin{proof}
Let $q\colon \ E\to E/F$, $q'\colon \ E'\to E'/F'$ and $Q'\colon \ L_p(\mu')\to L_p(\mu')/F'$ be the quotient maps. Let $j\colon \ E'/F'\to L_p(\mu')/F'$ be the canonical inclusion. Assume (iii). Then, for any
finite dimensional $X$, $\|(juq)_X\colon \ E(X) \to (L_p/F')[X]\|\le C$, and hence by the theorem, $uq$ admits an ``extension'' $w\colon \ L_p\to L_{p'}$ with $\|w\|_{\text{reg}}\le C$ such that $Q'w_{|E}=juq$. Note that $Q'w(F) = 0$ and hence $w(F) \subset F'$. In addition, $Q'w(E) = ju(E/F)$ implies $w(E)+F' = Q^{\prime -1}(u(E/F)) \subset E'$, and hence (since $F'\subset E'$) $w(E)\subset E'$. It follows easily that $w$ dilates $u$. This shows that (iii) implies (i). The rest is obvious.
\end{proof}

\begin{rem}\label{comrem8.3}
Assume $p=2$. Consider a subspace $G\subset L_2(\mu)$. Note that there are a priori many inequivalent ways by which we can ``realize'' $G$ isometrically as a subquotient of $L_2(\mu)$:\ whenever we have $F\subset E \subset L_2(\mu)$ with $G = E\ominus F$ then we may identify $G$ with $E/F$ and study the ``regular operators $u\colon \ E/F\to Y$ where $Y = E'/F'$, i.e.\ the operators $u$ such that $u_X$ is bounded for all $X$. It is natural to set
\[
 \|u\|_{\text{reg}} = \sup\{\|u_X\|\mid \dim(X) <\infty\},
\]
and then the preceding statement says that $\|u\|_{\text{reg}} = \inf\|\tilde u\|_{\text{reg}}$ where the infimum runs over all dilations $\tilde u\colon \ L_2(\mu)\to L_2(\mu')$ of $u$. But we should emphasize that all this depends a priori strongly on the choice of the realization of $G$ as $E\ominus F$ !  Thus it is natural to ask what are the linear maps $u\colon\ G \to G'$ (here $G'$ is another subspace of $L_2$) for which there is a  choice of $E,F$ 
and $E',F'$ with $G = E\ominus F$ and $G' = E'\ominus F'$
that turns $u$ into a regular operator. The answer is simple: these are the maps
that are regular viewed as maps from the subspace $G$ to the quotient
$L_2/G'^\perp$. Indeed, it is easy to see that these are the extremal choices.
\end{rem}

\begin{rem}\label{comremk}
Let ${\cl B}$ be a class of Banach spaces stable by $\ell_p$-direct sum. For
any $u\colon \ S_p\to Q_p$ as in Theorem \ref{comthm8.1} we set
\[
\|u\|_{\cl B} = \sup\{\|u_X\|\mid X\in {\cl B}\}.
\]
Let $SQ({\cl B})$ denote the class of all subquotients of spaces
in ${\cl B}$. The proof of Theorem \ref{comthm8.1} shows that the following
are equivalent:
\begin{itemize}
\item[\rm (i)] $u$ admits an extension $\tilde u$ with $\|\tilde u\|_{\cl
B}\le C$,
\item[\rm (ii)] $\sup\{\|u_X\|\mid X\in SQ({\cl B})\}\le C$,
\item[\rm (iii)] $\sup\{\|u_X\|\mid X\in SQ({\cl B}), \dim (X) < \infty\}\le
C$.
\end{itemize}
A similar generalization holds in the situation of Corollary
\ref{comcor8.2}.

\end{rem}

\section{Generalizations}\label{comsec12}

In this section, we wish to give a very general statement describing in particular the interpolation space $(B(\ell^n_{p_0}), B(\ell^n_{p_1}))_\theta$ for $0<\theta<1$, $1\le p_0,p_1 \le \infty$. In order to formulate our result in full generality, we will need to introduce more specific notation. For simplicity, we first restrict to the finite dimensional case, so we fix an integer $n\ge 1$. 

The set of all $n$-dimensional Banach spaces $B_n$ is equipped with the Banach--Mazur distance $\delta(E,F) = {\log} d(E,F)$ where $d$ is defined by
\begin{equation}
d(E,F) = \inf\{\|u\| \|u^{-1}\|\}\tag*{$\forall E,F\in B_n$}
\end{equation}
where the infimum runs over all isomorphisms $u\colon \ E\to F$. By convention, we identify two elements  $E$ and $F$ of $B_n$   if $E,F$ are isometric. Then $(B_n,\delta)$ is a compact metric space. 

For any $z$ in $\partial D$, we give ourselves $p(z)$ in $[1,\infty]$ and a subset ${\cl B}(z) \subset B_n$. We will assume that $z\mapsto p(z)$ is measurable and also that $z\mapsto {\cl B}(z)$ is measurable in a suitable sense, explained below.

Let $\{Y_m\mid m\ge 1\}$ be a dense sequence in the compact metric space $B_n$. Given $Y$ in $B_n$ and a class ${\cl B}\subset B_n$, let us denote
\[
d(Y,{\cl B}) = \inf\{d(Y,X)\mid X\in {\cl B}\}.
\]
We will assume that 
\[
 z\mapsto {\cl B}(z)
\]
is measurable in the following sense:\ for any $Y$ in $B_n$, the function $z\mapsto d(Y,{\cl B}(z))$ is (Borel) measurable on $\partial D$. It is easy to check that for any (non-void) ${\cl B}\subset B_n$ and any map $u\colon \ E\to F$ between $n$-dimensional Banach spaces we have
\[
 \gamma_{\cl B}(u) = \inf_{m\ge 1} \gamma_{\{Y_m\}} (u) d(Y_m,{\cl B}).
\]
This shows that for any such $u$, the function $z\mapsto \gamma_{{\cl B}(z)}(u)$ is measurable on $\partial D$.

This, together with    \ref{comsec3.10}, shows that $\{\Gamma_{{\cl B}(z)}(E,F)\mid z\in\partial D\}$ forms a compatible family if $E,F$ are finite dimensional.

Let us say that ${\cl B}\subset B_n$ is an  $SQ(p)$-class if ${\cl B}$ contains all the $n$-dimensional spaces that are subquotients of ultraproducts of spaces of the form $\ell_p(\{X_i\mid i\in I\})$ with $X_i\in {\cl B}$ for all $i$ in $I$, $I$ being an arbitrary finite set. We assume, of course, that any space isometrically isomorphic to one in ${\cl B}$ is in ${\cl B}$ as well. Moreover, we assume that ${\cl B}$ is non-void, which boils down to assuming that ${\bb C}$ belongs to ${\cl B}$.

We assume that ${\cl B}(z)$ is an $SQ(p(z))$-class for all $z$ in $\partial D$.

For any $\xi$ in $D$, we can then define ${\cl B}(\xi)$ as the class formed of all $X$ in $B_n$ that can be written as $X = X(\xi)$ for some compatible family $\{X(z)\mid z\in \partial D\}$ such that $X(z) \in {\cl B}(z)$ for almost all $z$ in $\partial D$.

We denote by $p(\xi)$ the number defined by
\[
\frac1{p(\xi)} = \int\limits_{\partial D} \frac1{p(z)} d\mu^\xi(z)
\]
where $\mu^\xi$ is given by  \eqref{poisson}. Recall that, if $\xi=0$, $\mu^0$ is just normalized Lebesgue measure on $\partial D$. By conformal equivalence, we may always restrict ourselves, if we wish,  to the case $\xi=0$.

Let $\{X(z)\mid z\in \partial D\}$ be a compatible family of $n$-dimensional Banach spaces and let
\begin{equation}\label{rep0}L(z)=\ell_{p(z)}^n(X(z))\end{equation}
 for all $z\in \partial D$. Then let $L(\xi)$ 
and $X(\xi)$ be the interpolated families
defined for $\xi\in D$. By \cite{CCRS2,Her},  we have isometrically
\begin{equation}\label{rep1}
 L(\xi)=\ell_{p(\xi)}^n(X(\xi))\quad\forall \xi \in D.\end{equation}
In particular, when $X(z)=\bb C $ for all $z$, if we set $\ell(z)=\ell_{p(z)}^n$, then we find
\begin{equation}\label{rep2}
 \ell(\xi)=\ell_{p(\xi)}^n\quad\forall \xi \in D.\end{equation}

Given $1\le p\le\infty$ and an $SQ(p)$-class ${\cl B}\subset B_n$, we denote by $\beta(p,{\cl B})$ the space $B(\ell^n_p)$ equipped with the norm
\begin{equation}
 \|T\|_{\beta(p,{\cl B})} = \sup_{X\in {\cl B}} \|T_X\colon \ \ell^n_p(X)\to \ell^n_p(X)\|.\tag*{$\forall T\in B(\ell^n_p)$}
\end{equation}
Note that
\[
 \|T\|_{\beta(p,{\cl B})} = \sup_m\{\|T_{Y_m}\| /d(Y_m,{\cl B})\}.
\]
This shows that $z\mapsto \|T\|_{\beta(p(z), {\cl B}(z))}$ is measurable on $\partial D$ and hence (since $\|T\|\le \|T_X\|\le \|T\|_{\text{reg}}$) for all $T\colon\ \ell^n_p\to \ell^n_p$) that $\{\beta(p(z), {\cl B}(z))\mid z\in \partial D\}$ is a compatible family.

We can now state our main result.

\begin{thm}\label{comthm12.1}
 With the preceding notation, we set
\begin{equation}
 \beta(z) = \beta(p(z), {\cl B}(z)).\tag*{$\forall z\in \partial D$}
\end{equation}
Then, for all $\xi$ in $D$ we have an isometric identity $$\beta(\xi) = \beta(p(\xi), {\cl B}(\xi)).$$
\end{thm}
\begin{rk} The easier ``half" of this statement is the following claim:
for all $n\times n$-matrices $T$
$$\| T\|_{ \beta(p(\xi), {\cl B}(\xi)) }\le \| T\|_{\beta(\xi) }.$$
This follows essentially by   the interpolation property. Indeed, assume
$ \| T\|_{\beta(\xi) }<1$, so that we can write $T=T(\xi)$ for some
analytic function $T(.)$ in the unit ball of the space $H^\infty_\#$ relative to the family
$  \{\beta(p(z), {\cl B}(z))\mid z\in \partial D\} $. Then (at least on a   full measure
subset of $\partial D$) we have 
$\| T(z)\colon \ L(z) \to L(z)\|\le 1$ where $X(z)\in {\cl B}(z)$ is essentially arbitrary  and $L(z)$ is as in \eqref{rep0}
(note that for simplicity we write here $T(z)$ instead of $T(z)_{X(z)}$). By
interpolation 
this implies  $\| T(\xi)\colon \ L(\xi) \to L(\xi)\|\le 1$, and hence, using  \eqref{rep1},
since $T=T(\xi)$, we find
$\| T \colon \ \ell_{p(\xi)}^n(X(\xi)) \to \ell_{p(\xi)}^n(X(\xi))\|\le 1$, and since  
the compatible family $X(z)$ is arbitrary, we conclude that $\| T\|_{ \beta(p(\xi), {\cl B}(\xi)) }\le 1$,
which proves the   claim.
\end{rk}

We denote by $SQL_p$ the class of all subquotients of an (abstract) $L_p$-space. 
Then $SQL_p\cap B_n$ is obviously an $SQ(p)$-class (actually the smallest possible one),
formed of all $n$-dimensional  subquotients of $L_p$. Note that for either $p=1$ or
$p=\infty$, $B_n$ is entirely included in $SQL_p$.
\begin{cor}\label{comcor12.2}
Let $1\le p_0,p_1\le \infty$ and $0<\theta<1$. Let $p$ be defined by $p^{-1} = (1-\theta) p^{-1}_0 + \theta p^{-1}_1$. Then, the space $(B(\ell^n_{p_0})$, $B(\ell^n_{p_1}))_\theta$ coincides with the space $B(\ell^n_p)$ equipped with the norm $$\quad\|T\| = \sup\|T_X\|\leqno{\forall T\in B(\ell^n_p)} $$ where  the supremum runs over all $X$ that can be written as $X=X(0)$ for some compatible family of $n$-dimensional spaces $\{X(z)\mid z\in \partial D\}$ such that ${m}(\{z\in \partial D\mid X(z) \in SQL_{p_0}\}) = 1-\theta$ and ${m}(\{z\in \partial D \mid X(z) \in SQL_{p_1}\}) = \theta$. Let $\Omega(n)$ denote the class of all such spaces $X$. Then,
the space $(B(\ell_{p_0})$, $B(\ell_{p_1}))^\theta$
 can be identified with the subspace of $B(\ell_p)$ formed of all $T$ such that
 $\sup\{ \|T_X\|\mid X\in \Omega(n), n\ge 1\}<\infty$, equipped with the norm
 $$T\mapsto \sup\{ \|T_X\|\mid X\in \Omega(n), n\ge 1\},$$
 provided we make the convention that if either $p_0=\infty$ or $p_1=\infty$,
 then  $B(\ell_{p_0})$ or $B(\ell_{p_1})$ should be replaced by 
 $B(c_0,\ell_\infty)$.
\end{cor}

\begin{cor}\label{comcor10.3}
Let $1<p<\infty$ and $0<\theta<1$. The unit ball of the space $(B_r(\ell^n_p), B(\ell^n_p))_\theta$ consists of all operators $T\colon \ \ell^n_p\to \ell^n_p$ such that $\|T_X\colon \ \ell^n_p(X)\to \ell^n_p(X)\|\le 1$ for all Banach spaces $X$ which can be written as $X=X(0)$ where $\{X(z)\mid z\in \partial D\}$ is a compatible family of $n$-dimensional Banach spaces such that $X(z)$ is a subquotient of $L_p$ for all $z$ in a subarc of $\partial D$ of measure $\ge \theta$. Let $SQ_p(\theta,n)$ denote the class formed of all such spaces. \\ Let $(\Omega,\mu)$, $(\Omega',\mu')$ be measure spaces. Then the unit ball of $(B_r(L_p(\mu), L_p(\mu'))$, $B(L_p(\mu), L_p(\mu')))^\theta$ consists of those $T$ in $B(L_p(\mu), L_p(\mu'))$ such that
\[
\sup\{\|T_X\|\mid n\ge 1, X\in SQ_p(\theta,n)\}\le 1.
\]
\end{cor}

\begin{proof}
Recall $J_\theta= \{e^{2\pi it}\mid 0<t<\theta\}$. We apply Theorem 10.1 with $p(z)\equiv p$ and ${\cl B}(z) = B_n$ for $z$ in $J_\theta$ and ${\cl B}(z)$ equal to $B_n\cap SQL_p$ for all $z$ in $\partial D\backslash J_\theta$. The second assertion is left to the reader.
\end{proof}

\begin{cor}\label{comcor11.4}
Fix $0<\theta<1$. Consider a measurable partition $\partial D = J'_0\cup J'_0 \cup J_1$ with
\[
|J'_0| = (1-\theta)/2,\quad |J''_0| = (1-\theta)/2,\quad |J_1| = \theta.
\]
We set
\[
B(z) = \begin{cases}
B(\ell^n_1)&\text{if $z\in J'_0$}\\
B(\ell^n_\infty)&\text{if $z\in J''_0$}\\
B(\ell^n_2)&\text{if $z\in J_1$.}
       \end{cases}
\]
We have then isometrically
\[
B(0) \simeq (B_r(\ell^n_2), B(\ell^n_2))_\theta.
\]
\end{cor}

\begin{proof}
Observe that, in this setting, if ${\cl B}(z)$ consists of all $n$-dimensional Banach (resp.\ Hilbert) spaces for all $z$ in $J'_0\cup J''_0$ (resp.\ $z$ in $J_1$), and if $p(z) = 1$ on $J'_0$, $p(z) = \infty$ on $J''_0$ and $p(z)=2$ on $J_1$, then in the preceding Theorem we have $\beta(z) = B(\ell^n_{p(z)})$. Therefore, since ${\cl B}(0)$ is the class of all $n$-dimensional $\theta$-Hilbertian spaces we have
\begin{equation}
\|T\|_{\beta(0)} = \sup\|T_X\|_{B(\ell^n_2(X))}\tag*{$\forall T\in B(\ell^n_2)$}
\end{equation}
where the sup runs over all $n$-dimensional $\theta$-Hilbertian spaces. Thus we obtain the same as \eqref{comeq4.0}.
\end{proof}

The proof of Theorem \ref{comthm12.1} is similar to that of the above Theorem \ref{comthm5.1}. We will merely describe the main ingredients.

\begin{lem}\label{comlem10.3}
Given $n$-dimensional Banach spaces $E,F$, let $\gamma(z) = \Gamma_{{\cl B}(z)} (E,F)$. Then for all $v\colon \ E\to F$ we have
\begin{equation}
 \|v\|_{\Gamma_{{\cl B}(\xi)}(E,F)}\le \|v\|_{\gamma(\xi)}.\tag*{$\forall\xi\in D$}
\end{equation}
\end{lem}

\begin{proof}
Assume $\|v\|_{\gamma(\xi)}<1$. By definition of $\gamma(\xi)$ (see \cite{CCRS3}) there is a bounded analytic function $z\to v(z)$ on $D$ such that $v(\xi)=v$ and $\underset{\sst z\in\partial D}{\text{ess sup}}\gamma_{{\cl B}(z)}(v(z)) < 1$. Then we can write (by measurable selection, see \ref{comsec1.10}) $v(z) = v_1(z)v_2(z)$ with $E \overset{\sst v_2(z)}{\hbox to 25pt{\rightarrowfill}} X(z)$ and $X(z) \overset{\sst v_1(z)}{\hbox to 25pt{\rightarrowfill}} F$ such that $\|v_1(z)\|<1$, $\|v_2(z)\|\le 1$ and $X(z) \in {\cl B}(z)$ a.e.

We now define a compatible family $\{Y(z)\mid z\in \partial D\}$ by setting $Y(z)=E$ equipped with the norm:
\begin{equation}
\|e\|_{Y(z)} = \|v_2(z)e\|_{X(z)} \le \|e\|.\tag*{$\forall e\in E$}
\end{equation}
We  have then
\begin{equation}\label{comeq10.2bis}
\|v(z)e\|_F = \|v_1(z)v_2(z)e\| \le \|e\|_{Y(z)}.
\end{equation}
By density, we may assume $v$ invertible, so that $v(\xi)$ is invertible and hence $v(z)$ is invertible a.e.\ on $\partial D$. Moreover, since $z\mapsto \det(v(z))$ is bounded and analytic on $D$ and $\ne 0$ at $\xi$, the function $z\mapsto \log|\det(v(z))|$ is in $L_1(\partial D)$. This implies that $z\to \log\|v(z)^{-1}\|$ is in $L_1(\partial D)$. We have
\[
k_1(z)\|e\|\le \|e\|_{Y(z)}
\]
with $k_1(z) = \|v(z)^{-1}\|^{-1}$. By the preceding observation, $k_1$ is in $L_1(\partial D)$, and thus $\{Y(z)\mid z\in\partial D\}$ is a compatible family. To conclude, applying the basic interpolation property to \eqref{comeq10.2bis} we find
\begin{equation}
\|v(\xi)e\|_F\le \|e\|_{Y(\xi)}\tag*{$\forall e\in E ~~\forall \xi\in D$}
\end{equation}
and also
\[
\|e\|_{Y(\xi)} \le \|e\|.
\]
Since $Y(z)$ and $X(z)$ are isometric for all $z$ in $\partial D$, $Y(\xi)$ is in ${\cl B}(\xi)$ and hence we conclude
\[
 \gamma_{{\cl B}(\xi)}(v) = \gamma_{{\cl B}(\xi)}(v(\xi))\le 1.
\]
By homogeneity (and by the density of invertibles) this completes the proof.
\end{proof}

Given a class ${\cl B}\subset B_n$, we denote by $\Gamma_{\cl B}(\ell^n_1,\ell^n_\infty)$ the space $B(\ell^n_1,\ell^n_\infty)$ equipped with the norm:
\begin{equation}
 \gamma_{\cl B}(u) = \inf\{\|u_1\| \|u_2\|\},\tag*{$\forall u\colon\ \ell^n_1\to\ell^n_\infty$}
\end{equation}
where the infimum runs over all factorizations of $u$ of the form $\ell^n_1 \overset{\sst u_2}{\longrightarrow} X \overset{\sst u_1}{\longrightarrow} \ell^n_\infty$ for some $X$ in ${\cl B}$. We denote by $\Gamma^*_{\cl B}(\ell^n_\infty,\ell^n_1)$ the dual space, i.e.\ the space $B(\ell^n_\infty, \ell^n_1)$ equipped with the dual norm $\gamma^*_{\cl B}$. Moreover, $\gamma^*_{\cl B}$ can be described as follows.
Assume that ${\cl B}\subset B_n$ is an $SQ(p)$-class. A linear map $v\colon \ \ell^n_\infty\to\ell^n_1$ satisfies $\gamma^*_{\cl B}(v)\le 1$ iff there are $\lambda,\mu$ in the unit ball of respectively $\ell^n_{p'}$ and $\ell^n_p$ and a linear map $a\colon \ \ell^n_p\to \ell^n_p$ such that $\|a\|_{\beta(p,{\cl B})}\le 1$ and $v_{ij} = \lambda_ia_{ij}\mu_j$ for all $i,j=1,\ldots, n$.

Let us set
\begin{equation}
\Gamma(z) = \Gamma_{{\cl B}(z)}(\ell^n_1,\ell^n_\infty)\quad \text{and}\quad \Gamma^*(z) = \Gamma^*_{{\cl B}(z)}(\ell^n_\infty, \ell^n_1).\tag*{$\forall z\in \partial D$}
\end{equation}
Then, for all $\xi$ in $D$, the space $\Gamma(\xi)$ coincides isometrically with $\Gamma_{{\cl B}(\xi)}(\ell^n_1, \ell^n_\infty)$. By duality, $\Gamma^*(\xi)$ coincides with $\Gamma^*_{{\cl B}(\xi)} (\ell^n_\infty, \ell^n_1)$. 
Similarly, a linear map $w\colon \ \ell^n_p \to \ell^n_p$ is in the unit ball of $\beta(p,{\cl B})^*$ iff it admits a factorization $w_{ij} = \lambda_ib_{ij}\mu_j$ for some $\lambda,\mu$ in the unit ball respectively of $\ell^n_p$ and $\ell^n_{p'}$ and some $b$ in the unit ball of $\Gamma_{\cl B}(\ell^n_1,\ell^n_\infty)$.

\begin{proof}[Sketch of Proof of Theorem \ref{comthm12.1}]
By the remark after Theorem \ref{comthm12.1}, it remains only to show
that $ \| T\|_{\beta(\xi) }\le \| T\|_{ \beta(p(\xi), {\cl B}(\xi)) }.$ Assume $\| T\|_{ \beta(p(\xi), {\cl B}(\xi)) }<1.$
Consider $w$ in the open unit ball of ${\beta(\xi) }^*$. It suffices to show that
$|\langle w,T\rangle| \le 1$. By what precedes, we can write $w_{ij} = \lambda_ib_{ij}\mu_j$
with $\lambda,\mu$ in the unit ball respectively of $\ell^n_{p(\xi)}$ and $\ell^n_{p(\xi)'}$ and some $b$ in the unit ball of $\Gamma_{{\cl B}(\xi)}(\ell^n_1,\ell^n_\infty)$. Then using \eqref{rep2} both for $p(\xi)$ and its conjugate
$p(\xi)'$, we can conclude by essentially  the same reasoning as above 
for Theorem  \ref{comthm5.1}.
\end{proof}

Note that if ${\cl B} =B_n$ then for any $a\colon \ \ell^n_p\to \ell^n_p$ with associated matrix $(a_{ij})$ we have
\[
 \|a\|_{\beta(p,{\cl B})} = \|a\|_{B_r(\ell^n_p)}=  \|[|a_{ij}|]\|_{B(\ell^n_p)}.
\]

Thus, Theorem \ref{comthm12.1}
generalizes on one hand the main result of \cite{P2} (which corresponds to the case $p_0=1, p_1=\infty$ and ${\cl B}(z)=B_n$ for all $z$), and on the other hand
Theorem \ref{comthm5.1} (which corresponds to the case $p_0= p_1=2$ with ${\cl B}(z)=B_n$ on a set of measure $1-\theta$ and ${\cl B}(z)=\{\ell_2^n\}$ on the complement).

 Theorem \ref{comthm12.1} can also be interpreted as the commutation of complex interpolation with
 the duality described at the beginning
 of  \S \ref{comsec3}. Given a class of $n$-dimensional  Banach spaces  $ {\cl B} $, the polar of  $ {\cl B} $ is the unit ball
 of a Banach space of operators on $\ell_p^n$. Let us denote
 the latter   Banach space by $ {\cl B} ^\bullet $.
  Then, taking say $p(z)$ constantly equal to $p$, 
 consider a measurable family of $SQ(p)$-classes
 of $n$-dimensional  Banach spaces  $\{{\cl B}(z)\mid z\in \partial D\}$,
 and for any $\xi\in D$ let ${\cl B}(\xi) $ be as above. 
  The preceding result then states that for all $\xi $ in $D$
 $${\cl B}(\xi)^\bullet ={\cl B} ^\bullet(\xi)$$
 where the right hand side means the result of complex interpolation applied to the family $  \{ {\cl B} ^\bullet(z)\mid z\in \partial D\}$.

\section{Operator space case}\label{comsec9}

This section is deliberately modeled on section \ref{comsec4bis}: the statement numbered as \ref{comsec9}.n in this section
corresponds to the one numbered
as \ref{comsec4bis}.n in section \ref{comsec4bis}.

In this section, we turn to operator spaces, i.e.\ closed subspaces $X\subset B(H)$ ($H$ Hilbert) equipped with an ``operator space structure,'' i.e.\ the sequence of normed spaces $(M_n(X), \|~~\|_{M_n(X)})$, where $\|\cdot\|_{M_n(X)}$ is simply the norm induced by $M_n(B(H))$. We refer to \cite{ER,P5} for general background on operator spaces, completely bounded (in short c.b.) maps, to \cite{P6} for notions on non-commutative vector-valued $L_p$-spaces used below, and to \cite{P3} for the non-commutative version of ``regular operators'' on the Hilbert--Schmidt class $S_2$ or on more general non-commutative $L_p$. We first give a presentation parallel to \S \ref{comsec4bis}.

Let $S_p$ denote the Schatten $p$-class $(1\le p<\infty)$. A mapping $T\colon \ S_p\to S_p$ is completely regular (\cite{P3}) iff it is a linear combination of bounded completely positive (c.p.\ in short) $u_j$, $(j=1,\ldots, 4)$ so that $u = u_1-u_2 + i(u_3-u_4)$. 
It is known (see \cite{P3}) that $u$ is {\it completely regular} iff the map $u_X = u\otimes Id_X$ is bounded on $S_p[X]$ for any operator space $X$. Equivalently, this holds iff the sequence of maps $u_{M_n}\colon \ S_p[M_n]\to S_p[M_n]$ is uniformly bounded. We define the corresponding norm by:
\[
\|u\|_{\text{reg}} = \sup_{n\ge 1}\|u_{M_n}\colon\ S_p[M_n]\to S_p[M_n]\|,
\]
and we denote by $B_r(S_p)$ the space of such maps. In analogy with the inclusion $B_r(\ell_2)\subset B(\ell_2)$ we have a norm one inclusion $B_r(S_p) \subset CB(S_p)$. Therefore, again we may consider the interpolation space $(B_r(S_p), CB(S_p))^\theta$. For simplicity, we concentrate first on the case $p=2$.  Unless we specify otherwise, $S_2$ or any non-commutative $L_2$-space is assumed equipped with its OH-operator space structure in the sense of \cite{P6}

Consider a compatible family of $N$-dimensional Banach spaces $\{X(z)\mid z\in\partial D\}$ as before. Assume that each $X(z)$ is equipped with an operator space structure,
such that, for each $n\ge 1$, the family  $\{M_n(X(z))\mid z\in \partial D\}$ is compatible.
Then we say that $\{X(z)\}$ is a compatible family of operator spaces. 

Let us set $M_n(X)(z) = M_n(X(z))$. We equip $X(0)$ (resp.\ $X(\xi)$ for $\xi\in D$) with an operator space structure by setting
\begin{equation}\label{comeq9.0}
 M_n(X(0)) = M_n(X)(0)
\end{equation}
(resp.\ $M_n(X(\xi)) = M_n(X)(\xi)$). That this is indeed an operator space structure follows from Ruan's fundamental theorem (see \cite[p. 33]{ER} \cite[p. 35]{P5}).
See \cite{P5,P6} for   complex interpolation of operator spaces. 

Now assume that $X(z)$ is completely isometric to $OH_n$ for all $z$ in a set of measure $\ge \theta$. Then any operator space that is completely isometric to $X(0)$ will be called $\theta$-0-Euclidean (arcwise if the set can be an arc), and any ultraproduct of (resp. arcwise) $\theta$-0-Euclidean spaces will be called (resp. arcwise) $\theta$-0-Hilbertian.

We should recall that, by the duality of operator spaces for any operator space $E$, we have 
\begin{align}\label{comeq9.2}
 CB(S^n_1,E) &= M_n(E)\\
\intertext{and}
\label{comeq9.3}
 CB(E,M_n) &= M_n(E^*)
\end{align}
completely isometrically. Given two operator spaces $E,F$, we will denote by $\Gamma_{OH}(E,F)$ the space of all linear maps $u\colon \ E\to F$ wich factor through an operator Hilbert space $H$ equipped with the o.s.\ structure of $OH$ in the sense of \cite{P6} or \cite{P5}. We equip this space with the (Banach space) norm
\[
 \gamma_{OH}(u) = \inf\{\|u_1\|_{cb} \|u_2\|_{cb}\}
\]
where the infimum runs over all factorizations $u_2\colon \ E\to H$, $u_1\colon \ H\to F$ with $u=u_1u_2$. More generally, given an operator space $X$ let us denote
\[
 \gamma_X(u) = \inf\{\|u_1\|_{cb} \|u_2\|_{cb}\}
\]
where the infimum runs over all possible factorizations of $u$ of the form $u=u_1u_2$ with c.b.\ maps $u_2\colon \ E\to X$ and $u_1\colon \ X\to F$. If there is no such factorization, we set $\gamma_X(u) = \infty$.

 Similarly, we will denote by $\Gamma_{\theta OH}(E,F)$ the space of maps that factor through a $\theta$-O-Hilbertian space $X$ and we again equip it with the norm $\gamma_{\theta OH}(u) = \inf\{\gamma_X(u)\}$, with the inf over all $\theta$-O-Hilbertian spaces $X$.
\begin{lm}
Assume $E=S^n_1$ and $F=M_n$, if $u\colon \ E\to F$ factors through an ultraproduct $X=\Pi X_i/{\cl U}$ of operator spaces, then \begin{equation}\label{comeq9.4}
 \gamma_X(u) =\lim_{\cl U} \gamma_{X_i}(u).
\end{equation}
\end{lm}
\begin{proof}
Indeed, by definition of ultraproducts for operator spaces $M_n(X) = \Pi M_n(X_i)/{\cl U}$ and hence by \eqref{comeq9.2} and \eqref{comeq9.3} we have an isometric identity
\begin{equation}\label{comeq9.5}
 \Pi CB(S^n_1,X_i)/{\cl U} = CB(S^n_1,X)
\end{equation}
and an isometric embedding
\begin{equation}\label{comeq9.6}
 \Pi CB(X_i,M_n)/{\cl U}\subset CB(X,M_n).
\end{equation}
If $\dim(X_i) = d <\infty$ for all $i$, then $\dim(X)=d$ and \eqref{comeq9.6} becomes an equality because both sides have the same dimension. Now assume $u=u_1u_2$ with c.b.\ maps $u_2\colon \ S^n_1\to X$ and $u_2\colon \ X\to M_n$. By \eqref{comeq9.5} we have maps $u_2(i)\colon \ S^n_1\to X_i$ corresponding to $u_2$ with $\lim\limits_{\cl U}\|u_2(i)\|_{cb} = \|u_2\|_{cb}$. Let $Y_i\subset X_i$ be the range of $u_2(i)$ and let $Y = \Pi Y_i/{\cl U}$. Clearly $Y\subset X$ and $\dim Y\le n$. Consider the restriction $v_1 = u_{1|Y}\colon \ Y\to M_n$. We have $\|v_1\|_{cb} = \|u_{1|Y}\|_{cb} \le \|u_1\|_{cb}$ and by the remark following \eqref{comeq9.6} there are maps $v_1(i)\colon \ Y_i\to M_n$ corresponding to $v_1$ such that $\lim\limits_{\cl U}\|v_1(i)\|_{cb} = \|v_1\|_{cb}$. Let $u_1(i)\colon \ X_i\to M_n$ be an extension of $v_1(i)$ with $\|u_1(i)\|_{cb} = \|v_1(i)\|_{cb}$. Let $u_i = u_1(i)u_2(i)$. Clearly, we have $u_i-u\to 0$ and $\lim\limits_{\cl U} \gamma_{X_i}(u_i) \le \gamma_X(u)$. By an elementary perturbation argument (see \cite[p.~69]{P5}), since $\|u_i(x) -u(x)\|\to 0$ for all $x$, there is an isomorphism $w_i\colon \ u_i(S^n_1) \to u(S^n_1)$ such that
\[
 \lim_{\cl U} \|w_i\|_{cb} = \lim_{\cl U} \|w^{-1}_i\|_{cb} = 1\quad \text{and} \quad u = w_iu_i.
\]
Let $\widetilde w_i\colon \ M_n\to M_n$ be an extension of $w_i$ with $\|\widetilde w_i\|_{cb} = \|w_i\|_{cb}$ so that $u = \widetilde w_iu_i$. Thus we obtain
\[
 \lim_{\cl U} \gamma_{X_i}(u) \le \lim_{\cl U}\|\widetilde w_i\|_{cb} \gamma_{X_i}(u_i) \le \gamma_X(u).
\]
The converse inequality being trivial, this completes the proof of \eqref{comeq9.4}.
 \end{proof}
 
\begin{lem}\label{comlem9.2}
 Fix $n\ge 1$ and $0<\theta<1$. Let $E,F$ be $n$-dimensional operator spaces. Consider a linear map $u\colon \ E\to F$ in the unit ball of $(CB(E,F), \Gamma_{OH}(E,F))_\theta$. Then
\[
 \gamma_{\theta OH}(u)\le 1.
\]
More precisely, if $u$ is a linear isomorphism, then $u$ admits a factorization $u=u_1u_2$ with $\|u_2\colon \ E\to X\|_{cb}\le 1$ and $\|u_1\colon \ X\to F\|_{cb}\le 1$ where $X =X(0)$ and $X=Y(0)$ where $\{X(z)\}$ an $\{Y(z)\}$ are compatible families of operator spaces such that $Y(z) = X(z)\simeq OH_n$ $\forall z\in J_\theta$ and
\[
 X(z) \simeq F,\qquad Y(z)\simeq E\qquad \forall z\notin J_\theta.
\]
\begin{proof}
The argument for  Lemma \ref{comlem5.10} extends without difficulty. We leave the details to the reader.
\end{proof}

\begin{thm}\label{comthm9.3}
 Fix $n\ge 1$ and $0<\theta<1$. We have an isometric identity
\[
 (CB(S^n_1,M_n), \Gamma_{OH}(S^n_1, M_n))_\theta = \Gamma_{\theta OH}(S^n_1,M_n).
\]
\end{thm}

\begin{proof}
Consider $u\colon\ S^n_1\to M_n$ with $\gamma_{\theta OH}(u)<1$. By \eqref{comeq9.4} we may assume that $u=u_1u_2$ with $u_2\colon \ X\to M_n$, $u_1\colon \ S^n_1\to X$ satisfying $\|u_1\|_{cb} \|u_2\|_{cb}<1$ with $X$ a (finite dimensional) $\theta$-O-Euclidean space. Assume $X=X(0)$ with $X(z)$ as in the definition of $\theta$-O-Euclidean. 
Let us define $CB( S^n_1  ,   X)(z)$ 
and $CB(X, M_n)(z)$ by setting 
$$CB( S^n_1  ,   X)(z)=CB( S^n_1  ,   X(z))
\ \ {\rm and} \ \ CB(X, M_n)(z)=CB(X(z), M_n).$$

Using \eqref{comeq9.2}, \eqref{comeq9.3}, and \eqref{comeq9.0} we obtain isometric identities
\begin{align*}
 CB(S^n_1,X) &= CB(S^n_1, X)(0)\\
CB(X,M_n) &= CB(X, M_n)(0)
\end{align*}
from which it is easy to derive bounded analytic functions $z\to u_1(z)\colon \ X(z)\to M_n$ and $z\to u_2(z)\colon \ S^n_1\to X(z)$ on $S$ such that $u = u_1(0) u_2(0)$ with
\[
 {\rm ess}\sup_{z\in \partial D}\|u_1(z)\|_{cb} \|u_2(z)\|_{cb}<1.
\]
By \ref{comsec4.0},
this clearly implies by \eqref{comeq9.2} and \eqref{comeq9.3}, that the norm of $u$ in
$
 (CB(S^n_1,M_n), \Gamma_{OH}(S^n_1,M_n))_\theta$ is $ < 1$.
Since the converse direction is a particular case of Lemma \ref{comlem9.2}, the proof is complete.
\end{proof}

\end{lem}

As before, if we are given a norm $\gamma$ on $B(E,F)$ (with $E,F$ finite dimensional) we define $\gamma^*(v)$ for any $v\colon \ F\to E$ by setting
\[
 \gamma^*(v) = \sup\{|tr(uv)|\mid u\colon \ E\to F , \ \gamma(u)\le 1\}.
\]
We denote by $N_0(v)$ the operator space version of the nuclear norm of   $v$.

We first need to extend Proposition \ref{compro7.2} to the operator space setting. \begin{prop}
 Consider $v\colon \ M_n\to S^n_1$.
\begin{itemize}
\item[\rm (i)] $\gamma^*_{OH}(v)\le 1$ iff there are $\lambda_1,\lambda_2,\mu_1,\mu_2$ in the unit ball of $S^n_4$ and $a\colon \ S^n_2\to S^n_2$ in the unit ball of $CB(S^n_2)$ (which, as a Banach space, coincides with $B(S^n_2)$) such that
\begin{equation}
 v(x) = \mu_1a(\lambda_1x\lambda_2)\mu_2.\tag*{$\forall x\in M_n$}
\end{equation}
\item[\rm (ii)] $N_0(v) \le 1$ iff there are $\lambda_1,\lambda_2,\mu_1,\mu_2$ in the unit ball of $S^n_4$ and $a\colon \ S^n_2\to S^n_2$ in the unit ball of $B_r(S^n_2)$ such that
\begin{equation}
 v(x) = \mu_1a(\lambda_1x \lambda_2)\mu_2.\tag*{$\forall x\in M_n$}
\end{equation}
\end{itemize}
\end{prop}

\begin{proof}
 \begin{itemize}
\item[(i)] This follows from \cite[Th 6.1]{P7} and \cite[Th 6.8]{P6}.
\item[(ii)] By \cite[Cor. 3.3]{P3} we have an isometric identity
\[
 (CB(M_n,M_n), CB(S^n_1,S^n_1))_{\frac12} = B_r(S^n_2,S^n_2).
\]  \end{itemize}
Now consider $v$ such that $N_0(v)<1$. Then there are $a,b$ in the open unit ball of $S^n_2$ such that $v(x) = \tilde v(ax b)$ for some $\tilde v\colon \ S^n_1\mapsto S^n_1$ with $\|\tilde v\|_{cb}<1$ (see e.g.\ \cite[Th. 5.9 and Rem 5.10]{P6}). We may as well assume, by perturbation, that $a,b$ are invertible and positive. Since $N_0(v^*) = N_0(v)$, applying the same argument to $v^*\colon \ M_n\to S^n_1$ and taking adjoints again, we obtain $c,d$ invertible and positive in the unit ball of $S^n_2$ such that
\[
 v(x) = c\hat v(x)d
\]
for some $\hat v\colon \ M_n\to M_n$ with $\|\hat v\|_{cb}\le 1$. Let $A(x) = ax b$, $B(x) = cx d$. Then, $B^{-1}v = \hat v$ and $\tilde v = vA^{-1}$ and hence
\[
\|B^{-1}v\colon \ M_n\to M_n\|_{cb}<1\quad \text{and}\quad \|vA^{-1}\colon \ S^n_1\to S^n_1\|_{cb} <1. 
\]
Consider then, again on the strip $S$, the analytic function $f\colon \ S\to B(M_n,M_n)$ defined by $f(z) = B^{z-1}vA^{-z}$. Since for any real $t$, $B^{it}$ and $A^{it}$ are isometric both on $M_n$ and $S^n_1$, we have
\[
 \|f(z)\colon \ M_n\to M_n\|_{cb} < 1\quad \text{for any}\quad z\quad \text{in}\quad \partial_0
\]
and $\|f(z)\colon \ S^n_1\to S^n_1\|_{cb} < 1$ for any $z$ in $\partial_1$. This implies $\|f(1/2)\colon \ S^n_2\to S^n_2\|_{\text{reg}}<1$. Thus if we set $a = f(1/2)$, we obtain
\[
 v = B^{1/2}a(A^{1/2}x) = \mu_1a(\lambda_1x \lambda_2)\mu_2
\]
where $\mu_1=c^{1/2}$, $\mu_2=d^{1/2}$, $\lambda_1 = a^{1/2}$, $\lambda_2 = b^{1/2}$. Conversely assume $v(x) = \mu_1a(\lambda_1x \lambda_2)\mu_2$ with $\lambda_1\lambda_2$, $\mu_1,\mu_2$ as in (ii) and $\|a\|_{\text{reg}}\le 1$. The mapping $x\to \lambda_1x\lambda_2$ from $M_n$ to $S^n_2$ corresponds to a tensor $T$ in the unit ball of $S^n_2[M^*_n]$ (see e.g. \cite[Th. 1.5]{P6}); the mapping $x\to a(\lambda_1x\lambda_2)$ corresponds to $(a\otimes id)(T)$ and hence is also in the unit ball of $S^n_2[M^*_n]$ (this follows from \cite[(2.1) and (2.2)]{P3} which can be checked easily using \cite[Lemma 5.4]{P6}), therefore the mapping $x\to \mu_1a(\lambda_1x\lambda_2)\mu_2$ is also in the latter ball (cf.\ e.g.\ \cite[Th. 1.5]{P6}). So we conclude that the tensor associated to $v$ is in the unit ball of $S^n_1[M^*_n] = S^n_1[S^n_1] = S^n_1 \otimes^{\wedge} S^n_1 = N_0(M_n,S^n_1)$.

\end{proof}

\begin{pro}\label{comlem9.5}
 Fix $n\ge 1$ and $0<\theta<1$. Consider $v\colon \ M_n\to S^n_1$. Then $v$ belongs to the unit ball of $(N_0(M_n,S^n_1), \Gamma^*_{OH}(M_n, S^n_1))_\theta$ iff there are $\lambda_1,\lambda_1,\mu_1,\mu_2$ in the unit ball of $S^n_4$ and $a\colon \ S^n_2\to S^n_2$ in the unit ball of $(B_r(S^n_2), CB(S^n_2))_\theta$ such that
\begin{equation}
 v(x) = \mu_1a(\lambda_1x\lambda_2)\mu_2.\tag*{$\forall x\in M_n$}
\end{equation}
\end{pro}

\begin{proof}
The proof of Proposition \ref{compro5.9} can be easily generalized to the operator space case. We merely give a hint and let the reader check the details. Given $v$ in the open unit ball of $(N_0(M_n, S^n_1)$, $\Gamma^*_{OH}(M_n, S^n_1))_\theta$, by the preceding Proposition and a measurable selection (see \ref{comsec1.10}), we can find 
an analytic function $z\to v(z)$ such that $v(\theta)=v$ that is bounded on $S$ 
together with measurable functions on $\partial S$ $z\to \lambda_j(z)$, $z\to\mu_j(z)$ $(j=1,2)$ with values in the unit ball of $S^n_4$,  and 
$z\to a(z)$ 
such that $a_{|\partial_j}$ takes its values in the unit ball of $B_r(S_2^n)$ for $j=0$ and in that 
of $B(S_2^n)$ for $j=1$,  such that if we define $A(z)(x) = \mu_1(z) x\mu_2(z)$ and $B(z)x = \lambda_1(z) x\lambda_2(z)$, we have 
for all $z$ in $\partial$
\begin{equation}\label{comeq9.7}
v(z) = B(z)a(z)A(z).
\end{equation}
Choosing $\vp>0$ sufficiently small and replacing $\lambda_j$ and $\mu_j$ $(j=1,2)$ respectively by $\vp I+|\lambda_j|$ and $\vp I+|\mu_j|$, we may assume that $\lambda_j\ge \vp I$ and $\mu_j \ge \vp I$. By the matricial Szeg\"o theorem, there are bounded $M_n$-valued analytic functions $F_1,F_2,G_1,G_2$ such that $|F_j| = \lambda_j$ and $|G_j| =\mu_j$ on $\partial S$. We now replace $\lambda_j,\mu_j$ by $F_j,G_j$ and make the corresponding change of $a$ so that \eqref{comeq9.7} holds.   Then since $z\mapsto B(z)^{-1}$ and $z\mapsto A(z)^{-1}$ are analytic, the analyticity of $z\mapsto v(z)$ guarantees that of $z\mapsto a(z)$. It follows that $\|a(\theta)\|_{(B_r(S^n_2), B(S^n_2))_\theta}\le 1$. Since $v=v(\theta)$, and since $F_j(\theta)$, $G_j(\theta)$ are (by the maximum principle) in the unit ball of $S^n_4$, we conclude that $v=v(\theta) = B(\theta)a(\theta)A(\theta)$ can be factorized as stated in Proposition \ref{compro5.9}. This proves the ``only if part.'' The `if part'' can also be checked by a (simpler) adaptation of the corresponding part of the proof of Proposition \ref{compro5.9}.
\end{proof}

We can now state the analogue of Theorem \ref{comthm5.1} (recall that $  CB(S^n_2) =B(S^n_2)   $ isometrically) .

\begin{thm}\label{comthm9.4}
Let ${\cl O}{\cl H}(\theta,n)$ be the set of $n$-dimensional arcwise $\theta$-0-Hilbertian operator spaces. Let $0<\theta<1$. Consider $CB(\theta,n) =(B_r(S^n_2), B(S^n_2))_\theta$. Then, for any $T\colon \ S^n_2\to S^n_2$ we have
\begin{equation}\label{comeq9.1}
\|T\|_{CB(\theta,n)} = \sup_{X\in \cl{OH}(\theta,n)} \|T_X\|_{B(S^n_2[X])}= \sup_{X\in \cl{OH}(\theta,n)} \|T_X\|_{CB(S^n_2[X])}.
\end{equation}
Moreover, the supremum is unchanged if we restrict it to those $X=X(0)$ with $X(z) \simeq S^n_2$~~$\forall z\in J_\theta$ and $X(z)\simeq M_n$~~$\forall z\notin J_\theta$, where $\simeq$ means here completely isometric.
\end{thm}

\begin{proof}
The proof of Theorem \ref{comthm5.1} can be adapted to the operator space framework in the style of the preceding proof. This yields the first equality
in \eqref{comeq9.1}. But on one hand,
if $X$ is $OH$ of any dimension, the cb-norm 
of $T_X$ is equal to its norm on $  S^n_2[X] $ (cf. \cite{P7}), and on the other hand, if $X$ is any operator space, the cb-norm of  $T_X$ on $  S^n_2[X] $ is
$\le \|T\|_{ B_r(S^n_2) }$ (cf. \cite[(2.1)]{P3}), therefore by interpolation,
if $X$ is $\theta$-O-Hilbertian, we have  
$$\|T_X\|_{CB(S^n_2[X])}\le \|T\|_{ CB(\theta,n) }.$$
This yields the other equality in  \eqref{comeq9.1}. 
We leave the remaining details to the reader.
\end{proof}
\begin{prop}
In the situation of the preceding theorem, assume that $T$ is a Schur multiplier, i.e. $T=M_\varphi$, as in Corollary \ref{comcor5.8} above. 
Then \eqref{comeq9.1} and   \eqref{comeq5.15} are equal. 
\end{prop}
\begin{proof} Clearly, there is a simultaneously contractive projection
 on $ B(S^n_2)  $  and on $ B_r(S^n_2)  $ onto the subspace formed of all the Schur multipliers.  Obviously, on one hand
 $ \|  M_\varphi\|_{B(S^n_2)}=\| \varphi\|_{\ell_\infty^{n^2}},$
 and on the other hand we claim that
 $$  \|  M_\varphi\|_{B_r(S^n_2)}=\| \varphi\|_{{\cl M}[n]}.$$
 From the latter claim, the result becomes clear by general
 interpolation arguments, so it remains only to check this claim.
 We will actually show that, for any $1\le p\le \infty$, we have
 $$  \|  M_\varphi\|_{B_r(S^n_p)}=\| \varphi\|_{{\cl M}[n]}.$$
 To verify this, recall from \cite[Cor. 3.3]{P3}, that, if $p^{-1}=\theta$, 
 $$(CB(S^n_\infty),   CB(S^n_1))_\theta =
 B_r(S^n_p).$$
 But here again the Schur multipliers are simultaneously contractively complemented for the pair $(CB(S^n_\infty),   CB(S^n_1))$, therefore the Schur multipliers in
 $B_r(S^n_p)$ can be identified isometrically with the
 space $(Y_0,Y_1)_\theta$, where $Y_0$
 (resp. $Y_1$) denotes the subspace of $ CB(S^n_\infty)  $
 (resp. $CB(S^n_1)$)  formed by    the Schur multipliers, equipped
 with the induced norm. 
 The proof is then concluded by observing that
 for all $\varphi$ we have
 $$\| M_\varphi   \| _{Y_0}=\| M_\varphi   \| _{CB(S^n_\infty)}
 =\| \varphi\|_{{\cl M}[n]} \quad
 {\rm and}  \quad
  \| M_\varphi   \| _{Y_1}=\| M_\varphi   \| _{CB(S^n_1)}
 =\| \varphi\|_{{\cl M}[n]}.$$
\end{proof}

The next three statements are parallel to those of Section \ref{comsec4bis} and the proofs there are easy to adapt.

\begin{rem}\label{comrem9.5}
 Consider $\varphi$ in $B(S^n_2)$. We have
\[
 \|\varphi\|_{CB(\theta,n)^*} = \inf\{\|\lambda_1\|_{S^n_4} \|\lambda_2\|_{S^n_4} \|\mu_1\|_{S^n_4} \|\mu_2\|_{S^n_4} \|v\|_{(CB(S^n_1,M_n), \Gamma_{OH}(S^n_1,M_n))_\theta}\}
\]
where the infimum runs over all factorizations of $\varphi$ of the form $\varphi(x) = \mu_1v(\lambda_1x \lambda_2)\mu_2$, $\forall x\in S^n_2$.
\end{rem}

Let $({\cl M},\tau), ({\cl M}',\tau')$ be a pair of semifinite hyperfinite von Neumann algebras equipped with semifinite, normal faithful traces. For any operator space $E$, the operator space $L_p[\tau;E]$ was defined in \cite[Chapter 3]{P6}, and the regular operators from $L_2(\tau)$ to $L_2(\tau')$ were introduced in \cite{P3}. We denote by $B_r(L_2(\tau), L_2(\tau'))$ the normed space of regular maps $T\colon \ L_2(\tau)\to L_2(\tau')$ (this is denoted by $B_r(L_2(\tau), L_2(\tau'))$ in \cite{P3}). We have then

\begin{cor}\label{comcor9.6}
Let $B_r=B_r(L_2(\tau),L_2(\tau'))$ and $B = B(L_2(\tau), L_2(\tau'))$. Then the space $(B_r,B)^\theta$ consists of those $T$ in $B$ such that $T_X$ is bounded for any $\theta$-O-Hilbertian operator space $X$. Moreover,
\[
 \|T\|_{(B_r,B)^\theta} = \sup_{n,X\in OH(\theta,n)} \|T_X\| = \sup\|T_X\|_{cb}=\sup\|T_X\|
\]
where the last two sup run over all arcwise $\theta$-O-Hilbertian operator spaces $X$.  
\end{cor}

Let $X$ be an operator space. For any $\vp>0$, we denote by $\Delta^0_X(\vp)$ the smallest number $\delta$ such that for any $n$ and any $T\colon \ S^n_2\to S^n_2$ with $\|T\|_{\text{reg}}\le 1$ and $\|T\|\le \vp$ we have $\|T_X\|\le \delta$.

\begin{cor}\label{comcor9.7}
Assume that an operator space $X$ satisfies $\Delta^0_X(\vp)\in O(\vp^\alpha)$ for some $\alpha>0$. Then for any $0<\theta<\alpha$, $X$ is completely isomorphic to a subquotient of a $\theta$-O-Hilbertian space.
\end{cor}

\begin{thm}\label{comthm9.9}
 Let $C\ge 1$ be a constant. The following properties of an operator space $X$ are equivalent:
\begin{itemize}
\item[\rm (i)] There is an operator space $Y$ that is a quotient of a subspace of an arcwise $\theta$-O-Hilbertian space and a complete isomorphism $w\colon \ X\to Y$ such that $\|w\|_{cb}\|w^{-1}\|_{cb}\le C$.
\item[\rm (ii)] $\|T_X\|\le C$ for any $n$ and any $T$ such that $\|T\|_{(B_r(S^n_2), B(S^n_2))_\theta}\le 1$. 
\end{itemize}

\end{thm}

\begin{rk}
 Let $u\colon \ Z\to Y$ be a linear map between operator spaces. Then
\[
\gamma_{\theta OH}(u) = \sup_n\sup\{\|T\otimes u\colon \ S^n_2\lbrack Z\rbrack \to S^n_2\lbrack Y\rbrack\|\}
\]
where the second sup  runs over all $T\colon \ S^n_2\to S^n_2$ such that $\sup\{\|T_X\|\mid X\  \theta$-O-Hilbertian$\}\le 1$. This is easy to check by a simple adaptation
of \cite{He}.
\end{rk}

The $cb$-distance $d_{cb}(E,F)$ between two operator spaces is defined by:
\[
 d_{cb}(E,F) = \inf\{\|u\|_{cb}\|u^{-1}\|_{cb}\}
\]
where the infimum runs over all complete isomorphisms $u\colon \ E\to F$. If there exists no such $u$, we set $d_{cb}(E,F)=\infty$.

\begin{rk}
An operator space may be Hilbertian as a Banach space while being $\theta$-O-Hilbertian for no $0<\theta<1$. The simplest examples are the spaces $R = \overline{\text{span}}[e_{1n}] \subset B(\ell_2)$ and $C = \overline{\text{span}}[e_{n1}] \subset B(\ell_2)$. The fact that they are not $\theta$-O-Hilbertian follows from the observation any $n$-dimensional $\theta$-O-Euclidean space $E_n$ must satisfy
\begin{equation}\label{comeq9.19}
 d_{cb}(E_{n},OH_n) \le(\sqrt n)^\theta
\end{equation}
while it is known that
\begin{equation}\label{comeq9.20}
d_{cb}(R_n, OH_n) = d_{cb}(C_n, OH_n) = \sqrt n.
\end{equation}
We refer the reader  \cite[p. 219]{P5} for \eqref{comeq9.20}, and finally to \cite[Th. 6.9]{P6} and a simple interpolation argument for \eqref{comeq9.19}
\end{rk}

\begin{rk} There is no significant difficulty to extend the results of
\S \ref{comsec8} to the operator space setting but we choose to  skip the details.
\end{rk}

\section{Generalizations (Operator space case)}\label{sec13}

In this section, we briefly describe the extension of the results in \S \ref{comsec12} to the operator space case.
The extension of statement \ref{comsec12}.n will be numbered \ref{sec13}.n. All the background for this extension can be found in \cite[chapter 7]{P6} to which we refer the reader. The main difference from the Banach space case is the lack of local reflexity of (general) operator spaces:\ in the Banach case, any operator $u\colon \ E\to F$ such that $\gamma_{SQL_p}(u)<1$ can be factorized with constant $<1$ through a subspace of quotient of $\ell^m_p$ for some $m$ finite, but in the operator space analogue, for $u\colon \ E\to F$ with $\gamma_{SQOL_p}(u)<1$ we can only assert that there is a subquotient of an ultraproduct of the family $\{S^m_p\mid m\ge 1\}$, through which $u$ factors,   with $cb$-norms $<1$. However, if $E=S^n_1$ (or merely a quotient of $S^n_1$) and $F=M_n$ (or merely a subspace of $M_n$), then we can replace this ultraproduct by $S^m_p$ for some finite $m$. See \cite[pp.~81--82]{P6} (and also \cite{Ju}) for clarifications.

Let $\delta_{cb}(E,F) = {\log} d_{cb}(E,F)$. Let $OS_n$ denote the set of all $n$-dimensional  operator spaces, where, by convention, we identify two spaces if they are completely isometric. Then $(OS_n, \delta_{cb})$ is a complete (non-separable) metric space, see e.g.\ \cite[chapter 21]{P5}. For any subset ${\cl B} \subset \bigcup_n OS_n$  and any linear mapping $u\colon \ E\to F$ between finite dimensional operator spaces, we set
\[
 \gamma_{c{\cl B}}(u) = \inf\{\|u_1\|_{cb} \|u_2\|_{cb}\}
\]
where the infimum runs over all $X$ in ${\cl B}$ and all possible factorizations
$
 E \overset{\sst u_2}{\longrightarrow} X \overset{\sst u_1}{\longrightarrow} F $  of $u$ through  $X$.
We will say that ${\cl B}$ is an $SQO(p)$ class if it contains all the subquotients of ultraproducts of spaces of the form $\ell_p(\{S_p[X_i]\mid i\in I\})$ with $X_i\in {\cl B}$ for all $i$ in $I, I$ being an arbitrary finite set. We also always assume ${\bb C}\in {\cl B}$. Then it is easy to check that $u\to \gamma_{c{\cl B}}(u)$ is a norm on $CB(E,F)$ (say with $E,F$ finite dimensional). We denote by $\Gamma_{c{\cl B}}(E,F)$ the resulting normed space. We will denote by $SQOL_p$ the class of all subquotients of ultraproducts of $\{S^m_p\mid m\ge 1\}$. Clearly this is an example of $SQO(p)$-class.

For any $z$ in $\partial D$, we give ourselves $p(z)$ in $[1,\infty]$ and a subset ${\cl B}(z) \subset OS_n$. We will assume $z\to p(z)$ Borel measurable and $z\mapsto {\cl B}(z)$ ``measurable'' in the following sense:\ for any operator $u\colon \ E\to F$ between two finite dimensional operator spaces, the mapping $z\mapsto \gamma_{{\cl B}(z)}(u)$ is measurable on $\partial D$.
We assume that ${\cl B}(z)$ is an $SQO(p(z))$-class for any $z$ in $\partial D$.

We then define $p(\xi)$ and ${\cl B}(\xi)$, for $\xi$ in $D$, exactly as in \S \ref{comsec12}. Given $1\le p\le \infty$ and an $SQO(p)$-class ${\cl B} \subset OS_{n^2}$, we denote by  $c\beta(p,{\cl B})$ the space $CB(S^n_p)$ equipped with the norm
\begin{equation}
 \|T\|_{c\beta(p,{\cl B})} = \sup_{X\in {\cl B}} \|T_X\colon \ S^n_p[X] \to S^n_p[X]\|_{cb}.\tag*{$\forall T\in CB(S^n_p)$}
\end{equation}
We prefer not to worry about measurability questions here, so we will assume that $z\mapsto {\cl B}(z)$ is chosen so that $z\mapsto \|T\|_{c\beta(p(z), {\cl B}(z))}$ is measurable for any $T$ in $CB(S^n_p)$. Then if we set
\[
 c\beta(z) = c\beta(p(z), {\cl B}(z))
\]
the family $\{c\beta(z)\mid z\in \partial D\}$ is compatible, and hence extends, by complex interpolation, to a family $\{c\beta(\xi)\mid \xi\in\ovl D\}$. The next statement identifies $c\beta(\xi)$.

\begin{thm}\label{thm17.1}
 For any $\xi$ in $D$ we have a completely  isometric identity $c\beta(\xi) = c\beta(p(\xi), {\cl B}(\xi))$.
\end{thm}

\begin{cor}\label{cor17.2}
 Let $1\le p_0,p_1\le \infty$ and $0<\theta<1$. Define $p$ by $p^{-1} = (1-\theta) p^{-1}_0 + \theta p^{-1}_1$. Then the space $(CB(S^n_{p_0}), CB(S^n_{p_1}))_\theta$ coincides with the space $CB(S^n_p)$ equipped with the norm
\begin{equation}
\|T\| = \sup\{\|T_X\colon \ S^n_p[X] \to S^n_p[X]\|\tag*{$\forall T\in CB(S^n_p)$}
\end{equation}
where the supremum runs over all $X$ that can be written as $X=X(0)$ for some compatible family of $n^2$-dimensional operator spaces $\{X(z)\mid z\in \partial D\}$ such that $m(\{z\in \partial D\mid X(z)\in SQOL_{p_0}\}) = 1-\theta$ and $m(\{z\in\partial D\mid X(z)\in SQOL_{p_1}\}) = \theta$. Let $\Omega_o(n^2)$ denote the class of all such spaces $X$. Then the space $(CB(S_{p_0}), CB(S_{p_1}))^\theta$ can be identified with the subspace of $CB(S_p)$  formed of all $T$ such that
\begin{equation}\label{eq17.1}
T\mapsto \sup\{\|T_X\colon \ S_p[X] \to S_p[X]\|\mid X\in\Omega_o(n^2), n\ge 1\}
\end{equation}
is finite, equipped with the norm \eqref{eq17.1}, provided we make the convention that if either $p_0=\infty$ or $p_1=\infty$, $CB(S_{p_0})$ or $CB(S_{p_1})$ are replaced by $CB(S_\infty, B(\ell_2))$.
\end{cor}

\begin{cor}\label{cor17.3}
Let $1<p<\infty$ and $0<\theta<1$. The unit ball of $(B_r(S^n_p), CB(S^n_p))_\theta$ consists of all operators $T\colon \ S^n_p\to S^n_p$ such that $\|T_X\colon \ S^n_p[X]\to S^n_p[X]\|\le 1$ for all operator spaces $X$ that can be written as $X=X(0)$ where $\{X(z)\mid z\in\partial D\}$ is a compatible family of $n^2$-dimensional operator spaces such that $X(z)$ is a subquotient of an ultraproduct of $\{S^m_p\mid m\ge 1\}$ for all $z$ in a subarc of measure $\ge \theta$. Let $SQO_p(\theta,n^2)$ be the class of all such spaces. \\ Let $(M,\tau)$ and $(M',\tau')$ be hyperfinite von Neumann algebras equipped with semifinite, normal faithful traces. Then the unit ball of 
\[
 B_r(L_p(\tau), L_p(\tau')), CB(L_p(\tau), L_p(\tau')))^\theta
\]
consists of those $T$ in $CB(L_p(\tau), L_p(\tau'))$ such that 
\begin{equation}\label{eq17.2}
  \sup\{\|T_X\colon \ L_p(\tau;X)\to L_p(\tau'; X)\|\mid X\in SQO_p(\theta,n)\ n\ge 1 \}\le 1.
\end{equation}\end{cor}

\begin{cor}\label{comcor12.4bis}
Fix $0<\theta<1$. Consider a measurable partition $\partial D = J'_0\cup J'_0 \cup J_1$ with
\[
|J'_0| = (1-\theta)/2,\quad |J''_0| = (1-\theta)/2,\quad |J_1| = \theta.
\]
We set
\[
B(z) = \begin{cases}
CB(S^n_1)&\text{if $z\in J'_0$}\\
CB(M_n)&\text{if $z\in J''_0$}\\
B(S^n_2)&\text{if $z\in J_1$.}
       \end{cases}
\]
We have then isometrically
\[
B(0) \simeq (B_r(S^n_2), B(S^n_2))_\theta.
\]
\end{cor}

The proof of Lemma 10.4 extends with routine modifications and yields:

\begin{lem}\label{lem17.4}
 Given $n$-dimensional operator spaces $E,F$. Let $c\gamma(z) = \Gamma_{c{\cl B}(z)}(E,F)$. Then for all $w\colon  \ E\to F$ we have
\begin{equation}
\|w\|_{\Gamma_{c{\cl B}(\xi)}(E,F)} \le \|w\|_{c\gamma(\xi)}.\tag*{$\forall \xi\in D$}
\end{equation}
\end{lem}

\section{Examples with the Haagerup tensor product}\label{comsec10}

Using Kouba's interpolation theorem (\cite{Kou}) we will produce some interesting examples of compatible families of operator spaces involving quite naturally more than 2 spaces. Let $J{(1)}, J{(2)},\ldots, J{(d)}$ be arbitrary measurable subsets of $\partial D$.

We will denote for $j=1,2,\ldots, d$
\[
 J_0(j) = J(j),\qquad J_1(j)= \partial D\backslash J(j).
\]
Then for any $\vp = (\vp(j))$ in $\{0,1\}^d$ we set
\[
 J_\vp = J_{\vp(1)}(1)\cap\cdots\cap J_{\vp(d)} (d).
\]
Note that $\{J_\vp\mid \vp\in \{0,1\}^d\}$ is a partition of $\partial D$ into $2^d$ measurable subsets.

Now let $(A^1_0,A^1_1), \ldots, (A^d_0,A^d_1)$ be $d$ compatible pairs of finite dimensional operator spaces. We define a compatible family $\{X(z)\mid z\in\partial D\}$ by setting
\begin{equation}
 X(z) = A^1_{\vp(1)}\otimes_h\cdots\otimes_h A^d_{\vp(d)}.\tag*{$\forall z\in J_\vp$}
\end{equation}

\begin{thm}\label{comthm10.1}
 We have a completely isometric identification
\[
 X(0) \simeq (A^1_0,A^1_1)_{\theta_1} \otimes_h\cdots\otimes_h (A^d_0,A^d_1)_{\theta_d}
\]
where $\theta_j= m(J_1(j))$.
\end{thm}

\begin{proof}
Kouba's theorem (cf.\ \cite{Kou}) for the Haagerup tensor product implies the following (see \cite{P6}):\ Let $\{A(z)\}$ and $\{B(z)\}$ be two compatible families of finite dimensional operator spaces. Let $T(z) = A(z) \otimes_h B(z)$. Then $\{T(z)\}$ is a compatible family such that
\[
T(0)\simeq A(0) \otimes _h B(0)
\]
completely isometrically.

In other words, the operations $X\to X(0)$ (interpolation) and $X,Y\to X \otimes_h Y$ (Haagerup tensor product) are commuting. In the situation of Theorem \ref{comthm10.1}, we may iterate the preceding and we obtain 
\[
 X(0)\simeq X^1(0) \otimes_h\cdots \otimes_h X^d(0)
\]
where $X^j(z)$ is defined by
\[
 X^j(z) = \begin{cases}
           A^j_0&\text{if $z\in J_0(j)$}\\
A^j_1&\text{if $z\in J_1(j)$}
          \end{cases}~~.
\]
By  \cite[Cor. 5.1]{CCRS3}, we have then
\[
 X^j(0) = (A^j_0, A^j_1)_{\theta_j}
\]
where $\theta_j = m(J_1(j))$ (recall $m$ is the \emph{normalized} Lebesgue measure on $\partial D$). \end{proof}

In operator space theory, the pair $(R,C)$ of row and column Hilbert spaces plays a fundamental r\^ole. It was studied as an interpolation pair in \cite{P6} (see also \cite{P5}). We use transposition $R\to C$ (or $C\to R$) to identify an element of $R$ with one in $C$ for the purposes of interpolation. This allows us to view $(R,C)$ as a compatible pair. Let $C[\theta] = (R,C)_\theta$. Then $C[\theta]$ can be identified with the subspace of column vectors in the Schatten class $S_p$ for $p=(1-\theta)^{-1}$. So we set
\begin{equation}\label{comeq10.1}
 C_{[p]} = C[1-1/p] = (R,C)_{1-1/p} = (C,R)_{1/p}.
\end{equation}
Let $\theta_j\ge 0$ be such that $\theta_0 + \theta_1 +\cdots+ \theta_d = 1$. Let $\Delta_0 \cup\cdots\cup \Delta_d$ be a partition of $\partial D$ into arcs so that $m(\Delta_j) =\theta_j$. For any $z$ in $\Delta_j$ we set 
\begin{equation}\label{comeq10.2}
 K(z) = C\otimes_h \cdots\otimes_h C \otimes_h R \otimes\cdots\otimes_h R
\end{equation}
where the product has its $j$ first factors equal to $C$ and the following $d-j$ equal to $R$. Let $K_j$ be space appearing on the right-hand side of \eqref{comeq10.2}. It is well known (see \cite{ER} or \cite[p. 96]{P5}) that $K_j$ can be identified with the compact operators from $\ell^{\otimes(d-j)}_2$ to $\ell^{\otimes j}_2$, and in the extreme case $j=0$ (resp.\ $j=d$) we find $\ell^{\otimes d}_2$ with its row (resp.\ column) operator space structure.

The space $K_0 \cap K_1\cap\cdots\cap K_d$ has already appeared in Harmonic Analysis over ${\bb F}_\infty$, the free group with countably infinitely many generators $\{g_i\}$. Indeed, it was shown in \cite{Buch}
(in \cite {HP} for the   $d=1$ case) that $K_0\cap\cdots\cap K_d$ can be identified with the closed span in $C^*_\lambda({\bb F}_\infty)$ of $\{\lambda(g_{i_1}g_{i_2}\ldots g_{i_d})\mid i_1,i_2,\ldots, i_d\in {\bb N}\}$. For an extension of this to the non-commutative $L_p$-space over ${\bb F}_\infty$, see \cite{PaP}. These results motivated us to study the interpolation spaces associated to the family $K(z)$ defined by \eqref{comeq10.2}. Curiously, one can identify them quite easily:

\begin{cor}\label{comcor10.2}
We have a completely isometric identity
\[
 K(0)\simeq C_{[p_0]} \otimes_h C_{[p_1]} \cdots \otimes_h C_{[p_{d-1}]}
\]
where 
\[
 p_j = (\theta_0+\theta_1 +\cdots+ \theta_j)^{-1}.
\]
Note that $p_0\ge p_1\ge\cdots\ge p_{d-1}$.
\end{cor}

\begin{proof} By routine arguments (particularly easy here because
of the ``homogeneity"
of the spaces $R$ and $C$), one can reduce to the case when $R,C$ are replaced by their
  $n$-dimensional version $R_n,C_n$. To lighten the notation, we ignore this and still denote
  them by $R,C$.  Then the corollary follows from Theorem \ref{comthm10.1}, once one observes that if we use
\[
 (A^j_0,A^j_1) = (C,R)
\]
and set 
\[
 J_1(j) = \Delta_0\cup\cdots\cup \Delta_{j-1} \qquad (1\le j\le d),
\]
  we   find that if $z\in \Delta_j$
\[
 X^1(z) \otimes_h\cdots \otimes_h X^d(z) = K_j.
\]
Therefore, to conclude it suffices to calculate $m(J_1(j)) = \theta_0 +\cdots+\theta_{j-1} = (p_{j-1})^{-1}$ and to recall \eqref{comeq10.1}.
\end{proof}

\end{document}